%% file: paper.tex
\documentclass{gtpart}
\usepackage{pinlabel}
\usepackage{amsmath,amssymb,diagrams,rotating}
\title{Dehn Twists in Heegaard Floer Homology} 

\author{Bijan Sahamie} 
\givenname{Bijan}
\surname{Sahamie}
\address{Mathematisches Institut der Universit\"{a}t zu K\"{o}ln, Weyertal 86-90, 50931 K\"{o}ln Germany
\footnote{Current address: Mathematisches Institut der LMU M\"{u}nchen, Theresienstrasse 39, 80333 M\"{u}nchen Germany}}
\email{sahamie@math.lmu.de}
\urladdr{http://www.math.lmu.de/~sahamie}


\theoremstyle{plain} 
\newtheorem{theorem}{Theorem}[section]   
\newtheorem{lem}[theorem]{Lemma}         
\newtheorem{prop}[theorem]{Proposition}
\newtheorem{cor}[theorem]{Corollary}

\theoremstyle{definition}
\newtheorem{definition}[theorem]{Definition}   
\newtheorem{rem}{Remark}

\makeop{Homo}
\numberwithin{equation}{section}

\newarrow{Mto}{|}{-}{-}{-}{>}

\begin{document}
\include{abbreviations}
\begin{abstract}    
We derive a new exact sequence in the hat-version of 
Heegaard Floer homology. As a consequence we see a
functorial connection between the invariant of Legendrian
knots $\loss$ and the contact element. As an application 
we derive two vanishing results of the contact element making
it possible to easily read off its vanishing out of
a surgery presentation in suitable situations. 
\end{abstract}
\maketitle

\section{Introduction}\label{parone}
Heegaard Floer homology is a Floer-type homology theory developed by 
P. \ozs$\,$ and Z. \sza. There are two invariants in Heegaard Floer 
homology interesting for contact geometry. First to mention is the contact 
element, introduced in \cite{OsZa05} by \ozs$\,$ and \sza. This contact
element is an isotopy invariant of contact structures and an obstruction to 
overtwistedness. It is useful in the sense that there are examples 
of contact structures (see \cite{Stip01},\cite{Stip02},\cite{Stip03}) where 
conventional techniques fail to detect, but the contact element is 
able to detect, tightness versus overtwistedness. The second invariant to mention
is the isotopy invariant $\loss$ of Legendrian knots found 
by Lisca, \ozs, Stipsicz and \sza. There is also an isotopy invariant of contact
manifolds with boundary (see \cite{HKM2}) in Sutured Floer homology (see \cite{AJU}).
As a variant of this contact geometric invariant, Honda, Kazez and Mati\'{c} found and 
isotopy invariant of Legendrian knots, called $EH$, in the Sutured Floer theory.
This invariant is related to $\loss$ as shown by Stipsicz and Vertesi in \cite{StipVert}.\vspace{0.3cm}\\
In this paper we start with the observation that the hat-version 
of knot Floer homology can be defined, and is well-defined, even 
for knots that are not null-homologous 
(see \S\ref{knotfloerhomology}). Furthermore, we see that in case of the
hat-version of the knot Floer homology we can relax the admissibility
condition (see \S\ref{knotfloerhomology}). After that, in \S\ref{introd}, 
we will derive 
in what way the hat-version of Heegaard Floer homology behaves on 
Dehn twist changes of the gluing map 
(see Propositions~\ref{THMTHM} and \ref{THMTHM2}). The representation
found naturally imposes the existence of an exact sequence 
(see Corollaries~\ref{motiv} and \ref{motiv2}). In \S\ref{naturality}
we set up invariance properties (see Propositions~\ref{ch3invar01} 
to \ref{ch3invar04}) 
suitable for showing that the maps involved in the sequence are 
topological, i.e.~only depend on the cobordism that can 
be associated to the Dehn twist (see Theorem~\ref{completeinvar}). 
One of the maps involved in the sequence preserves 
contact geometric information when representing a $(+1)$-contact 
surgery (see Theorem~\ref{maps}). This leads to a 
functorial connection between the 
invariant $\loss$ and the contact element when performing 
$(+1)$-contact surgeries (see \S\ref{contsetup}). 
Finally, in \S\ref{applics} we give some applications.
The first to mention is Proposition~\ref{result01} in which we
give a new proof for a result that can already be derived using
known results (see \S\ref{applics}). A second application is Proposition~\ref{calculation}
where we give a calculation of a contact element after performing
a $(+1)$-contact surgery by using the invariant of Legendrian knots with
Theorem~\ref{maps}. Finally, with Theorem~\ref{result02}, we derive a new vanishing result 
of the contact element which can be easily read off from a surgery representation
of the underlying contact manifold. Everything here is done with $\ztwo$-coefficients.
A suitable introduction of coherent orientations will be given
in a future article.
\paragraph{Acknowledgements.} 
The present article contains the results, given in chapter $3$ of the 
author's Ph.D.~thesis. The author wants to thank his 
advisor Hansj\"org Geiges for his constructive comments 
which helped make the exposition clearer at numerous spots. The 
author wishes to warmly thank Andr\'{a}s Stipsicz for his help and his constructive
criticism on the first version of this article.
\section{Introduction to Heegaard Floer theory}
\subsection{Handle decompositions and Heegaard diagrams}\label{handles}
We briefly review the connection between Heegaard diagrams and handle
decompositions to fix our point of view on the subject.\vspace{0.3cm}\\
Let $Y$ be a closed, oriented $3$-manifold. $Y$ admits a handle 
decomposition with one $0$-handle $h^0$ and one $3$-handle $h^3$,
a number $l$ of $1$-handles and a number $g$ of $2$-handles. The union
\[
  H_0:=h^0\cupb h^1_{0,1}\cupb\dots\cupb h^1_{0,l}
\]
is a handlebody of genus $l$. By dualizing the $2$-handles and $3$-handles
in the handle decomposition of $Y$ we see that the union of 
these is a handlebody $H_1$ of genus $g$. Since 
$Y=H_0\cupb H_1$ is closed obviously the genera of $H_0$ and $H_1$ coincide.
The manifold $Y$ is determined by the following data: The images of the 
attaching circles of the $2$-handles on $\Sigma:=\partial H_0$.
We can equivalently interpret this handle decomposition as a decomposition
relative to the splitting surface $\Sigma$. By dualizing the handlebody
$H_0$ we can write the manifold $Y$ as
\begin{equation}
  (h^3_0\cupb h^2_{0,1}\cupb\dots\cupb h^2_{0,g})
  \cupb 
  (\Sigma\times[0,1])
  \cupb
  (h^2_{1,1}\cupb \dots\cupb h^2_{1,g}\cupb h^3_1).
  \label{handledecomp}
\end{equation}
The information necessary to describe the $3$-manifold $Y$ 
in terms of a handle decomposition like $(\ref{handledecomp})$
is a triple $(\Sigma,\alpha,\beta)$, where $\Sigma$ is the 
splitting surface used in the decomposition $(\ref{handledecomp})$,
$\alpha=\{\alpha_1,\dots,\alpha_g\}$ are the images of the 
attaching circles of the $h^2_{0,i}$ in $\Sigma\times\{0\}$ 
and $\beta=\{\beta_1,\dots,\beta_g\}$ the images of the attaching circles of 
the $2$-handles $h^2_{1,i}$ in $\Sigma\times\{1\}$. Observe that the
$\alpha$-curves are the co-cores of the $1$-handles in the dual picture, 
and that sliding the $1$-handle $h^1_{0,i}$ over $h^1_{0,j}$ 
means, in the dual picture, that $h^2_{0,j}$ is slid over $h^2_{0,i}$.

\subsection{Heegaard Floer homologies}\label{prelim:01:1}
The Heegaard Floer homology groups $\hfminus(Y)$ and $\hfhat(Y)$ of a
$3$-manifold $Y$ were introduced in \cite{OsZa01}. The definition was 
extended for the case where $Y$ is equipped with a null-homologous 
knot $K\subset Y$ to
variants $\hfkminus(Y,K)$, $\hfkhat(Y,K)$ in \cite{OsZa04}.\vspace{0.3cm}\\
A $3$-manifold $Y$ can be described by a Heegaard diagram, which is a
triple $(\Sigma,\alpha,\beta)$, where $\Sigma$ is an oriented genus-g surface
and $\alpha=\{\alpha_1,\dots,\alpha_g\}$, 
$\beta=\{\beta_1,\dots,\beta_g\}$ are two sets of pairwise disjoint simple closed 
curves in $\Sigma$ called {\bf attaching circles} (cf.~\S\ref{prelim:01:1}). 
Each set of curves $\alpha$ and $\beta$ is required to consist of 
linearly independent curves in $H_1(\Sigma,\Z)$. In the 
following we will talk about the curves in the set $\alpha$ (resp.~$\beta$) as  
{\bf $\alpha$-curves} (resp.~{\bf $\beta$-curves}). Without loss 
of generality we may assume that the 
$\alpha$-curves and $\beta$-curves intersect 
transversely. To a Heegaard diagram we may associate the triple
$(\symg,\talpha,\tbeta)$ consisting of the $g$-fold symmetric power of
$\Sigma$, 
\[
  \symg=\Sigma^{\times g}/S_g,
\] 
and the submanifolds $\talpha=\alpha_1\times\dots\times\alpha_g$
and $\tbeta=\beta_1\times\dots\times\beta_g$. We define 
$\cfminus(\Sigma,\alpha,\beta)$ as the free $\ztwo[U]$-module 
generated by the set
$\talpha\cap\tbeta$. In the following
we will just write $\cfminus$. For two intersection
points $x,y\in\talpha\cap\tbeta$ define $\pitwo(x,y)$ to be the set of
homology classes of {\bf Whitney discs} 
$\phi\co\disc\lra\symg$ ($\disc\subset\C$) that 
{\bf connect $x$ with $y$}. The map $\phi$ is called {\bf Whitney} if 
$\phi(\disc\cap\{Re<0\})\subset\talpha$ and $\phi(\disc\cap\{Re>0\})\subset\tbeta$. 
We call $\disc\cap\{Re<0\}$ the {\bf $\alpha$-boundary of $\phi$} and
$\disc\cap\{Re>0\}$ the {\bf $\beta$-boundary of $\phi$}. Such 
a Whitney disc {\bf connects $x$ with $y$} if $\phi(i)=x$ and $\phi(-i)=y$. 
Note that $\pitwo(x,y)$ can be interpreted as the subgroup of elements in
$H_2(\symg,\talpha\cup\tbeta)$ represented by discs with appropriate 
boundary conditions. We endow 
$\symg$ with a symplectic structure~$\omega$. By choosing an almost complex 
structure $J$ on $\symg$ suitably (cf.~\cite{OsZa01})
all moduli spaces of holomorphic Whitney discs are Gromov-compact manifolds.
Denote by $\modphi$ the set of holomorphic Whitney discs in the equivalence
class $\phi$, and $\mu(\phi)$ the formal dimension of $\modphi$. Denote by 
$\modhatphi=\modphi/\R$ the quotient under the translation action of 
$\R$ (cf.~\cite{OsZa01}). Define $H(x,y,k)$ to be the subset of classes in
$\pitwo(x,y)$ that admit moduli spaces of dimension $k$. Fix a point 
$z\in\Sigma\backslash(\alpha\cup\beta)$ and define the map 
\[
  n_z\co\pitwo(x,y)\lra\Z,\,\phi\lmt\#(\phi,\{z\}\times\symgmo).
\] 
A boundary operator $\partial^-\co\cfminus\lra\cfminus$ is given by defining it
on the generators $x$ of $\cfminus$ by
\[
  \partial^- x=\sum_{y\in\talpha\cap\tbeta}\sum_{\phi\in H(x,y,1)}\#\modhatphi\cdot U^{n_z(\phi)}y.
\]
Define $\cfhat$ to be the free $\ztwo$-module generated by 
$\talpha\cap\tbeta$. By sending $U$ to zero we can define a projection
$\pi\co\cfminus\lra\cfhat$. With this projection the differential $\parminus$
induces a morphism $\parhat$ on $\cfhat$. The almost-complex structure
$J$ on $\symg$ is chosen in such a way that $\{z\}\times\symgmo$ is a complex
submanifold of $\symg$. This means a holomorphic Whitney disc intersects 
$\{z\}\times\symgmo$ always positively. Thus $\parhat$ is a differential on
$\cfhat$. We define
\[
  \hfminus(Y):=H_*(\cfminus,\partial^-)\quad\mbox{\rm and }\quad
  \hfhat(Y):=H_*(\cfhat,\parhat).
\]
These homology groups are topological invariants of the manifold $Y$. 
We would like to note that not all Heegaard diagrams are suitable
for defining Heegaard Floer homology; there is an additional
condition that has to be imposed called {\bf admissibility}. 
This is a technical condition in the compactification of the moduli spaces
of holomorphic Whitney discs. A detailed knowledge of this condition 
is not important in the remainder of the present article since all
constructions are done nicely so that there will never be a 
problem. We advise the interested reader to \cite{OsZa01} .

\subsubsection{Knot Floer Homology}\label{knotfloerhomology}
Knot Floer homology is a variant of the Heegaard Floer homology of a
manifold. We briefly introduce the theory here and finally argue why
the construction carries over verbatim to give an invariant even for
knots that are not necessarily null-homologous. For a more detailed treatment
we point the reader to \cite{Saha}.\vspace{0.3cm}\\
Given a knot $K\subset Y$, we can specify a certain subclass of 
Heegaard diagrams.
\begin{definition} \label{knotdiagram} A Heegaard 
diagram $(\Sigma,\alpha,\beta)$ is said to
be {\bf subordinate} to the knot $K$ if $K$ is isotopic to a knot lying
in $\Sigma$ and $K$ intersects $\beta_1$ once transversely and is
disjoint from the other $\beta$-circles.
\end{definition}
Since $K$ intersects $\beta_1$ once and is disjoint from the other 
$\beta$-curves we know that $K$
intersects the core disc of the $2$-handle represented by $\beta_1$ once
and is disjoint from the others (after possibly isotoping the knot $K$).
\begin{lem}\label{help} Every pair $(Y,K)$ admits a Heegaard diagram
subordinate to $K$.
\end{lem}
\begin{proof}
By surgery theory (see \cite{GoSt}, p. 104) 
we know that there is a handle decomposition of $Y\backslash\nu K$, i.e.
\[
  Y\backslash\nu K
  \cong
  (T^2\times [0,1])
  \cupb 
  h^1_{2}\cupb\dots h^1_{g}
  \cupb 
  h^2_1\cupb\dots\cupb h^2_g
  \cupb h^3
\]
We close up the boundary $T^2\times\{0\}$ with an 
additional $2$-handle $h^{2*}_1$ and a $3$-handle $h^3$ to obtain
\begin{equation}
  Y\cong
  h^3\cupb h^{2*}_1
  \cupb
  (T^2\times I)
  \cupb 
  h^1_2\cupb\dots h^1_g
  \cupb 
  h^2_1\cupb\dots\cupb h^2_g
  \cupb h^3.\label{handledecomp02}
\end{equation}
We may interpret $h^3\cupb h^{2*}_1\cupb(T^2\times[0,1])$ as 
a $0$-handle $h^0$
and a $1$-handle $h^{1*}_1$. Hence, we obtain the following decomposition of
$Y$:
\[
  h^0
  \cupb
  h^{1*}_1
  \cupb
  h^{1}_2
  \cupb
  \dots
  \cupb
  h^1_g
  \cupb
  h^2_1
  \cupb
  \dots
  \cupb
  h^2_g
  \cupb
  h^3.
\]
We get a Heegaard diagram $(\Sigma,\alpha,\beta)$ where
$\alpha=\{\alpha_1\}^*\cup\{\alpha_2,\dots,\alpha_g\}$ are the co-cores 
of the $1$-handles and $\beta=\{\beta_1,\dots,\beta_g\}$ are 
the attaching circles of the $2$-handles.
\end{proof}
Having fixed such a Heegaard diagram $(\Sigma,\alpha,\beta)$ we can encode 
the knot $K$ in a pair of points. After isotoping $K$ onto $\Sigma$, 
we fix a small interval $I$ in $K$ containing the intersection point 
$K\cap\beta_1$. This interval should be chosen small enough such 
that $I$ does not contain any other intersections of $K$ with other 
attaching curves. The boundary $\partial I$ of $I$ determines two 
points in $\Sigma$ that lie in the complement of the attaching circles, 
i.e.~$\partial I=z-w$, where the orientation of $I$ is given by the 
knot orientation. This leads to a doubly-pointed Heegaard diagram 
$(\Sigma,\alpha,\beta,w,z)$. Conversely, a doubly-pointed Heegaard 
diagram uniquely determines a topological knot class: Connect 
$z$ with $w$ in the complement of the attaching circles $\alpha$ 
and $\beta\backslash\beta_1$ with an arc $\delta$ that crosses 
$\beta_1$ once. Connect $w$ with $z$ in the complement of $\beta$
using an arc $\gamma$. The union $\delta\cup\gamma$ is represents the 
knot klass $K$ represents. The orientation on $K$ is given by orienting $\delta$ such 
that $\partial\delta=z-w$. If we use a different path 
$\widetilde{\gamma}$ in the complement of $\beta$, we observe that 
$\widetilde{\gamma}$ is isotopic to $\gamma$ (in $Y$): Since  
$\Sigma\backslash\beta$ is a sphere with holes an isotopy can 
move $\gamma$ across the holes by doing handle slides. Isotope 
the knot along the core discs of the $2$-handles to cross the 
holes of the sphere. Indeed, the knot class does not depend
on the specific choice of $\delta$-curve.\vspace{0.3cm}\\
The knot chain complex $\cfkhat(Y,K)$ is the free $\ztwo$-module 
generated by the intersections $\talpha\cap\tbeta$. 
The boundary operator $\parhat^w$, for $x\in\talpha\cap\tbeta$, is 
defined by
\[
  \parhat^w(x)
  =
  \sum_{y\in\talpha\cap\tbeta}
  \sum_{\phi\in H(x,y,1)}
  \#\modhatphi\cdot y,
\]
where $H(x,y,1)\subset\pitwo(x,y)$ are the homotopy classes
with $\mu=1$ and $n_z=n_w=0$. We denote by $\hfkhat(Y,K)$
the associated homology theory $H_*(\cfkhat(Y,K),\parhat^w)$.
The crucial observation for showing invariance is, that two 
Heegaard diagrams subordinate to a given knot can be connected 
by moves that {\it respect the knot complement}.
\begin{lem}(\cite{OsZa04})\label{helplem} Let 
$(\Sigma,\alpha,\beta,z,w)$ and $(\Sigma',\alpha',\beta',z',w')$ 
be two Heegaard diagrams subordinate to a given knot $K\subset Y$. 
Let $I$ denote the interval inside $K$ connecting $z$ with $w$,
interpreted as sitting in $\Sigma$. Then these two diagrams 
are isomorphic after a sequence of the following moves:
\begin{enumerate}
  \item[($m_1$)] Handle slides and isotopies among the 
  $\alpha$-curves. These isotopies may not cross~$I$.
  \item[($m_2$)] Handle slides and isotopies among 
  the $\beta_2,\dots,\beta_g$. These isotopies may 
  not cross $I$.
  \item[($m_3$)] Handle slides of $\beta_1$ over 
  the $\beta_2,\dots,\beta_g$ and isotopies.
  \item[($m_4$)] Stabilizations/destabilizations.
\end{enumerate}
\end{lem}
For the convenience of the reader we include a short proof of this lemma.
\begin{proof} 
By Theorem 4.2.12 of \cite{GoSt} we can transform two 
relative handle decompositions into each other by 
isotopies, handle slides and handle creation/annihilation of 
the handles written at the right of $T^2\times[0,1]$ in 
$(\ref{handledecomp02})$. Observe 
that the $1$-handles may be isotoped along the boundary 
$T^2\times\{1\}$. Thus, we can transform two Heegaard diagrams 
into each other by handle slides, isotopies, creation/annihilation 
of the $2$-handles $h^2_i$ and we may slide the $h^1_i$ over 
$h^1_j$ and over $h^{1*}_1$ (the latter corresponds to $h^1_i$ 
sliding over the boundary $T^2\times\{1\}\subset T^2\times I$ 
by an isotopy). But we are not allowed to move $h^{1*}_1$ off 
the $0$-handle. In this case we would lose the relative 
handle decomposition. In terms of Heegaard diagrams 
we see that these moves exactly translate into the moves given 
in ($m_1$) to ($m_4$). Just note that sliding the $h^1_i$ over $h^{1*}_1$,
in the dual picture, looks like sliding $h^{2*}_1$ over the $h^2_i$. 
This corresponds to move ($m_3$).
\end{proof}
\begin{prop}[\cite{Saha}, Proposition 2.4.4]\label{knotfloer} Let $K\subset Y$ be an arbitrary knot. 
The knot Floer homology group $\hfkhat(Y,K)$ is a topological invariant
of the knot type of $K$ in $Y$. These homology groups split with 
respect to $\spinc(Y)$.
\end{prop}
There are no homological requirements on the knot $K$ needed for proving 
Lemma~\ref{help} and Lemma~\ref{helplem}. Thus we may define the knot Floer homology for
an arbitrarily chosen pair $(Y,K)$. To conclude that the defined groups
are indeed invariants of the pair $(Y,K)$ we have to observe that every
move, ($m_1$) to ($m_4$), induces an isomorphism between the respective
knot Floer homologies. The invariance proof of knot Floer homology Ozsv\'{a}th
and Szab\'{o} give in \cite{OsZa04} uses the maps from the invariance
proof of Heegaard Floer homology with just one slight modification. 
In knot Floer homology they require the holomorphic discs counted to 
have trivial intersection number $n_w$. The positivity of intersections
in the holomorphic case and the additivity of the intersection number
imply that the associated maps between the knot Floer homologies 
are isomorphisms. We do not need any homological information of the 
knot $K$. For details we point the reader to \cite{Saha}.
\subsubsection{Admissibility of $\hfkhat$}\label{admsec}
Finally, we would like to address the admissibility conditions imposed
on the Heegaard diagrams, used in the definition of the knot Floer homologies.
We may relax the admissibility condition, given by Ozsva\'{a}th and Szab\'{o} (see \cite{OsZa04})
and still get well-defined knot invariants. A {\bf periodic domain} $\dom$ is a linear
combination of the components $\Sigma\backslash\{\alpha\cup\beta\}$ such that
the boundary of $\dom$ consists of a linear combination of $\alpha$-curves and
$\beta$-curves and such that $n_z(\dom)=0$, where $n_z(\dom)$ is the multiplicity
of $\dom$ at the region containing the base point $z$. Furthermore, we denote by
$\mathcal{H}(\dom)$ its associated homology class which is given by closing the
boundary components of $\dom$ with the cores of the $2$-handles associated to the
$\alpha$-curves and $\beta$-curves.
\begin{definition}\label{extweakadm} We call a 
doubly-pointed Heegaard diagram
$(\Sigma,\alpha,\beta,w,z)$ {\bf extremely weakly admissible}
for the $\spinc$-structure $s$ if for every non-trivial periodic 
domain, with $n_w=0$ and $\left<c_1(s),\mathcal{H}(\dom)\right>=0$, 
the domain has both positive and negative coefficients.
\end{definition}
It is not hard to see that the following result holds.
\begin{theorem}[\cite{Saha}, Theorem 2.4.6] Let $(\Sigma,\alpha,\beta,w,z)$ be an extremely
weakly admissible Heegaard diagram, then $\parhat^w$ is
well-defined and a differential.\hfill$\square$
\end{theorem}

\subsection{Contact Structures}
A $3$-dimensional contact manifold is a pair $(Y,\xi)$ where $Y$ is a 
$3$-dimensional manifold and $\xi\subset TY$ a hyperplane bundle, 
 that can be written as the kernel of a $1$-form 
$\alpha$ with the property
\begin{equation}
  \alpha\wedge d\alpha\not=0. \label{contcond}
\end{equation}
$1$-forms with the property $(\ref{contcond})$ are called {\bf contact
forms}. The contact form $\alpha$ is not uniquely determined. The
existence of a contact form implies that $TY/\xi$ is a $1$-dimensional
trivial bundle. Thus there are non-vanishing vector fields on $Y$ 
transverse to $\xi$. The vector field $R_\alpha$ defined by the conditions
\[
  \alpha(R_\alpha)\not=0
  \;\mbox{\rm and }\;\;
  \iota_{R_\alpha}d\alpha=0
\]
is called {\bf Reeb field} of the contact form $\alpha$. Two contact 
manifolds $(Y,\xi)$ and $(Y',\xi')$ are called 
{\bf contactomorphic} if there is a diffeomorphism 
$\phi\co Y\lra Y'$ such that $T\phi(\xi)=\xi'$. A diffeomorphism
preserving contact structures in this manner 
is called {\bf contactomorphism}. Every contact manifold is 
locally contactomorphic to the standard contact
space $(\R^3,\xistd)$, where $\xistd$ is the contact structure 
given by the kernel of the $1$-form $dz-y\,dx$ 
({\bf Darboux's theorem}). This property tells us that locally 
contact manifolds cannot be distinguished, and ,thus, invariants
of contact manifolds always have to be of global nature. An important
property of contact structures is known as {\bf Gray stability}. Gray stability means 
that each smooth homotopy of contact structures $(\xi_t)_{t_\in[0,1]}$ 
is induced by an ambient isotopy $\phi_t$ of the underlying 
manifold, i.e.~such that $T\phi_t(\xi_0)=\xi_t$. This property 
naturally gives a connection between contact structures and
the topology of the manifold. Submanifolds tangent to the contact 
structure are also interesting objects to study. The contact 
condition implies that on a $3$-dimensional contact manifold 
$(Y,\xi)$ only $1$-dimensional submanifolds, i.e.~knots and 
links, can be tangent to $\xi$. These submanifolds are 
called {\bf Legendrian knots and links}. Their investigation
is naturally motivated by the contact-analogue of surgery, called 
{\bf contact surgery}. Contact surgery in arbitrary dimensions 
was introduced by Eliashberg in \cite{eliash2}. His construction, in dimension $3$, corresponds to
$(-1)$-contact surgeries. For $3$-dimensional contact manifolds Ding and Geiges gave
in \cite{DiGei04} a definition of contact-$r$-surgeries (cf.~also \cite{DiGei}) for 
arbitrary $r\in\Q>0$. It is nowadays one of the most significant tools for $3$-dimensional 
contact geometry.

\subsection{Open Books, the Contact Element and the Invariant LOSS}\label{prelim:01:2}
\subsubsection{Open Books and the Contact Element}
We start by recalling some facts about open book decompositions of
$3$-manifolds. For details we point the reader to 
$\cite{Etnyre01}$.\vspace{0.3cm}\\
An {\bf open book} is a pair $(P,\phi)$ consisting of an oriented 
genus-$g$ surface $P$ with boundary and a homeomorphism $\phi\co P\lra P$ 
that is the identity near the boundary of $P$. The surface $P$ is called
{\bf page} and $\phi$ the {\bf monodromy}. Recall that
an open book $(P,\phi)$ gives rise to a $3$-manifold by the following 
construction: Let $c_1,\dots,c_k$ denote the boundary components of
$P$. Observe that
\begin{equation}
  (P\times[0,1])/(p,1)\sim(\phi(p),0) \label{ob:01}
\end{equation}
is a $3$-manifold with boundary given by the tori
\[
  \left((c_i\times[0,1])/(p,1)\sim(p,0)\right)\cong c_i\times\sone.
\]
Fill in each of the holes with a full torus $\disc^2\times\sone$: we glue
a meridional disc $\disc^2\times\{\star\}$ onto $\{\star\}\times\sone\subset c_i\times\sone$.
In this way we define a closed, oriented $3$-manifold $Y(P,\phi)$. Denote by
$B$ the union of the cores of the tori $\disc^2\times\sone$. The set $B$
is called {\bf binding}. Observe that the definition of $Y(P,\phi)$ defines
a fibration
\[
  P\hookrightarrow Y(P,\phi)\backslash B\lra\sone.
\]
Consequently, an open book gives rise to a Heegaard decomposition 
of $Y(P,\phi)$ and, thus, induces a Heegaard diagram of $Y(P,\phi)$. To see 
this we have to identify a splitting surface of $Y(P,\phi)$, i.e.~a surface 
$\Sigma$ that splits the manifold into two components. Observe that the boundary
of each fiber lies on the binding $B$. Thus
gluing together two fibers yields a closed surface $\Sigma$ of genus $2g$.
The surface $\Sigma$ obviously splits $Y(P,\phi)$ into two components and
 can therefore be used to define a Heegaard 
decomposition of $Y(P,\phi)$ (cf.~\cite{Geiges}).
\begin{figure}[ht!]
\labellist\small\hair 2pt
\pinlabel {Page $P\!\times\!\{1/2\}$ of the open book} [bl] at 29 187
\pinlabel {$z$} [bl] at 189 112
\pinlabel {$a_i$} [t] at 76 22
\pinlabel {$b_i$} [t] at  153 22
\endlabellist
\centering
\includegraphics[height=3cm]{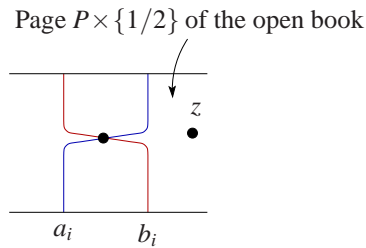}
\caption{Definition of $b_i$ and positioning of the point $z$.}
\label{Fig:zpointpos}
\end{figure}

Let $a=\{a_1,\dots,a_n\}$ be a {\bf cut system} of $P$, i.e.~a set of
disjoint properly embedded arcs such that $P\backslash\{a_1,\dots,a_n\}$
is a disc. One can easily show that being a cut system implies that
$n=2g$. Choose the splitting surface
\[
  \Sigma:=P\times\{1/2\}\cup_{\partial} (-P)\times\{1\}
\]
and let $\oa_i$ be the curve $a_i\subset P\times\{1/2\}$ with opposite 
orientation, interpreted as a curve in $(-P)\times\{0\}$.
Then define $\alpha_i:=a_i\cup\oa_i$. The curves $b_i$ are isotopic push-offs of the $a_i$. We
choose them like indicated in Figure~\ref{Fig:zpointpos}: We push the $b_i$
off the $a_i$ by following with $\partial b_i$ the positive boundary 
orientation of $\partial P$. Finally set $\beta_i:=b_i\cup\overline{\phi(b_i)}$.
The data $(\Sigma,\alpha,\beta)$ define a Heegaard diagram of $Y(P,\phi)$ 
(cf.~\cite{HKM}). There is one intersection point of $\talpha\cap\tbeta$ 
sitting on $P\times\{1/2\}$. Denote this point by $EH(P,\phi,a)$. \vspace{0.3cm}\\
There is a natural way to define a cohomology theory from a given homology
 (see \cite{Bredon}): Use the ${\rm Hom}$-functor to define a cochain-module
 and use the naturally induced boundary to give the module the structure of
  a chain complex. We can define the {\bf Heegaard Floer cohomology} of a manifold $Y$ the same way.
One can easily show that the Heegaard Floer cohomology of a manifold $Y$ is 
isomorphic to the Heegaard Floer homology of $-Y$ (see \cite{OsZa01}). Observe that if
$(\Sigma,\alpha,\beta)$ is a Heegaard diagram for $Y$ then $(-\Sigma,\alpha,\beta)$
is a Heegaard diagram for $-Y$. The change of the surface orientation affects
the boundary operator through a modification of the boundary conditions of
the Whitney discs: we count holomorphic discs $\phi$ with 
$\phi(i)=x$, $\phi(-i)=y$ and $\phi(\disc^2\cap{Re<0})\subset\tbeta$ and 
$\phi(\disc^2\cap{Re<0})\subset\talpha$ (note that we switched the boundary 
conditions).
Hence the Heegaard Floer cohomology of $Y$ is given by the data
$(\Sigma,\beta,\alpha)$ (we changed the position of $\alpha$ and $\beta$).
The point $EH(P,\phi,a)$ can be interpreted as a generator of $\cfhat(-Y)$.
In this case $EH(P,\phi,a)$ is indeed a cycle and thus defines a cohomology
class $c(P,\phi)\in\hfhat(-Y)$. The class $[EH(P,\phi,a)]$ does not
depend on the choice of cut system $a$.\vspace{0.3cm}\\
Recall the connection between open books and contact structures on $3$-manifolds
(cf.~\cite{Etnyre01}). Every contact structure gives rise to an adapted
open book decomposition. The open book is uniquely
determined up to positive Giroux stabilizations. Given a contact structure 
$\xi$ on a manifold $Y$ we may define $c(Y,\xi):=c(P,\phi)$, where 
$(P,\phi)$ is an open book decomposition of $Y$ adapted to the contact 
structure $\xi$. The class $c(P,\phi)$ is invariant under handle slides, 
isotopies and positive Giroux stabilizations (see \cite{HKM}).
Thus $c(P,\phi)$ does not depend on the specific choice of adapted open
book and is an isotopy invariant of the contact manifold $(Y,\xi)$. This
cohomology class is called {\bf contact element}.

\subsubsection{The Invariant LOSS}\label{invariantLOSS}
Ideas very similar to those used to define the contact element are can
be utilized to define an invariant of Legendrian knots we will 
briefly call LOSS. This invariant is due to
{\bf L}isca, {\bf O}zsv\'{a}th, {\bf S}tipsicz and {\bf S}zab\'{o}
and was defined in \cite{LOSS}. It is basically the contact element,
but now it is interpreted as sitting in a filtered Heegaard Floer complex.
The filtration is constructed with respect to a fixed Legendrian 
knot.
\begin{figure}[ht!]
\labellist\small\hair 2pt
\pinlabel {Page $P\!\times\!\{1/2\}$ of the open book} [bl] at 131 189
\pinlabel $w$ [l] at 95 116
\pinlabel $z$ [b] at 167 104
\pinlabel $w$ [tl] at 332 53
\pinlabel $z$ [t] at 408 85
\endlabellist
\centering
\includegraphics[height=3cm]{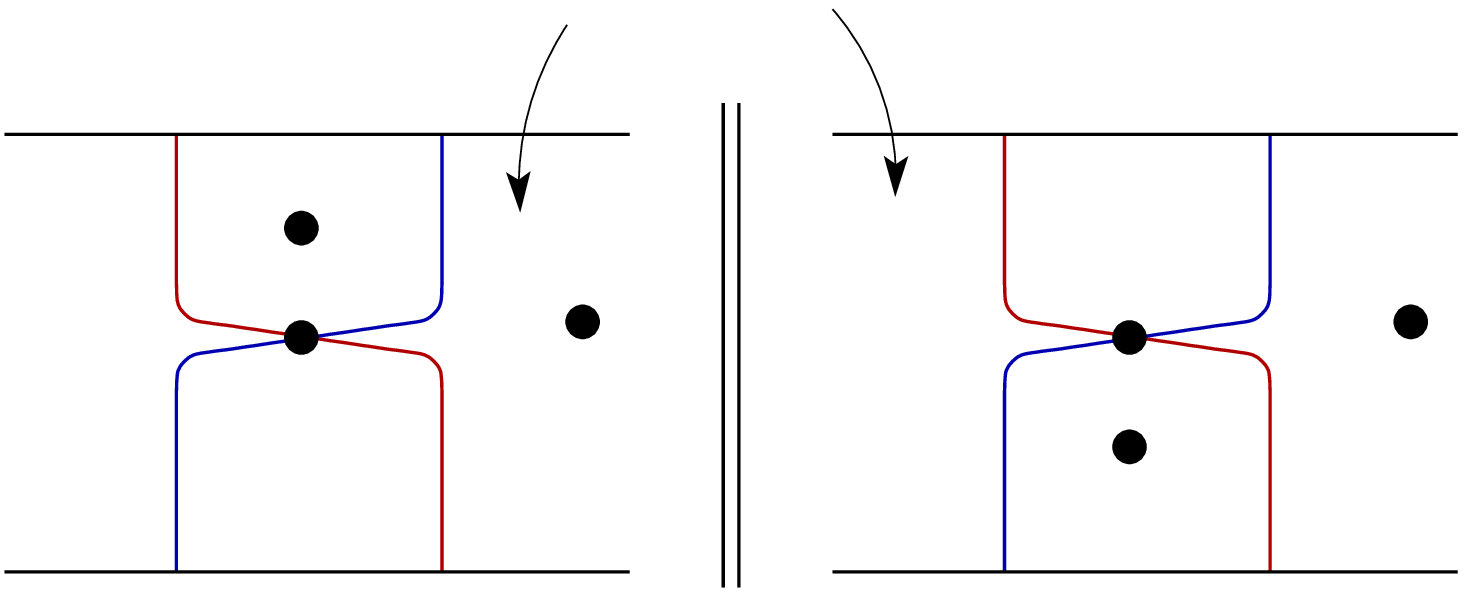}
\caption{Positioning of the point $w$ depending on the knot orientation.}
\label{Fig:wpointpos}
\end{figure}

Let $(Y,\xi)$ be a contact manifold and $L\subset Y$ a Legendrian 
knot. There is an open book decomposition of $Y$
subordinate to $\xi$ such that $L$ sits on the page $P\times\{1/2\}$ of 
the open book.
Choose a cut system that induces an $L$-adapted Heegaard 
diagram (cf.~\S\ref{postwist}, Definition~\ref{DefOne} and Lemma~\ref{LemOne}).
Figure~\ref{Fig:wpointpos} illustrates the positioning of a point $w$ in
the Heegaard diagram induced by the open book. 
Similar to the case
of the contact element there is one specific generator of $\cfhat(-Y)$
sitting of $P\times\{1/2\}$. This element may be interpreted as
sitting in $\cfkhat(-Y,L)$ and is a cycle there, too. The induced
element in the knot Floer homology is denoted by $\loss(L)$.
\begin{rem} Since this is an important issue we would like to recall the
relation between the pair $(w,z)$ and the knot orientation. In homology
we connect $z$ with $w$ in the complement of the $\alpha$-curves and 
$w$ with $z$ in the complement of the $\beta$-curves (oriented as is
obvious from definition). In {\bf cohomology} we orient in the opposite
manner, i.e.~we move from $z$ to $w$ in the complement of the $\beta$-curves
and from $w$ to $z$ in the complement of the $\alpha$-curves.
\end{rem}

\section{Algebraic Preliminaries}\label{prelim:01:3}
We outline some algebraic tools used in the next sections. We 
present this material for the sake of completeness.
\begin{lem}\label{alprim01} Suppose we are given two complexes $(C,\partial_C)$ and 
$(D,\partial_D)$ and a morphism $f\co D\lra C$ of complexes. Then 
$(C\oplus D,\partial^f)$ is a chain complex where 
$\partial^f:=\partial_C+f-\partial_D$, i.e.
\[
  \partial^f=\left(
  \begin{matrix}\partial_C & f \\ 
                         0 & -\partial_D
  \end{matrix}\right).
\]
\end{lem}
\begin{proof} For $(p,q)\in C\oplus D$ we calculate
\begin{eqnarray*}
(\partial^f)^2(p,q)
&=&
\partial^f\Bigl(\partial_Cp+f(q),-\partial_D q\Bigr)\\
&=&\Bigl(\partial_C^2p+\partial_Cf(p)+f(-\partial_Dp),\partial_D^2p\Bigr)\\
&=&0,
\end{eqnarray*}
where the last equality holds, since $\partial_C$ and $\partial_D$ are differentials 
and $f$ is a chain map.
\end{proof}
A nice, immediate consequence of this construction is the 
following Lemma.
\begin{lem} \label{propexact} There is a long exact sequence
\begin{diagram}[size=2em,labelstyle=\scriptstyle]
 \dots & \rTo^{-f_*} &H_*(C,\partial_C)&&\rTo^{\Gamma_1}
 &&H_*(C\oplus D,\partial^f)&&\rTo^{\Gamma_2}&&H_*(D,-\partial_D)
 &\rTo^{-f_*}&\dots,
\end{diagram}
\end{lem}
where $f_*$ is the map in homology induced by $f$, and $\Gamma_1$ 
and $\Gamma_2$ are given as follows:
\begin{itemize}
  \item $\Gamma_1$ is induced by the map 
  \[
    \gamma_1\co(C,\partial_C)\lra (C\oplus D,\partial^f),\;x\lmt x\oplus 0;
  \]
  \item $\Gamma_2$ is induced by the map 
  \[
    \gamma_2\co(C\oplus D,\partial^f)\lra(D,-\partial_D),\;x\oplus y\lmt -y.
  \]
\end{itemize}
\begin{proof} We first have to see that $\gamma_1$ and $\gamma_2$ 
are chain maps. Given an element $c\in C$, observe that
\[
  \gamma_1(\partial_C c)
  =
  \partial_C c
  =
  \partial^f c
  =
  \partial^f\gamma_1(c).
\] 
Furhtermore, we see that  
\[
  \gamma_2(\partial^f(c\oplus 0))
  =
  \gamma_2(\partial_C c)
  =
  0
  =
  \gamma_2(c\oplus 0)
  =
  -\partial_D(\gamma_2(c\oplus 0)).
\]
We continue with an element $d\in D$:
\[
  \gamma_2(\partial^f(0\oplus d))
  =
  \gamma_2(f(d)-\partial_D(d))
  =
  \partial_D(d)
  =
  -\partial_D(\gamma_2(0\oplus d)).
\]
Thus, both $\gamma_1$ and $\gamma_2$ are chain maps. 
Finally, $\gamma_1$ and $\gamma_2$ obviously fit into 
the short exact sequence
\begin{diagram}[size=2em,labelstyle=\scriptstyle]
0&\rTo&(C,\partial_C)&&\rTo^{\gamma_1}&&(C\oplus D,\partial^f)
&&\rTo^{\gamma_2}&&(D,-\partial_D)&\rTo&0
\end{diagram}
of chain complexes. Hence, by standard results in Algebraic 
Topology (see \cite{Bredon}) this short
exact sequence induces a long exact sequence
\begin{diagram}[size=2em,labelstyle=\scriptstyle]
 \dots & \rTo^{\partial_*} &H_*(C,\partial_C)
 &&\rTo^{\Gamma_1}&&H_*(C\oplus D,\partial^f)
 &&\rTo^{\Gamma_2}&&H_*(D,-\partial_D)&\rTo^{\partial_*}
 &\dots
\end{diagram}
It remains to show that the connecting homomorphism $\partial_*$ 
equals $-f_*$. Recall that for $d\in\ker(\partial_D)$ the morphism
$\partial_*$ is defined by
\[
  \partial_*[d]
  =
  [\gamma_1^{-1}(\partial^f(\gamma_2^{-1}(d)))].
\]
Of course, $\gamma_1$ and $\gamma_2$ are not necessarily invertible. However, 
we take the preimages as given in the equation, and, by standard algebraic topology, all
the elements in the preimage will belong to the same equivalence class. Observe:
\begin{eqnarray*}
  \partial_*[d]
  &=&[\gamma_1^{-1}(\partial^f(\gamma_2^{-1}(d)))]\\
  &=&[\gamma_1^{-1}(\partial^f(0\oplus-d))]\\
  &=&[\gamma_1^{-1}(-f(d))]\\
  &=&-[f(d)]\\
  &=&-f_*[d]
\end{eqnarray*}
\end{proof}
Of course, the whole construction works if $f$ goes the other way, i.e.~$f\co C\lra D$. 
In this case we form the complex $C\oplus D$ with the differential
\[
  \partial_f
  =\left(
  \begin{matrix} 
    \partial_C& 0\\
    f & -\partial_D
  \end{matrix}\right).
\]
In an analogous manner we obtain a long exact sequence
\begin{diagram}[size=2em,labelstyle=\scriptstyle]
 \dots & \rTo^{-f_*} &H_*(D,-\partial_D)&&\rTo^{\Gamma_1}
 &&H_*(C\oplus D,\partial_f)&&\rTo^{\Gamma_2}&&H_*(C,\partial_C)
 &\rTo^{-f_*}&\dots
\end{diagram}

\section{Two new Exact Sequences in Heegaard Floer Homology}\label{introd}
\subsection{Positive Dehn twists}\label{postwist}
Let an open book $(P,\phi)$ and a homologically essential closed 
curve $\delta$ in $P$ be given. We first ask how a Dehn twist along $\delta$ would 
change the associated Heegaard Floer homology. To do that, we first
have to see that there is a specific choice of attaching circles  
(cf.~\S\ref{prelim:01:2})
that are -- in a sense -- adapted to the closed curve $\delta$.
\begin{definition}\label{DefOne} Let a Heegaard diagram $(\Sigma,\alpha,\beta)$ 
and a homologically essential, simple, closed curve $\delta$ on $\Sigma$ be given. The
Heegaard diagram $(\Sigma,\alpha,\beta)$ is called
{\bf $\delta$-adapted} if the following conditions hold. 
\begin{enumerate}
  \item It is induced by an open book and the pair $\alpha$, $\beta$ is 
  induced by a cut system (cf.~\S\ref{prelim:01:2}) for this open book.
  \item The curve $\delta$ intersects $\beta_1$ once and does not intersect any 
  other of the $\beta_i$, $i\geq 2$.
\end{enumerate}
\end{definition}
We can always find $\delta$-adapted Heegaard diagrams. This is already
stated in \cite{HKM} and \cite{LOSS} but not proved. We wish to give
a proof because this specific choice is crucial throughout this article.
\begin{lem}\label{LemOne} Let $(P,\phi)$ be an open book and 
$\delta\subset P$ a homologically essential closed curve. There is a 
choice of cut system on $P$ that induces a $\delta$-adapted 
Heegaard diagram.
\end{lem}
Observe that  $a_1,\dots,a_{n}$ to be a cut system of a page $P$ 
essentially means to be a basis of $H_1(P,\partial P)$:
Suppose the curves are not linearly independent. In this
case we are able to identify a surface $F\subset P$, $F\not=P$,
bounding a linear combination of some of the curves $a_i$.
But this means the cut system disconnects the page $P$ in 
contradiction to the definition. Conversely, suppose the curves
in the cut system are homologically linearly independent.
In this case the curves cannot disconnect the page. If they
disconnected, we could identify a surface $F$ in $P$ with boundary
a linear combination of some of the $a_i$. But this contradicts
their linear independence. The fact that $\Sigma\backslash\{a_1,\dots,a_n\}$
is a disc shows that every element in $H_1(P,\partial P)$ can be
written as a linear combination of the curves $a_1,\dots,a_n$.
\begin{proof} Without loss of generality, we assume that $P$ 
has connected boundary: Suppose the boundary of $P$ has two
components. Choose a properly embedded arc connecting both
components of $\partial P$. Define this curve to be the first
curve $a_0$ in a cut system. Cutting out this curve $a_0$, we obtain a
surface with connected boundary. The curve $a_0$ determines two segments $S_1$
and $S_2$ in the connected boundary. We can continue using
the construction process for connected binding we state below. We 
just have to check the boundary points of the curves to 
remain outside of the segments $S_1$ and $S_2$. Given that $P$ has more
than two boundary components, we can, with this algorithm, inductively 
decrease the number of boundary components.\vspace{0.3cm}\\
The map $\phi$ is an element of the mapping 
class group of $P$. Thus, if $\{a_1,\dots,a_{n}\}$
is a cut system, then $\{\phi(a_1),\ldots,\phi(a_{n})\}$ is a 
cut system, too. It suffices to show that there is a cut system 
$\{a_1,\ldots,a_{n}\}$ such that $\delta$ intersects $a_i$ once 
if and only if $i=1$.
\begin{figure}[ht!]
\labellist\small\hair 2pt
\pinlabel $\gamma$ [bl] at 199 214
\endlabellist
\centering
\includegraphics[height=3cm]{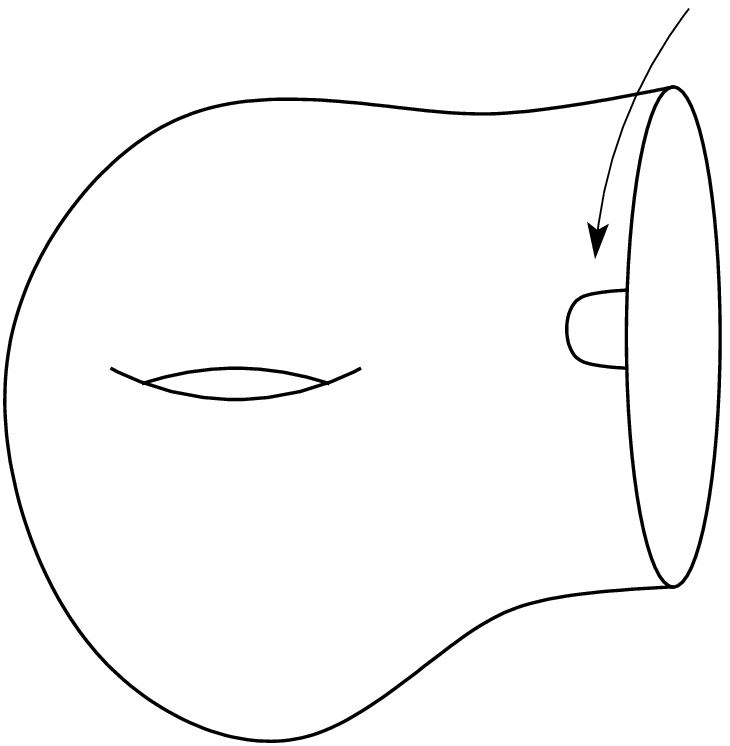}
\caption{Possible choice of curve $\gamma$.}
\label{Fig:figfour}
\end{figure}

We start by taking a band sum of $\delta$ with a small 
arc $\gamma$ as shown in Figure~\ref{Fig:figfour}. We are 
free to choose the arc $\gamma$.
Denote the result of the band sum by $a_2$. The arc $a_2$ indeed bounds a compressing disc 
in the respective handlebody because its boundary lies 
on $\partial P$. Because of our prior observation it suffices
to show that $a_2$ is a primitive class 
in $H_1(P,\partial P)$. This can be seen with an elementary homological
argument using a Mayer-Vietoris computation and the lemma below.
Cut open the surface along $\delta$. We obtain two new boundary components, $C_1$
and $C_2$ say, which we can connect with the boundary of $P$ with two arcs. These
two arcs, in $P$, determine a properly embedded curve, $a_1$ say, whose boundary
lies on $\partial P$. Furthermore, $a_1$ intersects $\delta$ in one single point, transversely.
The curve $a_1$ is primitve, too. To see, that we can extend to a cut system such that $\delta$ is
disjoint from $a_3,\dots,a_n$, cut open the surface $P$ along $\delta$ and $a_1$.
We obtain a surface $P'$ with one boundary component. The curves $\delta$ and $a_1$ determine
$4$ segments, $S_1,\dots, S_4$ say, in this boundary. We extend $a_2$ to a cut system 
$a_2,\dots,a_n$ of $P'$ and arrange the boundary points of the curves $a_3,\dots, a_n$ to 
be disjoint from $S_1,\dots,S_4$. The set $a_1,\dots,a_n$ is a cut system of $P$ with the
desired properties.
\end{proof}
As a consequence of the proof we may arrange $\delta$
to be a push-off of $a_2$ outside a small neighborhood where the
band sum is performed. Geometrically spoken, we cut open $\delta$
at one point, and move the boundaries to $\partial P$ to get
$a_2$. Given a positive Giroux stabilization, we can find a special 
cut system which is adapted to the curve $\gamma$. It is not hard to 
see that there is only one homotopy class of triangles that connect
the old with the new contact element and that the associated moduli 
space is a one-point space.
\begin{lem}\label{primitive} An embedded circle $\delta$ in an 
orientable, closed surface $\Sigma$ which is homologically essential 
is a primitive class of $H_1(\Sigma,\Z)$.
\end{lem}
\begin{proof}
Cut open the surface $\Sigma$ along $\delta$. We obtain a connected surface $S$
with two boundary components since $\delta$ is homologically essential 
in $\Sigma$. We can recover the surface $\Sigma$ by connecting 
both boundary components of $S$ with a $1$-handle and then capping off 
with a disc. There is a knot $K\subset S\cup h^1$ 
intersecting the co-core of $h^1$ only once and intersecting 
$\delta$ only once, too. To construct this knot take a union of 
two arcs in $S\cup h^1$ in the following 
way: Namely, define  $a$ as the core 
of $h^1$, i.e.~as $D^1\times\{0\}\subset D^1\times D^1\cong h^1$ 
and let $b$ be 
a curve in $S$, connecting the two components of the attaching 
sphere $h^1$ in $\partial S$. This curve exists since $\delta$ is
homologically essential which implies that it is non-separating. We 
define $K$ to be $a\cup b$. 
Obviously,
\[
  \pm1
  =
  \#(K,\delta)
  =
  \bigl<PD[K],[\delta]\bigr>.
\]
Since $H_1(\Sigma;\Z)$ is torsion free 
$H^1(\Sigma;\Z)\cong\mbox{\rm Hom}(H_1(\Sigma;\Z),\Z)$. Thus, 
$[\delta]$ is primitive.
\end{proof}
Figure~\ref{Fig:figone} depicts a small neighborhood of
the point $\delta\cap\beta_1$ in the Heegaard 
diagram induced by the open book decomposition. The page at 
the right side of the boundary pictured in 
Figure~\ref{Fig:figone} is $P\times\{1/2\}$. The dotted line
indicates the neighborhood of $\partial P$ where the monodromy
$\phi$ is the identity. The proof of Lemma~\ref{LemOne} shows 
that we can arrange a neighborhood of $\delta\cap\beta_1$ to 
look like in Figure~\ref{Fig:figone}, i.e.~it
is possible to arrange the curve $\delta$ and the attaching 
circles like indicated in Figure~\ref{Fig:figone} due to the 
arguments given in the proof of Lemma~\ref{LemOne}.
\begin{figure}[ht!]
\labellist\small\hair 2pt
\pinlabel {boundary of $P$} [b] at 181 317
\pinlabel {$\delta$} [Bl] at 96 223
\pinlabel {$z$} [Bl] at 243 232
\pinlabel {$\dom_z$} [Bl] at 315 226
\pinlabel {$_2$} [l] at 345 122
\pinlabel {$_1$} [b] at 378 95
\pinlabel {$\beta_2$} [Bl] at 365 298
\pinlabel {$\alpha_2$} [Bl] at 365 262
\pinlabel {$\beta_1$} [Bl] at 365 190
\pinlabel {$\alpha_1$} [Bl] at 365 154
\pinlabel {$\beta_2$} [Bl] at 365 9
\pinlabel {$\alpha_2$} [Bl] at 365 46
\endlabellist
\centering
\includegraphics[height=5cm]{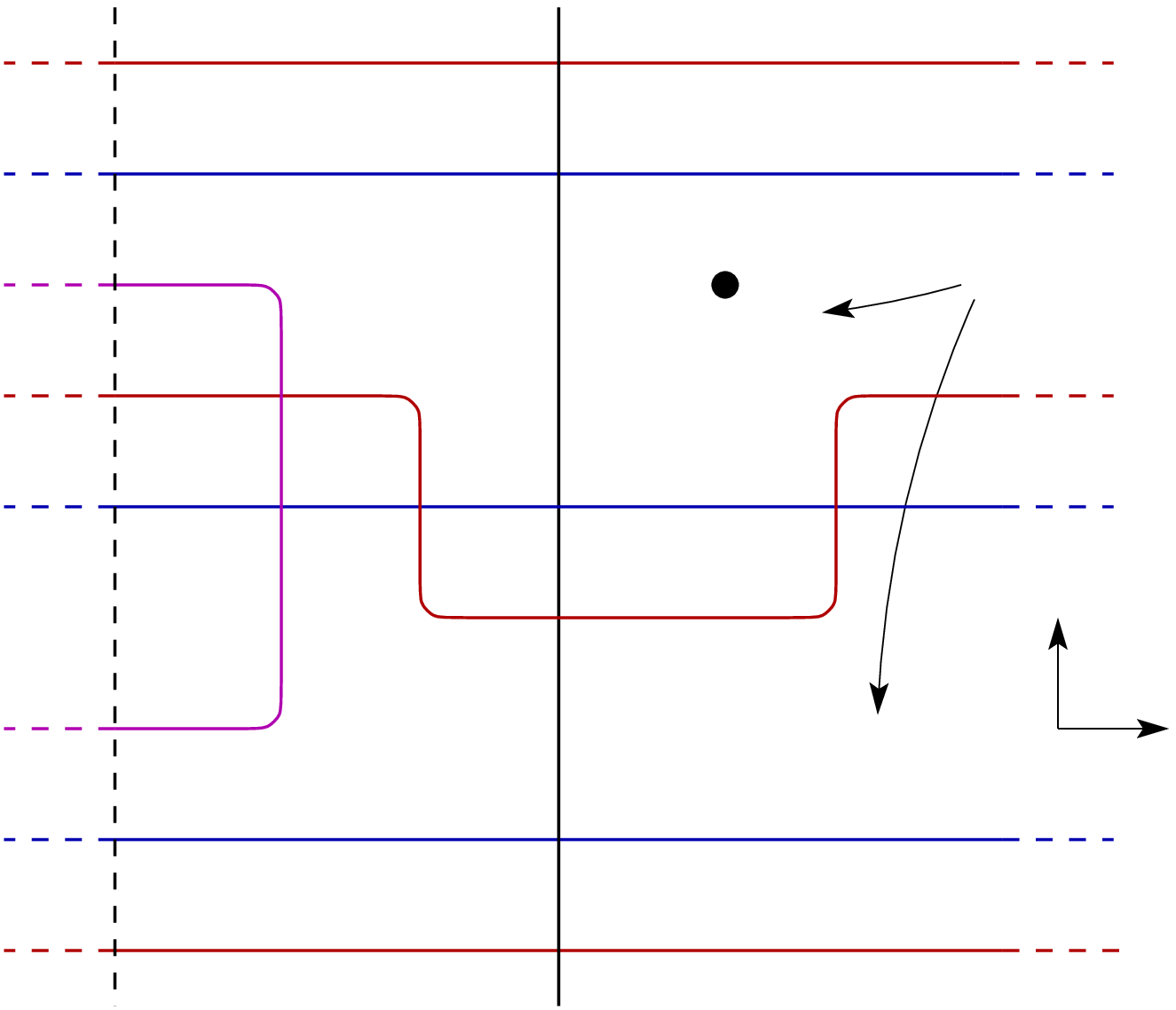}
\caption{A small neighborhood of $\delta\cap\beta_1$ in the Heegaard surface 
$\Sigma=P\times\{1/2\}\cup (-P)\times\{0\}$.}
\label{Fig:figone}
\end{figure}

With respect to the surface orientation given in Figure~\ref{Fig:figone}
this is the appropriate setup for performing a positive Dehn twist
along $\delta$: Denote by $\betaprime$ the $\beta$-curves after
performing the Dehn twist. Obviously, 
$\betaprime=\{\betaprime_1,\beta_2,\dots\beta_{2g}\}$. Observe that
\begin{equation}
  \talpha\cap\tbetaprime=\talpha\cap\tbeta\sqcup\talpha\cap\tdelta,
\end{equation}
where $\tdelta$ is given by the set 
$\delta=\{\delta,\beta_2,\dots,\beta_{2g}\}$ 
(by abuse of notation since $\delta$ also denotes the curve 
on $P$ but what is meant will be clear from the context). 
The set of curves $\delta$ may be interpreted as a set 
of attaching circles. In the following we will call the arc 
$\betaprime_1\cap\beta_1$ the 
{\bf $\beta$-part of $\betaprime_1$}
and the arc $\betaprime_1\cap\delta$ the 
{\bf $\delta$-part of $\betaprime_1$}.
Figure~\ref{Fig:figtwo} depicts the situation before and after 
the Dehn twist.
\begin{figure}[ht!]
\labellist\small\hair 2pt
\pinlabel {boundary of $P$} [b] at 235 328
\pinlabel {boundary of $P$} [b] at 667 328
\pinlabel {$z$} [r] at 268 250
\pinlabel {$z$} [r] at 703 250
\pinlabel {$\dom_*$} [Br] at 119 212
\pinlabel {$\dom_*$} [Br] at 544 212
\pinlabel {$\dom_{**}$} [b] at 160 168
\pinlabel {$\dom_{**}$} [b] at 597 168
\pinlabel {$w$} [l] at 105 187
\pinlabel {$w$} [l] at 520 187
\pinlabel {$\dom_z$} [Bl] at 319 243
\pinlabel {$\dom_z$} [Bl] at 747 243
\pinlabel {$_2$} [l] at 348 122
\pinlabel {$_1$} [B] at 375 105
\pinlabel {$_2$} [l] at 780 122
\pinlabel {$_1$} [B] at 808 105
\pinlabel {$\beta_2$} [l] at 365 312
\pinlabel {$\alpha_2$} [l] at 365 276
\pinlabel {$\beta_1$} [l] at 365 204
\pinlabel {$\alpha_1$} [l] at 365 168
\pinlabel {$\alpha_2$} [l] at 365 60
\pinlabel {$\beta_2$} [l] at 365 25
\pinlabel {$\beta_2$} [l] at 796 312
\pinlabel {$\alpha_2$} [l] at 796 276
\pinlabel {$\betaprime_1$} [l] at 796 204
\pinlabel {$\alpha_1$} [l] at 796 168
\pinlabel {$\alpha_2$} [l] at 796 60
\pinlabel {$\beta_2$} [l] at 796 25
\endlabellist
\centering
\includegraphics[height=5cm]{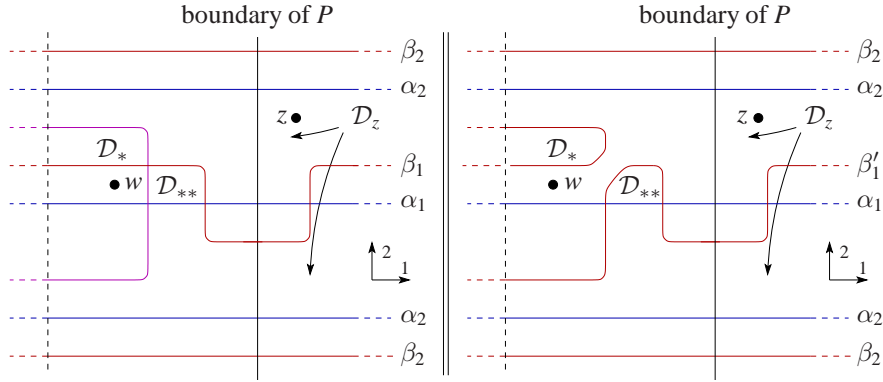}
\caption{Before and after the positive Dehn twist.}
\label{Fig:figtwo}
\end{figure}

The main observation is that there can be no holomorphic disc
in $(\Sigma,\alpha,\betaprime)$ that connects a 
$\talpha\cap\tbeta$-intersection of $\talpha\cap\tbetaprime$ with 
a $\talpha\cap\tdelta$-intersection of $\talpha\cap\tbetaprime$. 
Suppose there is a disc $\phi$ starting at 
$x\in\talpha\cap\tbeta$ and going to $y\in\talpha\cap\tdelta$
along its $\alpha$-boundary. Then, at the $\beta$-boundary, the 
disc $\phi$ has to run from $y$ to $x$ along the 
$\betaprime$-curves. Since $\delta\cap\beta_1$ contains only 
one point, namely the intersection that can be seen in 
Figures~\ref{Fig:figone} and \ref{Fig:figtwo}, the disc has 
to run through either $\domstar$ or
$\domststar$ (since $n_z(\phi)=0$ we cannot use 
the $\dom_z$-region). But since we are moving from the 
$\delta$-part of $\betaprime_1$ to the $\beta$-part
of $\betaprime_1$, we see that $n_*(\phi)<0$ or 
$n_{**}(\phi)<0$, in contradiction to holomorphicity. So, there
are just three choices for the $\beta$-boundary of a 
holomorphic disc.
\begin{enumerate}
  \item It starts at the $\delta$-part of $\beta_1'$ and stays 
  there. 
  \item It starts at the $\beta$-part of $\beta_1'$ and stays 
  there. 
  \item It starts at the $\beta$-part of $\beta_1'$ and runs 
  to the $\delta$-part of $\beta_1'$ and stays there. 
\end{enumerate}
This immediately shows that
\[
  \hfhat(Y^\delta)=H_*(\cfhat(\alpha,\beta)\oplus\cfhat(\alpha,\delta),\partial), 
\]
where $\partial$ is of the form 
\[
  \left(\begin{matrix}A & 
  C\\
  0&B\end{matrix}\right).
\]
If we perform a negative Dehn twist along $\delta$ in the
situation indicated in Figure~\ref{Fig:figone}, we would connect 
$\domstar$ with $\domststar$ and keep separate $\dom_w$ and $\dom_z$. 
Observe that we would have, a priori, no control of holomorphic 
discs like in the case of positive Dehn twists. To get back into
business, in case of negative Dehn twists, we have to first isotope 
$\delta$ inside the page of the open book appropriately (see \S\ref{negtwist}).

\begin{prop}\label{THMTHM} Let $(\Sigma,\alpha,\beta)$ be a 
$\delta$-adapted Heegaard diagram of $Y$ and denote by $Y^\delta$ 
the manifold obtained from $Y$ by composing the gluing map, given 
by the attaching curves $\alpha$, $\beta$, with a positive 
Dehn twist along $\delta$ as indicated in Figure~\ref{Fig:figtwo}. 
Then the following holds:
\[
  \hfhat(Y^\delta)
  \cong 
  H_*(\cfhat(\alpha,\beta)
  \oplus
  \cfhat(\alpha,\delta),
  \partial^{f}),
\]
where $\partial^{f}$ is of the form
\[
  \left(\begin{matrix}\parhat_{\alpha\beta}^w & 
  f\\
  0&\parhat_{\alpha\delta}^w\end{matrix}\right)
\]
with $f$ a chain map 
between $(\cfhat(\alpha,\delta),\parhat^w_{\alpha\delta})$ and 
$(\cfhat(\alpha,\beta),\parhat^w_{\alpha\beta})$.
\end{prop}
\begin{proof}  
There is a natural identification of intersection points 
\begin{diagram}[size=2em,labelstyle=\scriptstyle]
  \talpha\cap\tbetaprime & &
  \pile{\rTo \\ \lTo} & 
  \talpha\cap\tbeta\sqcup\talpha\cap\tdelta,
\end{diagram}
i.e.~we get an isomorphism 
\[
  \epsilon\co\cfhat(\alpha,\betaprime)
  \overset{\cong}{\lra}\cfhat(\alpha,\beta)
  \oplus\cfhat(\alpha,\delta)
\] 
of modules. Pick an intersection point 
$x\in\talpha\cap\tbetaprime$ such that 
$\epsilon(x)\in\talpha\cap\tbeta$. Looking at the boundary
\begin{equation}
  \parhat^\delta x=\sum_{y}\sum_{\phi}\#\modhatphi\cdot y
  \label{boundary}
\end{equation}
we want to see that the moduli space of holomorphic discs 
connecting $x$ with an intersection 
$y\in\epsilon^{-1}(\talpha\cap\tdelta)$ is empty: Assume this 
were not the case. This means there were a holomorphic disc $\phi$ 
connecting $x$ with an element 
$y=(y_1,\dots,y_n)\in\epsilon^{-1}(\talpha\cap\tdelta)$. 
Observe that $y_1$ is a point in $\delta\cap\alpha_1$. 
Hence, $\dom(\phi)$ includes $\domstar$ or $\domststar$ since 
these are the only domains giving a 
connection between $\talpha\cap\tbeta$ and $\talpha\cap\tdelta$. 
Boundary orientations force the coefficient of $\phi$ at 
$\domstar$ or $\domststar$ to be negative. Since holomorphic 
maps are orientation preserving, this cannot be the case. So, 
the point $x$ can be connected to points in 
$\epsilon^{-1}(\talpha\cap\tbeta)$ only.\vspace{0.2cm}\\
Next observe that discs $\phi$ appearing in the sum 
$(\ref{boundary})$ all have the property 
$n_*(\phi)=n_{**}(\phi)=0$. Indeed, suppose there were a disc $\phi$
with nonnegative intersection $n_*$ or $n_{**}$. The 
$\beta$-boundary of $\phi$ starts at $x$ and runs through
$\partial\domstar$ or $\partial\domststar$. The disc
$\phi$ is holomorphic, so, the $\beta$-boundary runs from
the $\beta$-part to the $\delta$-part of $\tbetaprime$.
At the end of the $\beta$-boundary of $\phi$ the disc
converges to a point in $\talpha\cap\tbeta$. Thus, the
$\beta$-boundary of $\phi$ has to come back through 
either $\domstar$ or $\domststar$. The boundary orientation would
force $\phi$ to negatively intersect $\{*\}\times\symgmo$ or
$\{**\}\times\symgmo$. This cannot happen.
\begin{figure}[ht!]
\labellist\small\hair 2pt
\pinlabel {$\dom_*$} [Br] at 72 238
\pinlabel {$c$} [r] at 102 205
\pinlabel {$a$} [B] at 50 152
\pinlabel {$w$} [l] at 63 102
\pinlabel {$d$} [r] at 103 79
\pinlabel {$\alpha_1$} [t] at 12 50
\pinlabel {$\betaprime_1$} [t] at 122 50
\pinlabel {$\dom_{**}$} at 156 98
\pinlabel {$b$} [B] at 176 153
\pinlabel {$z$} [l] at 170 205

\pinlabel {$\dom_*$} [Br] at 315 238
\pinlabel {$c$} [r] at 345 205
\pinlabel {$z$} [l] at 413 205
\pinlabel {$w$} [B] at 303 98
\pinlabel {$d$} [r] at 346 79
\pinlabel {$\dom_{**}$} at 399 98
\pinlabel {$\alpha_1$} [t] at 255 50
\pinlabel {$\delta$} [r] at 348 145

\pinlabel {$\dom_*$} [Br] at 558 238
\pinlabel {$a$} [B] at 530 152
\pinlabel {$b$} [B] at 662 153
\pinlabel {$z$} [l] at 656 205
\pinlabel {$w$} [l] at 539 102
\pinlabel {$\dom_{**}$} at 642 98
\pinlabel {$\alpha_1$} [t] at 498 50
\pinlabel {$\beta_1$} [b] at 592 148

\endlabellist
\centering
\includegraphics[width=11cm]{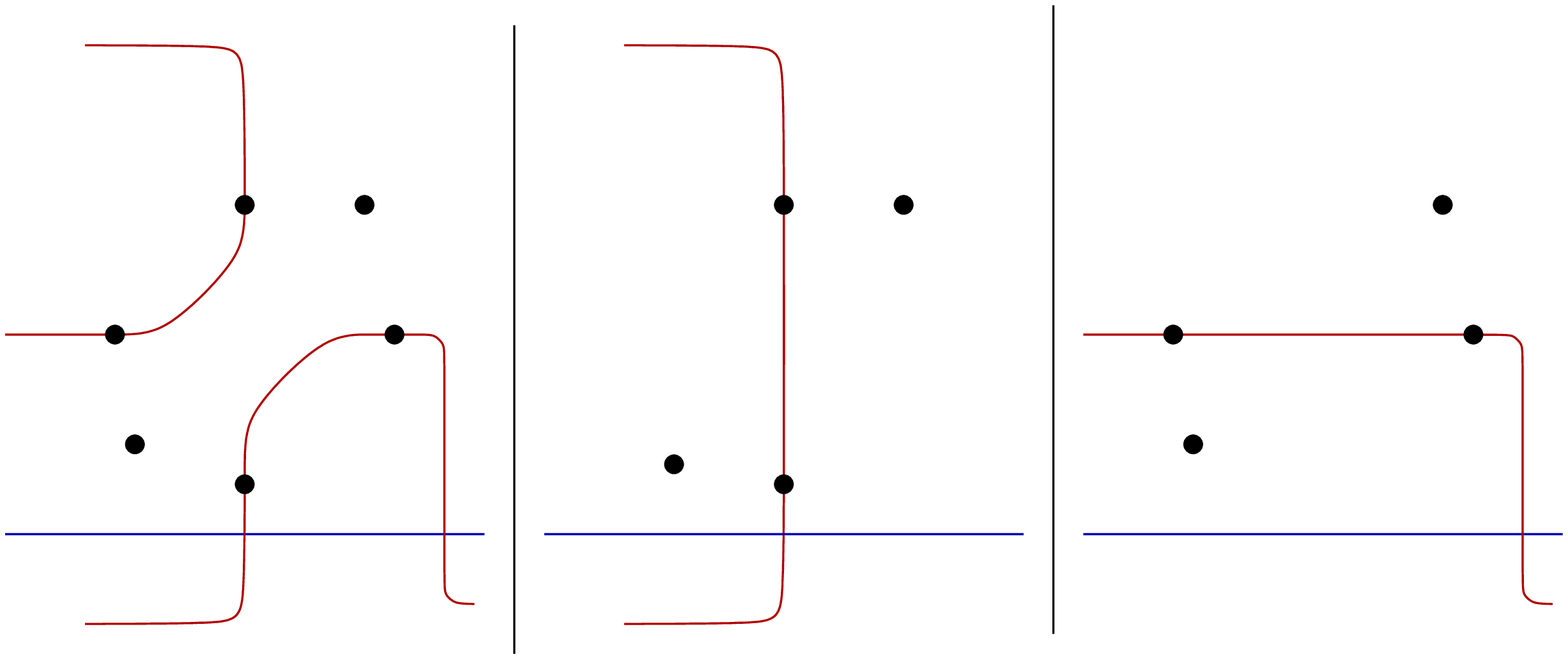}
\caption{Picture of the three different boundary conditions
arising in our discussion.}
\label{Fig:figfive}
\end{figure}

Denote by $[a,c]$ the small arc in $\beta_1'$ running through 
Figure~\ref{Fig:figfive} and define $[b,d]$ analogously. All 
discs arising in the sum have boundary conditions in $\talpha$ and 
\[
  \tbetaprime\backslash\{\{[a,c]\sqcup[b,d]\}
  \times\beta_2\times\ldots\times\beta_g\}.
\] 
Observe that 
$\tbetaprime\backslash\{\{[a,c]\sqcup[b,d]\}
\times\beta_2\times\ldots\times\beta_g\}$ has two components, one 
lying in $\tbeta$ and one lying in $\tdelta$. Since the 
$\beta$-boundary of the disc $\phi$ starts in, 
$\tbeta$ it remains there all the time. Moreover, looking at
discs $\phi$ in $(\Sigma,\alpha,\beta,z,w)$ with $n_z(\phi)=n_w(\phi)=0$, an 
analogous line of arguments as above shows that the $\beta$-boundary
of these discs stays away from 
\[
  [a,b]\times\beta_2\times\ldots\times\beta_g,
\]
where $[a,b]$ is the arc in $\beta$ pictured in the right of 
Figure~\ref{Fig:figfive}. Thus, the boundary conditions for
discs connecting intersections $\talpha\cap\tbeta$ are the same
in $(\Sigma,\alpha,\betaprime,z)$ and $(\Sigma,\alpha,\beta,z,w)$.
Thus, we have
\[
  \parhat^\delta x=\parhat_{\alpha\beta}^wx.
\]
Now suppose that $x\in\epsilon^{-1}(\talpha\cap\tdelta)$. Then 
\begin{eqnarray*}
  \parhat^\delta x
  &=&
  \sum_y
  \sum_{\phi}\#\modhatphi\cdot y
  \\
  &=&
  \sum_{y\in\talpha\cap\tdelta}
  \sum_{\phi}\#\modhatphi\cdot y 
  +
  \sum_{z\in\talpha\cap\tbeta}
  \sum_\phi\#\modhatphi\cdot z.
\end{eqnarray*}
With an analogous line of arguments as above we see that the 
first sum counts discs with $n_*=n_{**}=n_z=0$ only. The 
triviality of these intersection numbers and holomorphicity
implies that the discs have boundary conditions in 
$\talpha$ and 
\[
  \tbetaprime\backslash\{\{[a,c]\sqcup[b,d]\}
  \times\beta_2\times\ldots\times\beta_g\}.
\] 
As mentioned above this set has two components where one of them 
lies in $\tdelta$. The $\beta$-boundary of $\phi$ starts in 
$\tdelta$ and therefore remains there all the time. Again, we see
that discs connecting intersection points $\talpha\cap\tdelta$
in $(\Sigma,\alpha,\betaprime,z)$ and $(\Sigma,\alpha,\delta,z,w)$ 
have to fulfill identical boundary conditions. Thus, the moduli spaces
are isomorphic. This shows the equality
\[
  \parhat^\delta x
  =
  \parhat_{\alpha\delta}^wx
  +
  \sum_{z\in\talpha\cap\tbeta}
  \sum_\phi\#\modhatphi\cdot z.
\]
In the right sum we only count discs where $n_*\not=0$ or 
$n_{**}\not=0$. We will denote this right sum with $f(x)$.
We have to see that $f$ defines a chain map 
\[
  f\co(\cfhat(\alpha,\delta),\parhat^w_{\alpha\delta})
  \lra(\cfhat(\alpha,\beta),\parhat^w_{\alpha\beta}).
\] 
This can be proved in two ways: We know that 
$\partial^\delta=\partial^w_{\alpha\beta}
+\partial^w_{\alpha\delta}+f$. Hence, $f$ is a sum of three 
boundaries. The equality $0=(\partial^\delta)^2$ implies 
that $f$ is a chain map (cf.~Lemma~\ref{alprim01}). 
The second way is to test the chain map property directly. To do so, pick a generator 
$y\in\talpha\cap\tbetaprime$ lying in the preimage of 
$\talpha\cap\tdelta$ under $\epsilon$.
Observe that 
$(\parhat^w_{\alpha\beta}\circ f-f\circ\parhat^w_{\alpha\delta})(x)$
equals
\begin{eqnarray*}
  &\,&\sum_{z\in\talpha\cap\tdelta}
  \Bigl(\sum_{(y,\phi_2,\phi_1)}
  \#\Mhat(\phi_2)\#\Mhat(\phi_1)
  -\sum_{(y',\phi_2',\phi_1')}\#\Mhat(\phi_2')
  \#\Mhat(\phi_1')\Bigr)\cdot z\\
  &=&\sum_{z\in\talpha\cap\tdelta}c(x,z)\cdot z,
\end{eqnarray*}
where the first sum in the definition of $c(x,z)$ goes 
over elements $(y,\phi_2,\phi_1)$ in the set 
$\talpha\cap\tbeta\times\pitwo(y,z)\times\pitwo(x,y)$ with 
$\mu(\phi_2)=\mu(\phi_1)=1$, and the second sum goes over 
$(y',\phi_2',\phi_1')\in\talpha\cap\tdelta
\times\pitwo(y,z)\times\pitwo(x,y)$ 
with $\mu(\phi_2')=\mu(\phi_1')=1$. Furthermore, look 
at the 
boundary of a moduli space $\Mhat(\phi)$ connecting a 
point in $\talpha\cap\tdelta$ with a point in 
$\talpha\cap\tbeta$ with $\mu(\phi)=2$. Observe 
that we do not have to take care of boundary degenerations 
or spheres bubbling off since we are looking for maps with 
$n_z=0$ (cf.~\cite{OsZa01}). The only phenomenon 
appearing at the boundary is breaking. The boundary of 
$\Mhat(\phi)$ is modelled on
\[
  \bigsqcup_{\phi_1*\phi_2=\phi}\Mhat(\phi_1)\times\Mhat(\phi_2).
\]
There are two cases. Either $n_*(\phi_1)=n_*(\phi)$ or 
$n_*(\phi_2)=n_*(\phi)$ (the discussion for $n_{**}$ is analogous):
\begin{figure}[ht!]
\labellist\small\hair 2pt
\pinlabel {Intersection points in $\talpha\cap\tdelta$} [B] at 224 388
\pinlabel {$k$} at 18 227
\pinlabel {$0$} at 18 144
\pinlabel {$m$} at 405 227
\pinlabel {$n$} at 405 144
\pinlabel {$n_*\!=\!k$} at 212 231
\pinlabel {Intersection points in $\talpha\cap\tbeta$} [t] at 222 45
\endlabellist
\centering
\includegraphics[width=6cm]{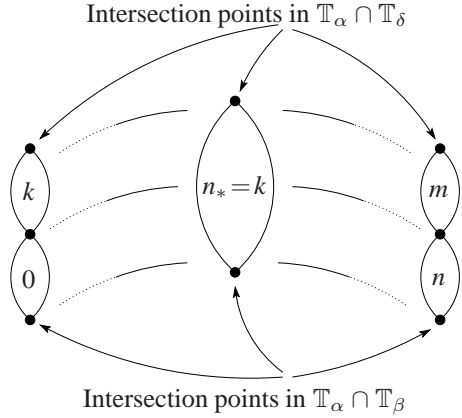}
\caption{Here we figure a moduli space with $\mu=2$ and its 
possible ends.}
\label{Fig:figsix}
\end{figure}

To prove this, we have to show that a given family of discs $\phi_n$
in $\Mhat(\phi)$ cannot converge to a broken disc $\phi_1*\phi_2$
with $n=n_*(\phi_1)\not=0$ and $m=n_*(\phi_2)\not=0$.
Figure~\ref{Fig:figsix} represents a moduli space of discs with $\mu=2$ 
and $n_*(\phi_n)=k$. We know that $n+m=k$, since intersection 
numbers behave additively under concatenation. Assume that $n,m$ were both 
non-zero: Since $n$ is non-zero, we know that $\phi_1$ connects 
a point in $\talpha\cap\tdelta$ with one in 
$\talpha\cap\tbeta$. The bottom intersection is a 
$\talpha\cap\tbeta$-intersection, since $\phi_n$ connects 
$\talpha\cap\tdelta$ with an $\talpha\cap\tbeta$-intersection by 
assumption. Hence, $\phi_2$ connects a point of 
$\talpha\cap\tbeta$ with a point in $\talpha\cap\tbeta$ and runs 
through the domain $\domstar$. This is simply not possible because 
of orientation reasons. Thus, either $n_*(\phi_1)=k$ and $n_*(\phi_2)=0$
or $n_*(\phi_1)=0$ and $n_*(\phi_2)=k$. This means the ends of 
$\Mhat(\phi)$ precisely look like
\[
  \Bigl(\bigsqcup_{\phi_2*\phi_1
  =\phi}\Mhat(\phi_2)^{*}\times\Mhat(\phi_1)\Bigr)
  \sqcup\Bigl(\bigsqcup_{\phi_2'*\phi_1'=\phi}\Mhat(\phi_2)
  \times\Mhat(\phi_1)^*\Bigr),
\]
where $*$ means that the associated discs have non-trivial 
intersection number $n_*$ or $n_{**}$. Now consider the union of 
moduli spaces of discs connecting the point $x$ and $z$ with 
Maslov index $2$. According to our discussion, the ends look like 
\[
  \Bigl(\bigsqcup_{(y,\phi_2,\phi_1)}
  \Mhat(\phi_2)\times\Mhat(\phi_1)^*\Bigr)
  \sqcup
  \Bigl(
  \bigsqcup_{(y',\phi_2',\phi_1')}\Mhat(\phi_2')^*
  \times\Mhat(\phi_1')\Bigr),
\]
where the first union goes over 
$(y,\phi_2,\phi_1)\in\talpha\cap\tbeta\times\pitwo(y,z)
\times\pitwo(x,y)$ with $\mu(\phi_2)=\mu(\phi_1)=1$
and the second union goes over $(y',\phi_2',\phi_1')\in
\talpha\cap\tdelta\times\pitwo(y,z)\times\pitwo(x,y)$ with 
$\mu(\phi_2')=\mu(\phi_1')=1$. Hence, the coefficients $c(x,z)$ all 
vanish, proving the theorem.
\end{proof}
An immediate, simple algebraic consequence 
(cf.~\S\ref{prelim:01:3}) of 
this description is the following Corollary.
\begin{cor}\label{motiv} Let $K\subset Y$ be the knot determined
by $\delta$. Then there is a long exact 
sequence
\begin{diagram}[size=2em,labelstyle=\scriptstyle]
 \dots & \rTo^{\partial_*} &\hfkhat(Y,K)&&\rTo^{\Gamma_1}&&
 \hfhat(Y_{-1}(K))&&\rTo^{\Gamma_2}&&\hfkhat(Y_0(K),\mu)&
 \rTo^{\partial_*}&\dots
\end{diagram}
with $\partial_*=-f_*$ where $f$ is the map defined in the proof of
Proposition~\ref{THMTHM}. The knot $\mu$ denotes a meridian of $K$.
\end{cor}
\begin{proof}
With Proposition~\ref{THMTHM} we see that 
$\hfhat(Y^\delta)$ fulfills the assumptions of 
Lemma~\ref{alprim01} and therefore Lemma~\ref{propexact}
applies. Finally, we apply Proposition~\ref{knotfloer} to identify
$H_*(\cfhat,\parhat^w)$ with the respective knot Floer homology. It 
is easy to observe that with respect to the framing induced by
the open book the manifold $Y^\delta$ 
equals $Y_{-1}(K)$, i.e.~the result of $(-1)$-surgery
along the knot $K$. We obtain the sequence
\begin{diagram}[size=2em,labelstyle=\scriptstyle]
 \dots & \rTo^{\partial_*} &\hfkhat(Y,K)&&\rTo^{\Gamma_1}&&
 \hfhat(Y_{-1}(K))&&\rTo^{\Gamma_2}&&\hfkhat(Y_{\alpha\delta},K_2)&
 \rTo^{\partial_*}&\dots,
\end{diagram}
where $(Y_{\alpha\delta},K_2)$ is the pair given by the 
data $(\Sigma,\alpha,\delta,z,w)$. It is easy to see that the pair
$(w,z)$ in the diagram $(\Sigma,\alpha,\delta)$ determines $\beta_1$ up
to orientation, i.e.~the attaching circle $\beta_1$ interpreted as a
knot in $Y_{\alpha\delta}$. This attaching circle $\beta_1$ is a meridian
for a tubular neighborhood $\mu$ of $K$ in $Y$. Finally, we have to
see that $Y_{\alpha\delta}$ equals the $0$-surgery along $K$ with respect
to the framing induced by the open book. This is straightforward.
\end{proof}
A few words about admissibility: The reader
may have noticed that we did not say anything about admissibility
of the Heegaard diagram $(\Sigma,\alpha,\delta,z,w)$, but nonetheless
talk about the knot Floer homology $\hfkhat(Y_{\alpha\delta},K_2)$
induced by this diagram. We could restrict to just saying we take
the homology induced by the data. The respective boundary operator
is well defined (finite sum) since $\parhat^\delta$ is. However, we 
would like to remark that the diagram $(\Sigma,\alpha,\delta,z,w)$
is always admissible {\it in a relaxed sense}. 
We may relax the weak-admissibility condition imposed by \ozs$\,$ and
\sza$\,$ for the definition of knot Floer homology to the 
extreme weak-admissibility condition given in Definition~\ref{extweakadm}.
The diagram $(\Sigma,\alpha,\delta,z,w)$ is always extremely 
weakly-admissible: Let $\dom$ be a non-trivial periodic domain with 
$n_w(\dom)=0$ (see~\S\ref{admsec}) and let $s$ be an arbitrary 
$\spinc$-structure such that $\bigl<c_s(s),\mathcal{H}(\dom)\bigr>=0$. By
definition of the boundary, $\partial\dom$ can be written as
\[
  \partial\dom
  =
  \sum_{i\geq1}\lambda_i\alpha_i
  +
  \kappa_1\delta
  +
  \sum_{j\geq2}\kappa_j\beta_j.
\]
Assuming that $\lambda_i\not=0$ for a $i\geq2$ or $\kappa_j\not=0$ for a
$j\geq2$, we see that $\dom$ has both positive and negative 
coefficients due to the fact that $\partial\dom$ runs through 
a configuration like given in 
Figure~\ref{Fig:zpointpos}. Thus, let us assume 
that $\lambda_i$ and $\kappa_j$ would vanish, for 
all $i,j\geq2$. The boundary of $\dom$ could be written 
as
\[
  \partial\dom
  =
  \lambda_1\alpha_1
  +
  \kappa_1\delta.
\]
However, $\kappa_1$ has to vanish, since $\delta$ runs through 
$\partial\overline{\dom_w}\cap\partial\overline{\dom_z}$ 
(see~Figure~\ref{Fig:figfive}). Finally, we get that 
$\partial\dom=\lambda_1\alpha_1$. Examining the middle part of 
Figure~\ref{Fig:figfive} we see that the part of $\alpha_1$ which is
at the right of $\delta$ is surrounded by the region $\dom_z$. Thus, 
$\lambda_1=0$.\vspace{0.3cm}\\
With help of the geometric realization of the $\bigwedge\,\!\!\!^*(H_1/Tor)$-module
structure given in \cite{OsZa01} we can easily prove the following 
proposition.
\begin{prop}\label{modstructure} 
The maps $\Gamma_1$ and $\Gamma_2$ from the exact sequence
of Corollary~\ref{motiv} respect the $\bigwedge\,\!\!\!^*(H_1/Tor)$-module 
structure of the Heegaard Floer groups in the following sense.
Let $\gamma\subset\Sigma$ be a curve. Then the following 
identities hold:
\begin{eqnarray*}
  A^{Y^\delta}_{[\gamma]_{Y^\delta}}(\Gamma_1(x))
  &=&
  \Gamma_1(A^Y_{[\gamma]_Y}(x))\\
  \Gamma_2(A^{Y^\delta}_{[\gamma]_{Y^\delta}}(x))
  &=&
  A^{Y_{\alpha\delta}}_{[\gamma]_{Y_{\alpha\delta}}}(\Gamma_2(x))
\end{eqnarray*}
\end{prop}
\begin{proof} Recall the geometric realization of 
the $\bigwedge\,\!\!\!^*(H_1/Tor)$-module structure. Given a point 
$x\in\talpha\cap\tbeta\subset\talpha\cap\tbetaprime$ 
(cf.~the proof of Proposition~\ref{THMTHM} for the appropriate
identification), by definition
\[
  A^{Y^\delta}_{[\gamma]_{Y^\delta}}(x)
  =
  \sum_{y}\sum_{\phi\in H(x,y,1)} a(\gamma,\phi)\cdot y,
\]
where $H(x,y,1)\subset\pitwo(x,y)$ is the set of Whitney discs
with $n_z=0$ and $\mu=1$. Furthermore, 
\[
  a(\gamma,\phi)
  =
  \#\modhatphi
  \cdot
  \#(u(\{-1\}\times\R),\gamma\times\symgmo)_{\talpha}
\]
where the right factor denotes the intersection number of $u(\{-1\}\times\R)$ 
and $\gamma\times\symgmo$ inside $\talpha$.
Fixing another point $y\in\talpha\cap\tbeta$, recall that these points
are connected by $\parhat^w_{\alpha\beta}$ if and only if they are
connected by $\parhat^\delta$. Moreover, there is an identification
of the respective moduli spaces. Thus, fixing a disc $\phi$ connecting
these points (in $\alpha\betaprime$), we know -- since $n_z(\phi)=0$ --
that $\phi$ connects these intersection points in the $\alpha\beta$-diagram,
too. Denoting by $[\phi]$ its class in $\pitwo$, we see that
\[
  \#\widehat{\mathcal{M}}^{\alpha\beta}_{[\phi]}
  =
  \#\widehat{\mathcal{M}}^{\alpha\betaprime}_{[\phi]}.
\]
Moreover, the intersection number in $\talpha$ used to define 
$a(\gamma,[\phi])$ coincides for both diagrams since $\phi$ is 
a common representative.
Thus, we see that
\[
  a^{Y^\delta}(\gamma,[\phi])=a^Y(\gamma,[\phi]).
\]
Recall that there are no connections from $\talpha\cap\tbeta$-intersections
to a $\talpha\cap\tdelta$-intersection in the $\alpha,\betaprime$-diagram.
Hence, the first equality given in the proposition follows.\vspace{0.3cm}\\
To show the second, fix a point
 $x\in\talpha\cap\tdelta\subset\talpha\cap\tbetaprime$. Use the same
line of arguments as above to show that the following identity holds:
\[
\begin{array}{ccccc}
  \displaystyle{
  A^{Y^\delta}_{[\gamma]_{Y^\delta}}(x)
  }
  &=&
  \displaystyle{
  \sum_y\sum_{\phi\in H(x,y,1)}a^{Y^\delta}(\gamma,\phi)\cdot y
  }
  &+&
  \displaystyle{
  \sum_z\sum_{\psi\in H(x,z,1)}a^{Y^\delta}(\gamma,\psi)\cdot z
  }
  \\
  &=&
  \displaystyle{
  A^{Y_{\alpha\delta}}_{[\gamma]_{Y_{\alpha\delta}}}(x)
  }
  &+&
  \displaystyle{
  \sum_z\sum_{\psi\in H(x,z,1)}a^{Y^\delta}(\gamma,\psi)\cdot z.
  }
\end{array}
\]
The second sum is an element in $\cfhat(\Sigma,\alpha,\beta,z,w)$.
Recall that $\Gamma_2$ is induced by the projection onto 
$\cfhat(\Sigma,\alpha,\delta,z,w)$. Hence, the second sum cancels
when projected under the map $\Gamma_2$. The second equality of the proposition
follows.
\end{proof}
In \S\ref{naturality} we will derive suitable naturality properties
of the sequence to show that the maps involved in the sequences 
are indeed topological. We will be interested in the maps denoted
by $\Gamma_1$ since these are directly related to the surgery
represented by the Dehn twist.
\begin{rem} A result similar to Corollary~\ref{motiv} and Corollary~\ref{motiv2}
can be derived that involves the groups $\hfkminus$ using methods applied in
this paragraph.
\end{rem}

\subsection{Negative Dehn Twists}\label{negtwist}
The approach for negative Dehn twists is pretty much the
same as for positive Dehn twists. In \S\ref{postwist} we already
mentioned that the situation indicated in Figure~\ref{Fig:figone}
is not suitable for performing negative Dehn twists. Performing a 
negative twist, we could not make an a priori statement 
about what generators can be connected by holomorphic discs like 
we did in \S\ref{postwist}. To get back into business we just 
need to isotope the curve $\delta$ inside the page a bit 
(or equivalently isotope some of the attaching circles). 
Figure~\ref{Fig:figthree} indicates a possible perturbation
suitable for our purposes. Comparing Figures~\ref{Fig:figtwo}
and~\ref{Fig:figthree} we see that we isotoped the curve $\delta$
a bit. 
\begin{figure}[ht!]
\labellist\small\hair 2pt
\pinlabel {boundary of $P$} [b] at 235 328
\pinlabel {boundary of $P$} [b] at 667 328
\pinlabel {$z$} [r] at 268 250
\pinlabel {$z$} [r] at 703 250
\pinlabel {$\dom_*$} at 158 118
\pinlabel {$\dom_*$} at 586 118
\pinlabel {$\dom_{**}$} [Bl] at 237 138
\pinlabel {$\dom_{**}$} [Bl] at 671 138
\pinlabel {$w$} at 202 157
\pinlabel {$w$} at 628 157
\pinlabel {$\dom_z$} [Bl] at 319 243
\pinlabel {$\dom_z$} [Bl] at 747 243
\pinlabel {$_2$} [l] at 348 122
\pinlabel {$_1$} [B] at 375 105
\pinlabel {$_2$} [l] at 780 122
\pinlabel {$_1$} [B] at 808 105
\pinlabel {$\beta_2$} [l] at 365 312
\pinlabel {$\alpha_2$} [l] at 365 276
\pinlabel {$\beta_1$} [l] at 365 204
\pinlabel {$\alpha_1$} [l] at 365 168
\pinlabel {$\alpha_2$} [l] at 365 60
\pinlabel {$\beta_2$} [l] at 365 25
\pinlabel {$\beta_2$} [l] at 796 312
\pinlabel {$\alpha_2$} [l] at 796 276
\pinlabel {$\betaprime_1$} [l] at 796 204
\pinlabel {$\alpha_1$} [l] at 796 168
\pinlabel {$\alpha_2$} [l] at 796 60
\pinlabel {$\beta_2$} [l] at 796 25
\endlabellist
\centering
\includegraphics[height=5cm]{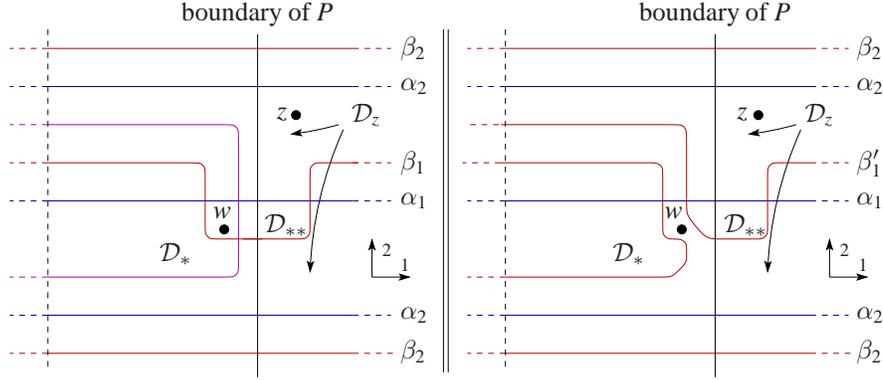}
\caption{Before and after a negative Dehn twist along $\delta$.}
\label{Fig:figthree}
\end{figure}
Observe that with this perturbation done, we again can read off the
behavior of holomorphic discs like in \S\ref{postwist} (carry over
the discussion of \S\ref{postwist} to this situation). As a consequence,
the following proposition can be proved. The proof of 
Proposition~\ref{THMTHM} carries over verbatim to a proof of
Proposition~\ref{THMTHM2}.
\begin{prop}\label{THMTHM2} Let $(\Sigma,\alpha,\beta)$ be a 
$\delta$-adapted Heegaard diagram of $Y$ and denote by $Y^\delta$ 
the manifold obtained from $Y$ by composing the gluing map, given 
by the attaching curves $\alpha$, $\beta$, with a negative 
Dehn twist along $\delta$ as hinted in Figure~\ref{Fig:figthree}. 
Then we have
\[
  \hfhat(Y^\delta)\cong 
  H_*(\cfhat(\alpha,\beta)\oplus
  \cfhat(\alpha,\delta),
  \partial^{f}),
\]
where $\partial^{f}$ is of the form
\[
  \left(\begin{matrix}\parhat_{\alpha\beta}^w & 0\\
  f&\parhat_{\alpha\delta}^w\end{matrix}\right)
\]
with $f$ a chain map 
between $(\cfhat(\alpha,\delta),\parhat^w_{\alpha\delta})$ and 
$(\cfhat(\alpha,\beta),\parhat^w_{\alpha\beta})$.\hfill $\square$
\end{prop}
\begin{cor}\label{motiv2}
Let $K\subset Y$ be the knot determined by $\delta$. Then there 
is a long exact sequence
\begin{diagram}[size=2em,labelstyle=\scriptstyle]
 \dots & \rTo^{\partial_*} &\hfkhat(Y_0(K),\mu)&&
 \rTo^{\Gamma_2}&&\hfhat(Y_{+1}(K))&&\rTo^{\Gamma_1}&&
 \hfkhat(Y,K)&\rTo^{\partial_*}&\dots
\end{diagram}
with $\partial_*=-f_*$ where $f$ is the map defined in the proof of
Proposition~\ref{THMTHM2}. The knot $\mu$
denotes a meridian of $K$. Moreover, identities
hold similar to those given in Proposition~\ref{modstructure}.\hfill$\square$
\end{cor}

\section{Invariance}\label{naturality}
Our goal in this paragraph is to show that the map $\Gamma_1$
in the sequences introduced are topological, i.e.~just depend
on the cobordism associated to the surgery represented by
the Dehn twist. To do that, we have to generalize our approach a
bit and try to see that everything we have done, especially the
proof of Proposition~\ref{THMTHM}, works without using a Heegaard
diagram that is necessarily induced by an open book.
Obviously, the geometric situation given in Figure~\ref{Fig:figfive} 
builds the foundation of the proof. To clarify the situation, look at
Figure~\ref{Fig:impsit}.
\begin{figure}[ht!]
\labellist\small\hair 2pt
\pinlabel {$z$} [l] at 221 205
\pinlabel {$\beta_1$} [t] at 285 123
\pinlabel {$\delta$} [l] at 149 44
\endlabellist
\centering
\includegraphics[height=3cm]{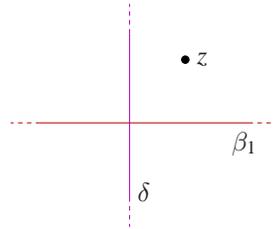}
\caption{The important geometric configuration.}
\label{Fig:impsit}
\end{figure}

We, for the moment, stick to the notation of \S\ref{introd}. We need the 
curve $\delta$ to intersect $\beta_1$ once, transversly and
to be disjoint from the other $\beta$-circles. In addition, the top right
domain at the point $\delta\cap\beta_1\in\Sigma$ has to contain the base point
$z$ (cf.~Figure~\ref{Fig:impsit}). Given this configuration, the proof of
Proposition~\ref{THMTHM} applies. The situation illustrated, does not occur 
exclusively when the Heegaard diagram is induced by an open book.
\begin{figure}[ht!]
\labellist\small\hair 2pt
\pinlabel {$K$} [B] at 18 264
\pinlabel {$K$} [B] at 355 264
\pinlabel {$K$} [B] at 686 264
\pinlabel {$z$} [l] at 247 227
\pinlabel {$z$} [l] at 584 227
\pinlabel {$z$} [l] at 917 227
\pinlabel {$\beta_1$} [tl] at 13 126
\pinlabel {$\beta_1$} [tl] at 351 126
\pinlabel {$\beta_1$} [tl] at 685 126
\pinlabel {$\alpha$} [Bl] at 229 112
\pinlabel {$\alpha$} [Bl] at 567 112
\pinlabel {$\alpha$} [Bl] at 901 112
\pinlabel {$\beta_2$} [tl] at 54 39
\pinlabel {$\beta_2$} [tl] at 392 39
\pinlabel {$\beta_2$} [tl] at 726 39
\pinlabel {$(a)$} at 178 32
\pinlabel {$(b)$} at 516 32
\pinlabel {$(c)$} at 846 32
\endlabellist
\centering
\includegraphics[width=12cm]{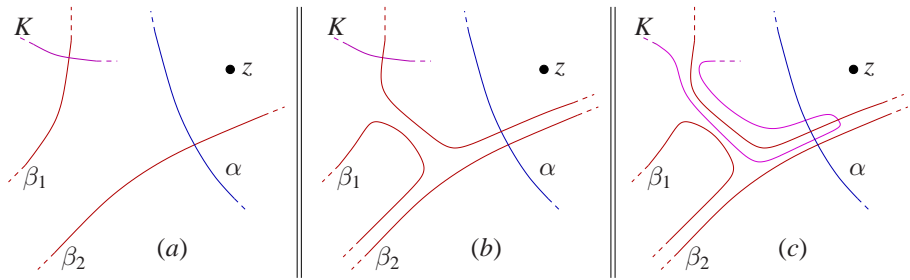}
\caption{Preparation of the Heegaard diagram.}
\label{Fig:slideprep}
\end{figure}

Given a Heegaard diagram subordinate to a knot $K$, we 
can isotope the knot $K$ onto the Heegaard surface. The isotoped
knot intersects just one $\beta$-circle once, transversely. Without loss of
generality $K$ intersects $\beta_1$. To generate a geometric configuration like 
indicated in Figure~\ref{Fig:impsit}, we 
may isotope the knot again to move the intersection $\beta_1\cap K$
to lie next to a $\dom_z$-region: 
Cutting the $\alpha$-circles out of the Heegaard surface, we obtain
a sphere with holes. The region $\dom_z$ is a region in this sphere.
Either $\dom_z$ is the whole sphere with holes or not. In case it 
is the whole sphere all the $\beta$-circles touch the region $\dom_z$ and
we are done. In case $\dom_z$ is not the whole sphere, there has to be 
at least one $\beta$-circle touching the boundary of $\dom_z$. If $\beta_1$ touches
the boundary of $\dom_z$, we are done. If $\beta_1$ does not touch the boundary of
$\dom_z$, we obtain a configuration like indicated in part $(a)$ of 
Figure~\ref{Fig:slideprep}. Without loss of generality we assume that $\beta_2$ touches
$\dom_z$. Note that it not possible for $\beta_2$ to separate $\dom_z$ from
$\beta_1$, since the complement of the $\beta$-circles in $\Sigma$ is 
connected. We are allowed to slide $\beta_1$
over this $\beta$-circle (cf.~part $(b)$ of 
Figure~\ref{Fig:slideprep}). After the 
handle slide there is a small arc $a$ 
inside $\beta_1$ touching $\dom_z$. By a small 
isotopy of the knot $K$ we can move the intersection 
point $K\cap\beta_1$ along 
the new $\beta_1$-circle until it enters 
the arc $a$ (cf.~part $(c)$ of Figure~\ref{Fig:slideprep}).\vspace{0.3cm}\\
Care has to be taken of the surgery framing. Here, we stick to
surgeries or to framed knots $K$ such that there exists a subordinate
Heegaard diagram with the framing induced by the Heegaard surface
coinciding with the framing of the knot. Evidence indicate that every
framing can be realized in this way.\vspace{0.3cm}\\
We saw that our discussion from the last paragraph can be carried over to a more
general situation. We, indeed, do not need the Heegaard diagram to be induced
by an open book. So far, we restricted the discussion to Heegaard diagrams
induced by open books, since we are interested in applications to the
contact geometric parts of the theory, which makes a discussion of this class of
diagrams inevitable.\\
Given two Heegaard diagrams subordinate to a pair $(Y,\delta)$, we transform the 
one diagram into the other by the moves introduced 
in Lemma~\ref{helplem}. These moves respect the knot complement of $\delta$. The 
goal is to show that each move preserves the exact sequence and the 
maps inherited. In the following we will call Heegaard diagrams, realizing
a geometric situation as given in Figure~\ref{Fig:figfive} for a
knot $\delta$, {\bf $\delta$-suitable}.
\begin{figure}[ht!]
\labellist\small\hair 2pt
\pinlabel {$\Theta$} [tr] at 17 34
\pinlabel {$x$} [B] at 126 235
\pinlabel {$y$} [tl] at 257 27
\pinlabel {$\talpha$} [Br] at 71 136
\pinlabel {$\tbeta$} [Bl] at 191 136
\pinlabel {$\talphaprime$} [t] at 126 26
\endlabellist
\centering
\includegraphics[width=4cm]{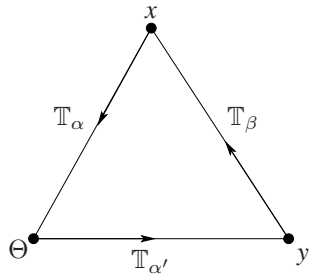}
\caption{Triangles that have to be counted for handle slides among the $\alpha$-curves.}
\label{Fig:figtriangle}
\end{figure}

We begin showing invariance under handle slides among the 
$\alpha$-curves. Although used in some papers it was never
explicitly mentioned which triangles are counted for handle slides
among the $\alpha$-curves. Given a Heegaard diagram
$(\Sigma,\alpha,\beta)$, denote by $\alpha'$ the attaching circles obtained
by a handle slide among the $\alpha$-curves. The associated map between the respective
homologies counts holomorphic triangles with boundary conditions in 
$\alpha$, $\alpha'$ and $\beta$. Figure~\ref{Fig:figtriangle} pictures a
Whitney triangle connecting a point $x\in\talpha\cap\tbeta$ with a point $y\in\talphaprime\cap\tbeta$.
Observe that in this situation $\Theta$ is a top-dimensional
generator of $\hfhat(\alpha',\alpha)$ (note the order of the
attaching circles). To not confuse the maps induced by handle slides
among the $\alpha$-circles with the maps induced by handle slides among
the $\beta$-circles, we introduce the following notation: let us denote
by $\Gamma_{\alpha,\alpha';\beta}$ the map induced by a handle slide among
the $\alpha$-circles (like indicated above) and by $\Gamma_{\alpha;\beta,\betaprime}$ the map induced
by a handle slide among the $\beta$-circles.
\begin{prop}\label{ch3invar01} Let $(\Sigma,\alpha,\beta,z)$ 
be a $\delta$-suitable Heegaard diagram and 
$(\Sigma,\alpha',\beta,z)$ be obtained by a handle slide of 
one of the $\alpha_i$. Denote by
\begin{eqnarray*}
\Gamma_{\alpha,\alpha';\beta}^w&\co&
\cfkhat(\Sigma,\alpha,\beta,z,w)\lra
\cfkhat(\Sigma,\alpha',\beta,z,w)\\
\Gamma_{\alpha,\alpha';\delta}^w&\co&
\cfkhat(\Sigma,\alpha,\delta,z,w)\lra
\cfkhat(\Sigma,\alpha',\delta,z,w)\\
\Gamma_{\alpha,\alpha';\beta'}&\co&
\cfhat(\Sigma,\alpha,\beta',z)\lra
\cfhat(\Sigma,\alpha',\beta',z)\\
\end{eqnarray*}
the induced maps. These maps induce a commutative diagram 
with exact rows
\[
\begin{diagram}[size=2em,labelstyle=\scriptstyle]
 \dots & 
 \rTo^{\partial_*} &   
 \hfkhat(\Sigma,\alpha,\beta,z,w)& 
 \rTo^{\Gamma_1} & 
 \hfhat(\Sigma,\alpha,\beta',z) & 
 \rTo^{\Gamma_2} &  
 \hfkhat(\Sigma,\alpha,\delta,z,w)& 
 \rTo^{\partial_*} &
 \dots\\
 && 
 \dTo^{\Gamma_{\alpha,\alpha';\beta}^{w,*}}
 & 
 &
 \dTo^{\Gamma_{\alpha,\alpha';\beta'}^{*}}
 & 
 &
 \dTo^{\Gamma_{\alpha,\alpha';\delta}^{w,*}}
 &&\\
 \dots & 
 \rTo^{\partial_*'} &
 \hfkhat(\Sigma,\alpha',\beta,z,w) & 
 \rTo^{\Gamma'_1} & 
 \hfhat(\Sigma,\alpha',\beta',z) & 
 \rTo^{\Gamma'_2} & 
 \hfkhat(\Sigma,\alpha',\delta,z,w) &
 \rTo^{\partial_*'} &
 \dots
\end{diagram}.
\]
\end{prop}
\begin{proof} The proof of this proposition is quite similar to 
the proof of Proposition~\ref{THMTHM}. To keep the exposition 
efficient, we do not point out all details here.
Start looking at the 
map $\Gamma_{\alpha,\alpha';\beta'}$. It is defined by counting 
triangles with boundary conditions in $\talpha$, $\talphaprime$, 
$\tbetaprime$.
\begin{figure}[ht!]
\labellist\small\hair 2pt
\pinlabel {$\dom_*$} [Br] at 72 238
\pinlabel {$c$} [r] at 107 205
\pinlabel {$a$} [B] at 55 152
\pinlabel {$w$} [l] at 65 102
\pinlabel {$d$} [r] at 108 79
\pinlabel {$\alpha_1$} [t] at 20 35
\pinlabel {$\betaprime_1$} [t] at 192 35
\pinlabel {$\dom_{**}$} at 156 98
\pinlabel {$b$} [B] at 181 153
\pinlabel {$z$} [l] at 174 205

\pinlabel {$\dom_*$} [Br] at 315 238
\pinlabel {$c$} [r] at 350 205
\pinlabel {$z$} [l] at 417 205
\pinlabel {$w$} [B] at 305 98
\pinlabel {$d$} [r] at 351 79
\pinlabel {$\dom_{**}$} at 399 98
\pinlabel {$\alpha_1$} [t] at 265 35
\pinlabel {$\delta$} [r] at 352 145

\pinlabel {$\dom_*$} [Br] at 558 238
\pinlabel {$a$} [B] at 532 152
\pinlabel {$b$} [B] at 664 153
\pinlabel {$z$} [l] at 660 205
\pinlabel {$w$} [l] at 544 102
\pinlabel {$\dom_{**}$} at 642 98
\pinlabel {$\alpha_1$} [t] at 513 35
\pinlabel {$\alphaprime_1$} [B] at 513 61
\pinlabel {$\beta_1$} [b] at 592 148
\endlabellist
\centering
\includegraphics[width=11cm]{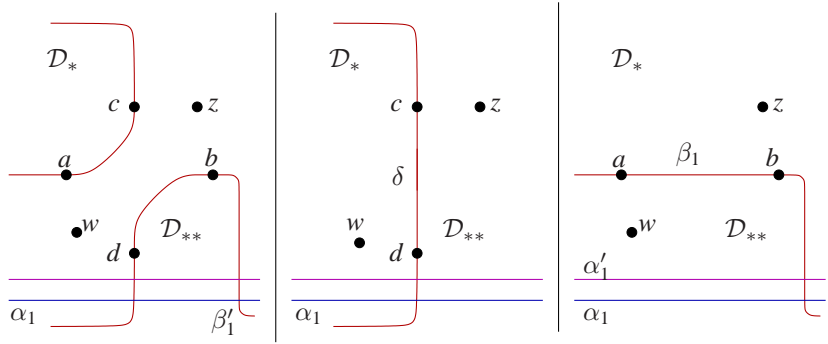}
\caption{Picture of the three different boundary conditions arising in our discussion.}
\label{Fig:figeight}
\end{figure}

Figure~\ref{Fig:figeight} illustrates the boundary conditions and how they 
look like near the region where the Dehn twist is performed. 
Analogous to the discussion in the proof of Proposition~\ref{THMTHM} 
the picture shows that 
\[
  \Gamma_{\alpha,\alpha';\beta'}
  =
  \left(
  \begin{matrix}
  \Gamma_{\alpha,\alpha';\beta}^w &
  \bargamma\\
  0&-\Gamma_{\alpha,\alpha';\delta}^w
  \end{matrix}\right),
\]
where $\bargamma$ is a map defined by counting triangles that 
connect $\talphaprime\cap\tdelta$-intersections 
with $\talpha\cap\tbeta$-intersections. This immediately shows 
commutativity of the first two boxes, i.e.
\begin{eqnarray*}
\Gamma_{\alpha,\alpha';\beta'}^*
\circ\Gamma_1&=&\Gamma'_1\circ\Gamma_{\alpha,\alpha';\beta}^{w,*}\\
\Gamma'_2\circ\Gamma_{\alpha,\alpha';\beta'}^*&=&-\Gamma_{\alpha,\alpha';\delta}^{w,*}\circ\Gamma_1.
\end{eqnarray*}
It remains to show that
\[
  \Gamma_{\alpha,\alpha';\beta}^{w,*}
  \circ
  \partial_*
  =
  \partial_*'
  \circ
  -\Gamma_{\alpha,\alpha';\delta}^{w,*}.
\]
Recall that $\partial_*$ equals the map $f$ in the definition 
of the boundary $\parhat^\delta$. These were defined by 
counting discs with $n_{*}\not=0$ or $n_{**}\not=0$. Look at 
the following expression
\[
  \Gamma_{\alpha,\alpha';\beta}^{w,*}
  \circ 
  f_*
  +
  f_*'
  \circ 
  \Gamma_{\alpha,\alpha';\delta}^{w,*}.
\]
The strategy to show its vanishing is analogous to the 
discussion of the chain map-property of $f$ in the proof 
of Proposition~\ref{THMTHM}. There are two ways to see 
this: Recall that $\Gamma_{\alpha,\alpha';\beta'}$
is a chain map. Hence, with the representation of 
$\parhat^\delta$ given in Proposition~\ref{THMTHM}, this
means that
\begin{equation}
f'\circ \Gamma_{\alpha,\alpha';\delta}^w
+\Gamma_{\alpha,\alpha';\beta}^w\circ f
=\parhat_{\alphaprime\beta}^w\circ\bargamma
+\bargamma\circ\parhat_{\alphaprime\delta}^w.\label{ch3eq01}
\end{equation}
Thus, 
\begin{eqnarray*}
0&=&
(f'\circ \Gamma_{\alpha,\alpha';\delta}^w
+\Gamma_{\alpha,\alpha';\beta}^w\circ f)_*\\
&=&f'_*\circ \Gamma_{\alpha,\alpha';\delta}^{w,*}
+\Gamma_{\alpha,\alpha';\beta}^{w,*}\circ f_*
\end{eqnarray*}
since all maps involved are chain maps. Hence, the third 
box commutes, too.
Alternatively, look at the ends of the 
moduli spaces of Whitney triangles with boundary conditions 
in $\talpha$, $\talphaprime$, $\tbetaprime$ with 
Maslov index $1$ and non-trivial intersection number 
$n_*$ or $n_{**}$. The ends look like given in Figure~\ref{Fig:figeleven}.
\begin{figure}[ht!]
\labellist\small\hair 2pt
\pinlabel {$\alphaprime$} [r] at 315 466
\pinlabel {$\alphaprime$} [r] at 305 365
\pinlabel {$\alphaprime$} [r] at 108 189
\pinlabel {$\alphaprime$} [r] at 523 189
\pinlabel {$\alphaprime$} [r] at 20 88
\pinlabel {$\alphaprime$} [r] at 315 88

\pinlabel {$\alpha$} [l] at 377 466
\pinlabel {$\alpha$} [l] at 383 365
\pinlabel {$\alpha$} [l] at 192 189
\pinlabel {$\alpha$} [l] at 602 189
\pinlabel {$\alpha$} [l] at 395 88
\pinlabel {$\alpha$} [l] at 682 105

\pinlabel {$\betaprime$} [t] at 346 299
\pinlabel {$\betaprime$} [t] at 148 127
\pinlabel {$\betaprime$} [t] at 562 127
\pinlabel {$\betaprime$} [l] at  96 62
\pinlabel {$\betaprime$} [r] at 614 62
\pinlabel {$\betaprime$} [t] at 355 28

\pinlabel {fixed point $\hattheta_+$} at 364 256

\endlabellist
\centering
\includegraphics[width=10cm]{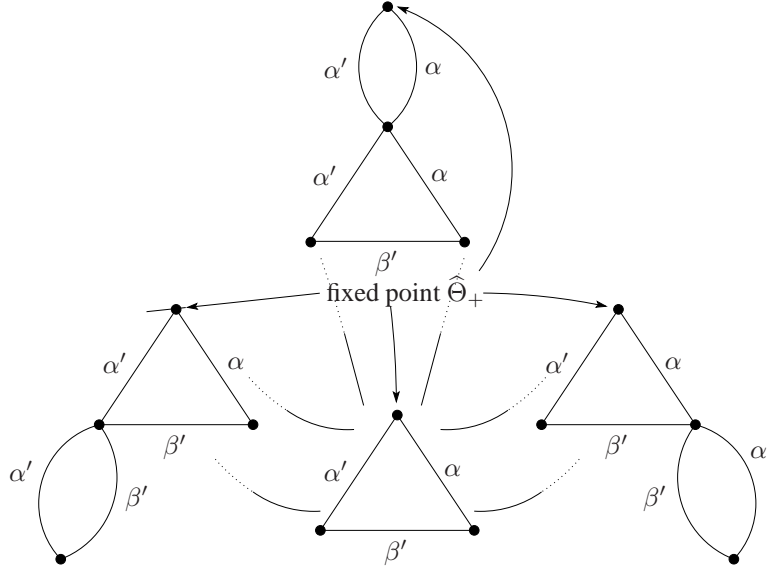}
\caption{The moduli space has three possible ends. But only
two of them count non-trivially, since $\parhat\hattheta^+=0$.}
\label{Fig:figeleven}
\end{figure}
There are three possible ends. But observe that 
the top end (cf.~Figure~\ref{Fig:figeleven}) corresponds to
 $\Gamma(x\otimes\parhat\hattheta^+)$, which vanishes since by 
definition $\parhat\hattheta^+=0$. Hence, for our situation 
there are just two 
possible types of ends to consider (the both at the bottom of 
Figure~\ref{Fig:figeleven}). Recall that breaking is the 
only phenomenon that appears here (cf.~proof of 
Proposition~\ref{THMTHM} or see \cite{OsZa01}). Proceeding as in
the proof of Proposition~\ref{THMTHM}, the commutativity of the
third box follows.
\end{proof}
\begin{prop}\label{ch3invar02} Isotopies of the 
$\alpha$-circles induce isomorphisms on the homologies such 
that all squares commute. Isotopies of the $\beta$-curves 
that miss the points $w$ and $z$ induce isomorphisms such 
that all squares commute.
\end{prop}
\begin{proof} We realize isotopies of the attaching circles by 
Hamiltonian isotopies. Hence, the induced map $\Phi$ on homology 
is defined by counting discs with dynamic boundary conditions in 
the $\alpha$-curves. The $\beta$-side remains untouched. Hence, 
by an analogous argument as in the proofs of Theorems~\ref{THMTHM} 
and \ref{ch3invar01} the map on homology splits into three 
components. The commutativity with $\Gamma_1$ and $\Gamma_2$ is 
then obviously true, and the only thing to show is the 
commutativity with the connecting homomorphism $\partial_*$ 
and $\partial_*'$. But this again can be done by counting 
appropriate ends of moduli spaces or by looking into the chain map 
equation of $\Phi$ with respect to the representation of 
$\parhat^\delta$.
\end{proof}
Consider the following situation: Let $(\Sigma,\alpha,\beta,z)$ be
a $\delta$-suitable Heegaard diagram. With the discussion in \S\ref{postwist}
we obtain a long exact sequence
\[
\begin{diagram}[size=2em,labelstyle=\scriptstyle]
 \dots & 
 \rTo^{\partial_*} &
 \hfkhat(\Sigma,\alpha,\beta,z,w) & 
 \rTo^{\Gamma_1} & 
 \hfhat(\Sigma,\alpha,\beta',z) &  
 \rTo^{\Gamma_2} & 
 \hfkhat(\Sigma,\alpha,\delta,z,w) &
 \rTo^{\partial_*} &
 \dots
\end{diagram}
\]
where we define the attaching circles 
\begin{eqnarray*}
  \betaprime&=&\{\betaprime_1,\beta_2,\dots,\beta_g\}\\
  \delta&=&\{\delta,\beta_2,\dots,\beta_g\}
\end{eqnarray*}
as it was done in \S\ref{postwist}. Define $\beta''$ by performing
a handle slide among the $\beta_i$, $i\geq2$, or by a 
handle slide of $\betaprime_1$ over $\beta_i$. Perform the same 
operation on the set of attaching circles $\beta$ to obtain $\widetilde{\beta}$.
Finally, take an isotopic push-off of $\delta$, $\delta'$ say, that
intersects $\delta$ in a cancelling pair of intersection points. Do
the same with the $\beta_i$, $i\geq2$, to get $\betaprime_i$, $i\geq 2$.
In this way we define another set of attaching circles $\delta'$ which
is given by
\[
  \deltaprime=\{\deltaprime,\betaprime_2,\dots,\betaprime_g\}.
\]
Using these data we have the following result.
\begin{prop}\label{ch3invar03} 
In the present situation, denote by
\begin{eqnarray*}
\Gamma_{\alpha;\beta,\widetilde{\beta}}^w&\co&
\cfkhat(\Sigma,\alpha,\beta,z,w)\lra
\cfkhat(\Sigma,\alpha,\widetilde{\beta},z,w)\\
\Gamma_{\alpha;\delta,\delta'}^w&\co&
\cfkhat(\Sigma,\alpha,\delta,z,w)\lra
\cfkhat(\Sigma,\alpha,\delta',z,w)\\
\Gamma_{\alpha;\beta',\beta''}&\co&
\cfhat(\Sigma,\alpha,\beta',z)\lra
\cfhat(\Sigma,\alpha,\beta'',z)\\
\end{eqnarray*}
the induced maps between the associated chain complexes. These maps induce a commutative diagram with exact rows
\[
\begin{diagram}[size=2em,labelstyle=\scriptstyle]
 \dots & 
 \rTo^{\partial_*} &
 \hfkhat(\Sigma,\alpha,\beta,z,w) & 
 \rTo^{\Gamma_1} & 
 \hfhat(\Sigma,\alpha,\beta',z) &  
 \rTo^{\Gamma_2} & 
 \hfkhat(\Sigma,\alpha,\delta,z,w) &
 \rTo^{\partial_*} &
 \dots\\
 && 
 \dTo^{\Gamma_{\alpha;\beta,\widetilde{\beta}}^{w,*}}
 & 
 &
 \dTo^{\Gamma_{\alpha;\beta',\beta''}^*}
 & 
 & 
 \dTo^{\Gamma_{\alpha;\delta,\delta'}^{w,*}} 
 &&\\
 \dots &
 \rTo^{\partial_*'} &
 \hfkhat(\Sigma,\alpha,\widetilde{\beta},z,w) & 
 \rTo^{\Gamma'_1} & 
 \hfhat(\Sigma,\alpha,\beta'',z) & 
 \rTo^{\Gamma'_2} & 
 \hfkhat(\Sigma,\alpha,\delta',z,w) &
 \rTo^{\partial_*'} &
 \dots
\end{diagram}.
\]
\end{prop}
Before going {\it in medias res}, we would like to explain 
our strategy. The idea behind all main proofs
concerning the exact sequences was to show that certain 
holomorphic discs cannot exist. Up to this
point we always used the base points $w$ and $z$ in the
sense that we tried to see what implications can be made
from the conditions $n_z=n_w=0$. In addition, keeping in
mind that holomorphic maps between manifolds of the same 
dimension are orientation preserving, we 
were able to prove everything we needed. Here, however, 
it is not so easy. First we would like to to see that the map
$\Gamma_{\alpha;\beta',\beta''}$ 
can be written as
\[
  \Gamma_{\alpha;\beta',\beta''} 
  =
  \left(
  \begin{matrix}
  \Gamma_1 &  \overline{\Gamma}  \\
  0 &
  \Gamma_2
  \end{matrix}
  \right).
\]
This means we would like to show that there are no triangles
connecting $\ab$-intersections of $\talpha\cap\tbeta$ with
$\alpha\deltaprime$-intersections of
$\talpha\cap\mathbb{T}_{\beta''}$ (cf.~Figure~\ref{Fig:fignine}). This part is very similar
to the proofs already given. We could try to continue in
the same spirit and identify moduli spaces as we did before,
but this is quite messy in this situation. The reason is that
we are counting triangles, and being forced to make an
intermediate stop at the point $\hattheta$, we are able
to {\it switch our direction} there. So, comparing the boundary
conditions given in the three triple diagrams is not very
convenient. Unfortunately we were not able to avoid these
inconveniences completely, but could minimize them.
After proving the splitting, we stick to
$\Gamma_{\alpha;\beta',\beta''}$ and show that the maps
$\Gamma_1$, $\Gamma_2$, $\overline{\Gamma}$ are chain
maps and that all boxes in the diagram commute. This is
realized by counting ends of appropriate moduli spaces of
holomorphic triangles and squares. Finally, to minimize
the messy task of comparing triangles in three diagrams, we just
stick to $\Gamma_1$ and show that this map
essentially equals $\Gamma_{\alpha;\beta,\widetilde{\beta}}^w $
on the chain level. The $5$-Lemma then ends the proof.
\begin{figure}[ht!]
\labellist\small\hair 2pt
\pinlabel {$a_1$} [tr] at 225 372
\pinlabel {$z$} [r] at 348 362
\pinlabel {$\dom_*$} at 157 290
\pinlabel {$y_1$} [B] at 129 243
\pinlabel {$w$} [r] at 106 178
\pinlabel {$a_2$} [Bl] at 379 243
\pinlabel {$y_2$} [Br] at 249 208
\pinlabel {$\dom_{**}$} at 341 176
\pinlabel {$\alpha_1$} [t] at 17 138
\pinlabel {$\betaprime_1$} [Bl] at 240 43
\pinlabel {$\beta''_1$} [Br] at 219 43
\endlabellist
\centering
\includegraphics[width=8cm]{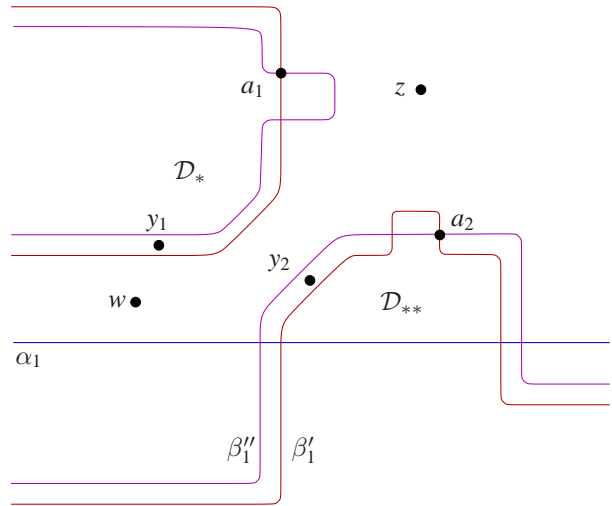}
\caption{The important part of the Heegaard diagram after handle slide.}
\label{Fig:fignine}
\end{figure}
\begin{proof} First observe that $\betaprime_1$ and $\beta''_1$
meet in two pairs of cancelling intersection points. Thus
\begin{eqnarray*}
  \Gamma_{\alpha;\beta',\beta''}
  &=&
  \fhat_{\alpha\beta'\beta''}(\,\cdot\otimes\hattheta)\\
  &=&
  \fhat_{\alpha\beta'\beta''}(\,\cdot\otimes\{a_1,\theta_2,\dots,\theta_g\})
  +
  \fhat_{\alpha\beta'\beta''}(\,\cdot\otimes\{a_2,\theta_2,\dots,\theta_g\}).
\end{eqnarray*}
So, we are looking for triangles with intermediate intersection 
$\{a_1,\theta_2,\dots,\theta_g\}$ and triangles with intermediate
intersection $\{a_2,\theta_2,\dots,\theta_g\}$.
\paragraph{Step 1 -- Splitting.}
Let $x\in\talpha\cap\tbeta$ and $y\in\talpha\cap\mathbb{T}_{\betatilde}$ be fixed.
Let
\[
  \left.\fhat_{\alpha\beta'\beta''}(x\otimes\{a_1,\theta_2,\dots,\theta_g\})\right|_y
\]
be the coefficient of $\fhat_{\alpha\beta'\beta''}(x\cdot\otimes\{a_1,\theta_2,\dots,\theta_g\})$
at the generator $y$. Suppose were is a triangle starting at $x$ and going 
to $y$ along the $\alpha$-boundary and then running to $a_1$
along its $\betaprime$-boundary. From that point we have to go 
back to $x$ again, following the red curve pictured in 
Figure~\ref{Fig:fignine}. At $a_1$ we have two choices:
we go upwards along the red curve, or we go downwards. Observe 
that going upwards, this would lead us to entering the $\dom_z$-region 
at some point and force $n_z$ to be non-zero in contradiction 
to our assumptions. Going downwards, we again enter the 
$\dom_z$-region and the boundary conditions force $n_z$ to be 
non-zero, again. Thus, there is no holomorphic triangle connecting
$x$ with $y$ along $a_1$. Thus
\[
  \left.\fhat_{\alpha\beta'\beta''}(x\otimes\{a_1,\theta_2,\dots,\theta_g\})\right|_y=0.
\]
The next step is to compute
\[
  \left.\fhat_{\alpha\beta'\beta''}(x\otimes\{a_2,\theta_2,\dots,\theta_g\})\right|_y.
\]
Suppose there were a triangle that contributes. Going along the boundary of that triangle we 
would start at $x$ and go to $y$ along the $\alpha$-boundary of the
triangle and then try to go to $a_2$ following the pink curve in Figure
\ref{Fig:fignine}. At some point we enter
$\dom_z$ forcing $n_z$ to be non-trivial. Hence, we have
\[
  \left.\fhat_{\alpha\beta'\beta''}(x\otimes\{a_2,\theta_2,\dots,\theta_g\})\right|_y=0.
\]
This shows that
\[
  \Gamma_{\alpha;\beta',\beta''} 
  =
  \left(
  \begin{matrix}
  \Gamma_1 &  \overline{\Gamma}  \\
  0 &
  \Gamma_2
  \end{matrix}
  \right).
\]
\paragraph{Step 2 -- $\Gamma_1=\Gamma_{\alpha;\beta,\widetilde{\beta}}^w$.}
First of all it is easy to see that holomorphic triangles, contributing in 
$\Gamma_{\alpha;\beta,\widetilde{\beta}}^w$, fulfill the property that $n_{y_1}=0$. Hence, together 
with $n_{w}=n_{z}=0$ the triangles have to stay away from the
regions surrounding $\beta\cap\delta$. Hence, we have
\[
  \Gamma_1=\Gamma_{\alpha;\beta,\widetilde{\beta}}^w+R.
\]
The map $R$ counts all holomorphic triangles not contributing to 
$\Gamma_{\alpha;\beta,\widetilde{\beta}}^w$.
Conversely, all holomorphic discs contributing to $\Gamma_1$ should be shown
to fulfill $n_*=n_{**}=n_{y_1}=n_{y_2}=0$. In this case $R=0$ and both maps
coincide on the chain level. Look at Figure~\ref{Fig:figten}: The situation for the $\alpha\beta\betatilde$-diagram is pictured.
\begin{enumerate}
\item Observe that there is exactly one holomorphic triangle with $n_{**}\not=0$.
This triangle contributes to $\overline{\Gamma}$.
\item There is no holomorphic triangle contributing to $\Gamma_1$ with $n_*\not=0$.
\item In a similar vein observe that these triangles in addition have trivial intersection
with $y_1$ and $y_2$.
\end{enumerate}
Thus, we see that $R=0$.
\begin{figure}[ht!]
\labellist\small\hair 2pt
\pinlabel {$z$} [r] at 257 292
\pinlabel {$y_1$} [B] at 46 215
\pinlabel {$a_2$} [Bl] at 289 176
\pinlabel {$w$} [l] at 73 111
\pinlabel {$\betatilde_1$} [l] at 357 111
\pinlabel {$\alpha_1$} [t] at 46 69
\pinlabel {$\beta_1$} [r] at 331 32
\endlabellist
\centering
\includegraphics[width=6cm]{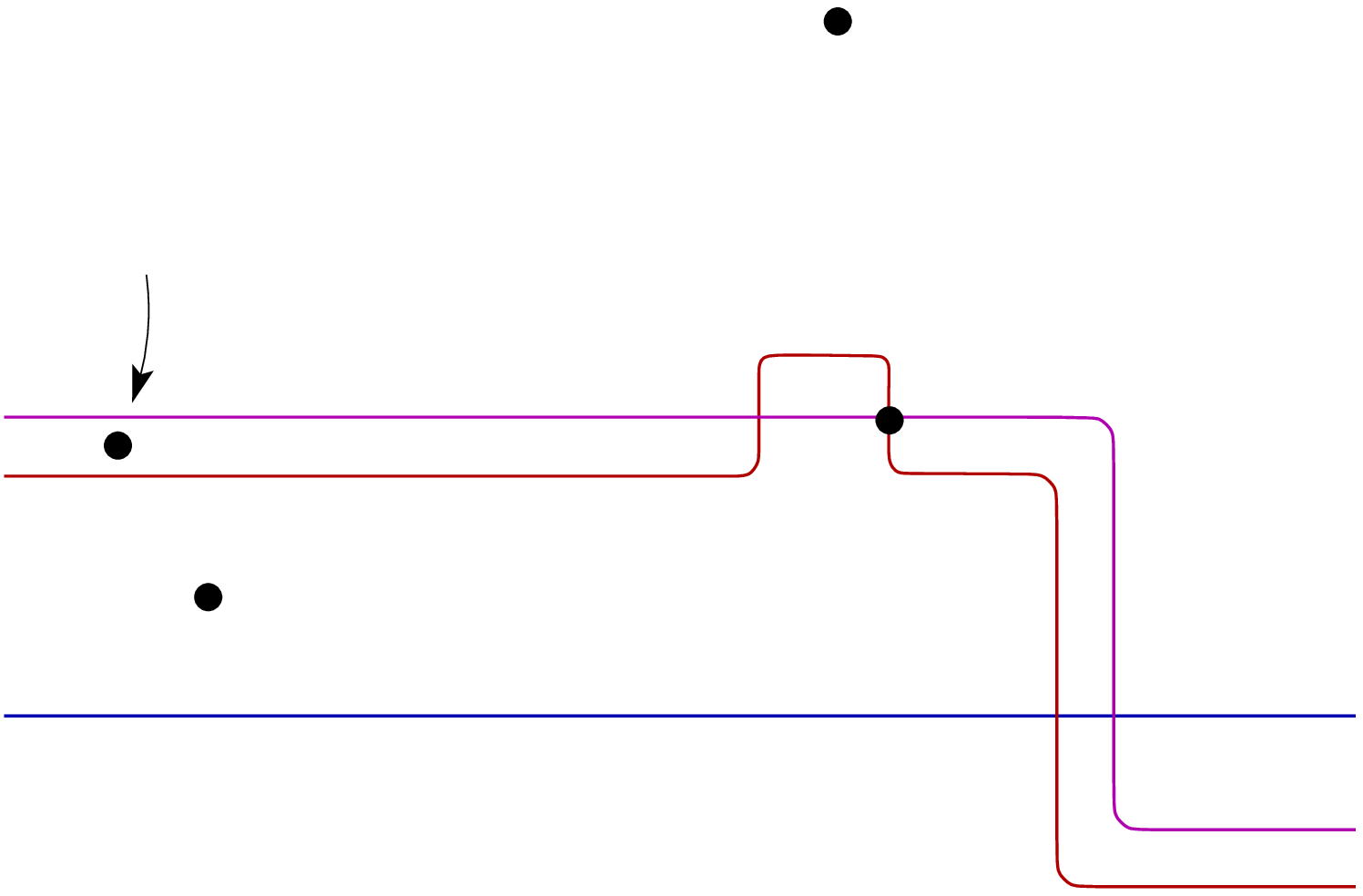}
\caption{What happens.}
\label{Fig:figten}
\end{figure}

\paragraph{Step 3 -- Chain map properties and commutativity.} Given points $x\in\talpha\cap\tdelta$ and
$y\in\talpha\cap\mathbb{T}_{\betatilde}$, look at the moduli space of holomorphic
triangles connecting $x$ with $y$, with Maslov index $1$. There are, a priori, eight ends 
from which we just write down four. The four ends missing in 
Figure~\ref{Fig:figtri} are those contributing to 
$\Gamma(\,\cdot\otimes\partial\hattheta)$, which vanishes since 
$\partial\hattheta=0$. We know that $\Gamma_{\alpha;\beta',\beta''}$ is a chain map, i.e.
\begin{eqnarray*}
  0&=&\partial\circ\Gamma_{\alpha;\beta',\beta''} 
   +\Gamma_{\alpha;\beta',\beta''} \circ\partial\\
   &=&
   \partial_{\alpha\betatilde}^w\circ\Gamma_1
   +\Gamma_1\circ\partial_{\ab}^w\\
   &&+
   \partial^w_{\alpha\betatilde}\circ\overline{\Gamma}
   +f'\circ\Gamma_2
   +\Gamma_1\circ f
   +\overline{\Gamma}\circ\partial_{\ad}^w\\
   &&+\partial^w_{\alpha\delta'}\circ\Gamma_2
   +\Gamma_2\circ\partial_{\alpha\delta}^w.
\end{eqnarray*}
The first two terms vanish since we identified $\Gamma_1$ 
with $\Gamma_{\alpha;\beta,\widetilde{\beta}}^w$, which
is a $(\partial_{\ab}^w,\partial_{\alpha\betatilde}^w)$-chain map.
The next four terms vanish since these correspond to the ends
illustrated in Figure~\ref{Fig:figtri}. Finally, since the whole
equation is zero. the last two terms cancel each other. Thus, $\Gamma_2$ is 
a chain map as desired. By construction, two of
three boxes in the diagram commute. We have to see that on
the level of homology
\[
  \Gamma_1\circ f = f'\circ\Gamma_2.
\]
Recall we showed that on the chain level
\[
  \partial^w_{\alpha\betatilde}\circ\overline{\Gamma}
   +f'\circ\Gamma_2
   +\Gamma_1\circ f
   +\overline{\Gamma}\circ\partial_{\ad}^w=0.
\]
Hence, $\overline{\Gamma}$ is a chain homotopy between
$\Gamma_1\circ f$ and $f'\circ\Gamma_2$.
\end{proof}
\begin{figure}[ht!]
\labellist\small\hair 2pt
\pinlabel {$\alpha\delta$} [B] at 89 374
\pinlabel {$\alpha\delta$} [B] at 278 374
\pinlabel {$\alpha\delta$} [l] at 97 256
\pinlabel {$\alpha\beta$} [l] at 287 256
\pinlabel {$\alpha\delta$} [B] at 468 264
\pinlabel {$\alpha\delta$} [B] at 657 264
\pinlabel {$(1)$} [t] at 89 140
\pinlabel {$(2)$} [t] at 278 140
\pinlabel {$(3)$} [t] at 468 140
\pinlabel {$(4)$} [t] at 657 140
\pinlabel {$\alpha\betatilde$} [t] at 27 140
\pinlabel {$\alpha\betatilde$} [t] at 216 140
\pinlabel {$\alpha\deltaprime$} [Br] at 415 158
\pinlabel {$\alpha\betatilde$} [Br] at 603 158
\pinlabel {$\alpha\betatilde$} [t] at 404 34
\pinlabel {$\alpha\betatilde$} [t] at 593 34
\endlabellist
\centering
\includegraphics[height=5cm]{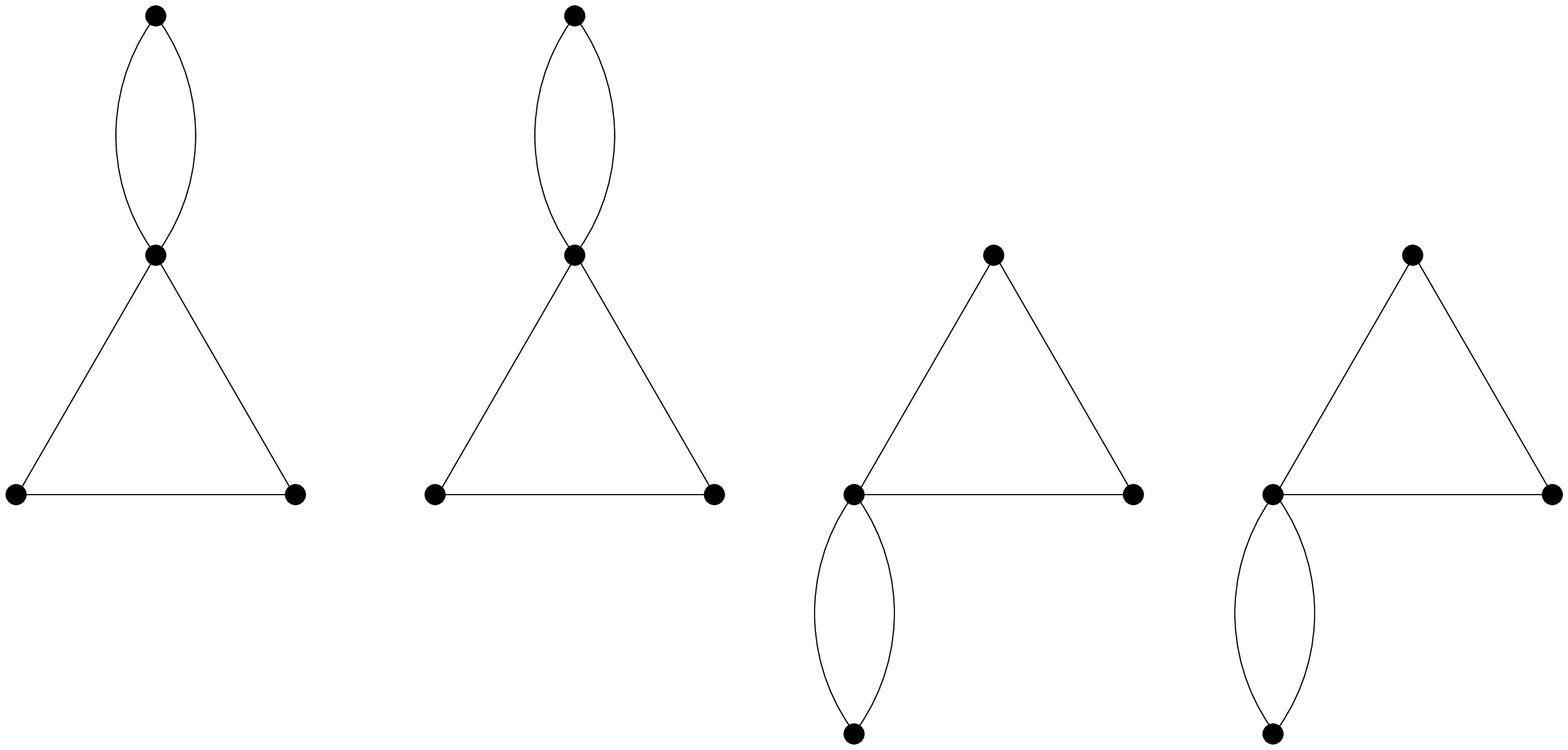}
\caption{The ends of the moduli space providing commutativity}
\label{Fig:figtri}
\end{figure}
%
%
In \cite{LOSS} the authors give an alternative proof for 
the independence of the contact element of the choice 
of cut system. We are especially interested in the technique
they used to prove Proposition 3.3 of \cite{LOSS}.
\begin{figure}[ht!]
\labellist\small\hair 2pt
\pinlabel {$\alpha_1$} at 268 254
\pinlabel {$\beta_1$} at 268 19
\endlabellist
\centering
\includegraphics[height=4cm]{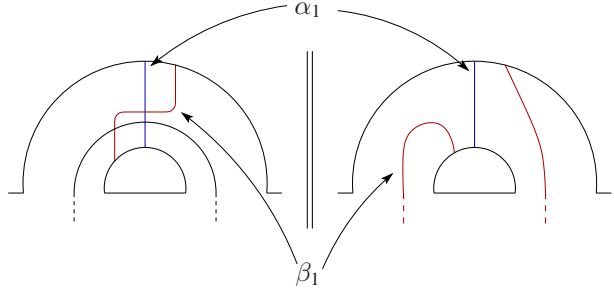}
\caption{Illustration of what happens while Giroux stabilizing.}
\label{Fig:figstabil}
\end{figure}
Recall, that given an open book $(P,\phi)$, a 
{\bf positive Giroux stabilization } of $(P,\phi)$ is 
the open book
$(P\cup h^1,\phi\circ D_\gamma^+)$ where $\gamma$ is
a closed curve in $P\cup h^1$ that intersects the
co-core of $h^1$ once, transversely. Fixing a homologically
essential, simple closed curve $\delta$ in $P$ we call the
Giroux stabilization {\bf $\delta$-elementary} if, after a
suitable isotopy, $\delta$ intersects $\gamma$ transversely in
at most one point (cf.~Definition~2.5. of \cite{LOSS}).
Their invariance proof relies on the fact that, given a 
positive Giroux stabilization, one can choose a 
cut system $a_1,\dots,a_n$
of $(P,\phi)$ such that the curve $\gamma$ does not intersect
any of the $a_i$. Observe that, given such a cut system for
$(P,\phi)$ and defining $a_{n+1}$ to be the co-core of the
handle $h^1$, then $a_1,\dots,a_{n+1}$ is a cut system for
the Giroux stabilized open book. Furthermore, observe that for $i\leq n$
\[
  \phi\circ D_\gamma^+(a_i)=\phi(a_i).
\]
Figure~\ref{Fig:figstabil} illustrates how $\phi\circ D_\gamma^+(\alpha_{n+1})$ looks like. Thus, all intersections
between $\alpha_i$ and $\beta_j$ for $i,j\leq n$ remain
unchanged, where $\alpha_{n+1}$ intersects only $\beta_{n+1}$
once, transversely. Furthermore, $D_\gamma^+(a_{n+1})$ is
disjoint from all $a_i$, $i\leq n$. And, hence, $\beta_{n+1}$ 
is disjoint from all $\alpha_i$, $i\leq n$. 
Thus, the induced Heegaard diagram looks like a stabilized
Heegaard diagram induced by the open book $(P,\phi)$ with
cut system $a_1,\dots,a_n$. Denote by $q$ the unique
intersection point of $\alpha_{n+1}$ and $\beta_{n+1}$. Then 
the map
\[
  \Phi
  \co
  \cfhat(P,\phi,\{a_1,\dots,a_n\})
  \lra
  \cfhat(P\cup h^1,D_\gamma^+\circ\phi,\{a_1,\dots,a_{n+1}\}),
\]
given by sending a generator $x$ of $\cfhat(P,\phi,\{a_1,\dots,a_n\})$
to $\Phi(x)=(x,q)$, is clearly an isomorphism of chain complexes preserving
the contact element. 
\begin{figure}[ht!]
\labellist\small\hair 2pt
\pinlabel {$\alpha_1$} at 268 254
\pinlabel {$\beta_1$} at 268 19
\endlabellist
\centering
\includegraphics[height=4cm]{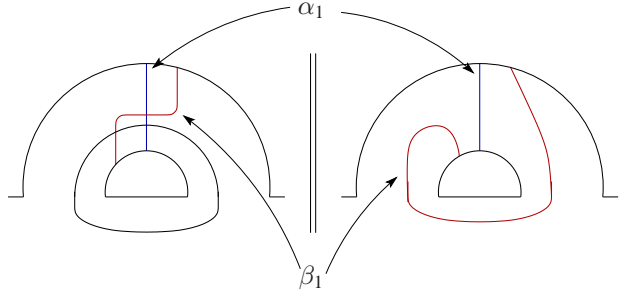}
\caption{The choice of $\gamma$ for a topological stabilization.}
\label{Fig:figstabil02}
\end{figure}

We will, however, focus our attention on a special version
of positive Giroux stabilization. Recall, that we call $(\Sigma\#T^2,\alpha',\beta')$ a 
stabilization of the Heegaard diagram $(\Sigma,\alpha,\beta)$ where we 
define $\alpha'=\alpha\cup\{\mu\}$ and $\beta'=\beta\cup\{\lambda\}$ with
$\mu$ a meridian and $\lambda$ a longitude of $T^2$. 
\begin{definition}\label{pGsrts} Let $(P,\phi)$ be an open book decomposition and let
$(P',\phi\circ D_\gamma^+)$ be a positive Giroux stabilization. We say
that the Giroux stabilization {\bf represents a topological stabilization}
if there is a cut system $\{a_1,\dots,a_n,a_{n+1}\}$ of $P'$ with the following
properties: 
\begin{enumerate}
\item[(1)] The set $\{a_1,\dots,a_n\}$ is a cut system for $P$.
\item[(2)] Denote by $(\Sigma,\alpha,\beta)$ the Heegaard diagram induced
by $(P,\phi,\{a_1,\dots,a_n\})$ and let $(\Sigma',\alpha',\beta')$ be the Heegaard
diagram induced by $(P',\phi\circ D_\gamma^+,\{a_1,\dots,a_{n+1}\})$. The diagram $(\Sigma',\alpha',\beta')$
is a stabilization of $(\Sigma,\alpha,\beta)$ up to isotopy of the attaching circles.
\end{enumerate}
\end{definition}
Look into 
Figure~\ref{Fig:figstabil02}. In this picture we present how to 
choose $\gamma$ such that the positive Giroux stabilization represents a 
topological stabilization. Indeed, the following lemma holds.
\begin{lem} Let $(P,\phi)$ be an open book decomposition and let $(P',\phi\circ D_\gamma^+)$
be a positive Giroux stabilization. The Giroux stabilization represents a topological stabilization
up to isotopy of the attaching circles if and only if $\gamma$ is isotopic to the black
curve pictured in Figure \ref{Fig:figstabil02}.
\end{lem} 
\begin{proof} Given an open book decomposition $(P,\phi)$ and a positive Giroux stabilization
$(P',\phi\circ D_\gamma^+)$ with $\gamma$ like indicated in Figure \ref{Fig:figstabil02}, this
stabilization clearly represents a topological stabilization up to isotopy: Recall that
$P'=P\cup h^1$. Choose a cut system $\{a_1,\dots,a_n\}$ of $P$ such that 
$\partial a_i$, $i=1,\dots,n$, is disjoint from the region where the handle $h^1$ is 
attached on. Define $a_{n+1}$ as the co-core of the handle $h^1$. Picturing the resulting
Heegaard diagrams we see that the positive Giroux stabilization represents a topological
stabilization up to isotopy.\\
Conversely, suppose we are given a Giroux stabilization representing a topolgical stabilization
up to isotopy, then we have to show that $\gamma$ is isotopic to the black curve, $\gamma_s$ say, 
indicated in Figure \ref{Fig:figstabil02}. First note that the handle is attached on one boundary
component of $P$. If $h^1$ connects two different boundary components of $P$, the genus of the 
resulting Heegaard surface would increase by $2$. By assumption there is a cut system 
$\{a_1,\dots,a_{n+1}\}$ for $P'$ fulfilling properties $(1)$ and $(2)$, given in 
Definition \ref{pGsrts}. As in Definition \ref{pGsrts}, denote by $(\Sigma,\alpha,\beta)$ and
$(\Sigma',\alpha',\beta')$ the respective Heegaard diagrams. By assumption, $\Sigma'=\Sigma\# T^2$
and, after applying suitable isotopies, $\alpha_i=\alpha'_i$ and $\beta_i=\betaprime_i$ for all
$i=1,\dots,n$. We have, that
\begin{eqnarray*}
  \alpha'_{n+1}&=& a_{n+1}\cup\overline{a_{n+1}}\\
  \beta'_{n+1}&\sim& a_{n+1}\cup\overline{\phi\circ D_\gamma^+(a_{n+1})}
\end{eqnarray*}
with 
\begin{eqnarray}
  \alpha'_{n+1}&\sim&\mu_{T^2}\label{alphaass}\\
  \beta'_{n+1}&\sim&\lambda_{T^2}\label{betaass}.
\end{eqnarray}
By (\ref{alphaass}), we see that $a_{n+1}$ is isotopic to the co-core of $h^1$. This can be
read off from Figure \ref{Fig:again}.
\begin{figure}[ht!]
\labellist\small\hair 2pt
\pinlabel {$\Sigma'$} at 125 488
\pinlabel {$P'$} at 667 488
\pinlabel {$h^1$} [r] at 65 112
\pinlabel {$a_{n+1}$} at 45 35
\pinlabel {$\mu_{T^2}$} [l] at 309 34
\pinlabel {$h^1$} [l] at 801 178
\pinlabel {$a_{n+1}$} [l] at 801 59
\endlabellist
\centering
\includegraphics[height=6cm]{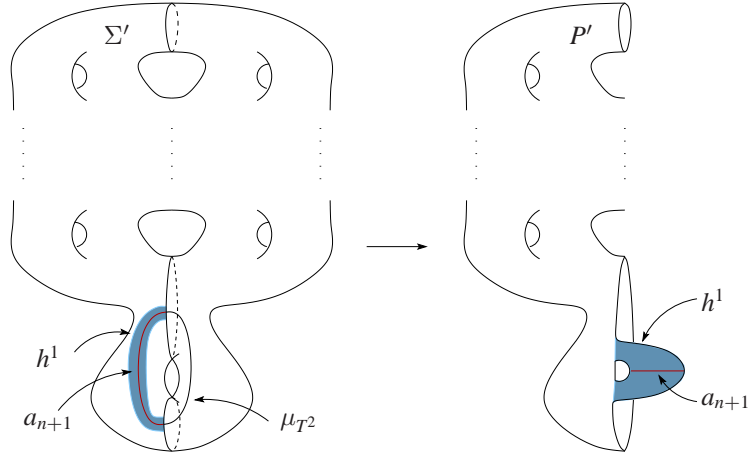}
\caption{The left portion pictures $\Sigma'$ and the right portion the page $P'$ and how
it is obtained from $P$.}
\label{Fig:again}
\end{figure}
Hence, we have
\[
  a_{n+1}\cup\overline{\phi\circ D_{\gamma}^+(a_{n+1})}
  =
  \beta_{n+1}
  \sim
  \lambda_{T^2}
  \sim
  a_{n+1}\cup\overline{\phi\circ D_{\gamma_s}^+(a_{n+1})}.
\]
So, $\phi\circ D_{\gamma}^+(a_{n+1})$ is isotopic to $\phi\circ D_{\gamma_s}^+(a_{n+1})$,
which is equivalent to saying that $D_{\gamma}(a_{n+1})$ is isotopic to $D_{\gamma_s}(a_{n+1})$. 
But this finally implies that $\gamma$ is isotopic to $\gamma_s$.
\end{proof}
%
%
\begin{prop}\label{ch3invar04} Let $(P,\phi)$ be an open book 
decomposition of $Y$ and $(P',\phi\circ D_\gamma^+)$ a positive
$\delta$-elementary Giroux stabilization representing a topological
stabilization (cf.~Definition~\ref{pGsrts} and look at Figure~\ref{Fig:figstabil02}). Then 
there are isomorphisms $\phi_1$, $\phi_2$ and $\phi_3$ on homology 
such that the following diagram commutes
\[
\begin{diagram}[size=2em,labelstyle=\scriptstyle]
 \dots & 
 \rTo^{\partial_*} &
 \hfkhat(P,\phi,\delta) & 
 \rTo^{\Gamma_1} & 
 \hfhat(P,D_\delta^+\circ\phi) & 
 \rTo^{\Gamma_2} & 
 \hfkhat(P,\widetilde{\phi}) &
 \rTo^{\partial_*} &
 \dots
 \\
 && 
 \dTo^{\phi_1}_\cong
 &
 &
 \dTo^{\phi_2}_\cong
 &
 &
 \dTo^{\phi_3}_\cong
 &&\\
 \dots & 
 \rTo^{\partial_*'} &
 \hfkhat(P',\phi\circ D_\gamma^+,\delta) & 
 \rTo^{\Gamma'_1} & 
 \hfhat(P',D_\delta^+\circ\phi\circ D_\gamma^+,z) &
 \rTo^{\Gamma'_2} &
 \hfkhat(P',\widetilde{\phi}\circ D_\gamma^+) & 
 \rTo^{\partial_*'} &
 \dots
\end{diagram}.
\]
\end{prop}
\begin{rem} General positive Giroux stabilizations do not preserve
the exact sequence. The reason is that in the general situation $\gamma\cap P$
and $\phi^{-1}(\delta)$ might intersect and cannot be separated.
In the topological situation, however, the special choice of $\gamma$
makes it possible to separate $\gamma\cap P$ from $\phi^{-1}(\delta)$.
\end{rem}

\begin{proof}
Denote by $\gamma_1$ the part of $\gamma$ that runs through 
$P$. Since we are just doing a topological stabilization, we
can attach the handle $h^1$ in such a way that $\gamma_1$
and $\phi^{-1}(\delta)$ are disjoint. Just choose 
$\gamma$ like indicated in Figure~\ref{Fig:figstabil02}. 
Even if $\phi^{-1}(\delta)$ intersects $\gamma_1$, we can 
separate them with help of a small isotopy.
By choosing a cut system $\{a_1,\ldots,a_n\}$ for $(P,\phi)$ 
appropriately, we can extend this cut system to a cut system 
for the stabilized open book by choosing $a_{n+1}$ like indicated
in Figure~\ref{Fig:figstabil02}. For all Heegaard diagrams
in the following, we will use this cut system.
Since $\phi^{-1}(\delta)$ and $\gamma$ are disjoint, the 
associated Heegaard
diagram of $(P',D_\delta^+\circ\phi\circ D_\gamma^+)$ will
look like a stabilization of the Heegaard diagram induced by
the open book $(P,D_\delta^+\circ\phi)$. The same holds for $(-P',\widetilde{\phi})$
and $(-P',\widetilde{\phi}\circ D_\gamma^+)$. Using the isomorphism
induced by stabilizations as discussed above we can define $\phi_1$,
$\phi_2$ and $\phi_3$ as indicated in Proposition~\ref{ch3invar04}.
These maps are all isomorphisms and obviously commute on the chain
level. 
\end{proof}
\begin{theorem}\label{completeinvar} 
The map $\Gamma_1$ is topological, i.e.~it just
depends on the cobordism induced by the surgery.
\end{theorem}
\begin{proof} The cobordism induced by the Dehn twist
depends only on the $3$-manifold $Y$ and the framed knot 
type $K$ which the curve $\delta$, together with its page framing,
represents inside $Y$. This pair, on the other hand, is described
by an open book decomposition adapted to $\delta$ and a
$\delta$-adapted cut system. These data determine a Heegaard 
diagram subordinate to the pair $(Y,K)$ (cf.~\S\ref{knotfloerhomology}). 
Given another adapted open book together with an 
adapted cut system, the associated Heegaard diagram is 
equivalent to the first after a sequence of moves which are 
described in Lemma~\ref{helplem}. All of these moves are
recovered via 
 Proposition~\ref{ch3invar01},
 Proposition~\ref{ch3invar02},
 Proposition~\ref{ch3invar03}
 and Proposition~\ref{ch3invar04}. Of course, after some
point, we might leave the class of Heegaard diagrams
induced by open books. But the propositions cited do not
use this open book structure as discussed at the beginning
of the section.
\end{proof}

\section{Implications to Contact Geometry}\label{contsetup}
In this section we will focus our attention on contact manifolds
$(Y,\xi)$.
Let $(P,\phi)$ be an open book decomposition that is adapted
to the contact structure $\xi$ (cf.~\S\ref{prelim:01:2}).
Recall that the contact element and the invariant defined in
\cite{LOSS} sit in the Heegaard Floer cohomology 
(cf.~\S\ref{prelim:01:2}). Because of the 
well-known equivalence
\[
  \hfhat\,\!^*(Y)=\hfhat_*(-Y)
\]
we will be interested in the behavior of $-Y$ rather than $Y$.
Recall from \S\ref{prelim:01:2} that we have two choices to
extract the Heegaard Floer homology of $-Y$ from data
given by a Heegaard diagram of $Y$. We can either switch
the orientation of the Heegaard surface or switch the
boundary conditions.\vspace{0.3cm}\\
Let $L\subset Y$ be a Legendrian knot and denote 
by $Y_L^+$ the manifold obtained by doing a
$(+1)$-contact surgery along $L$. There is an open book 
decomposition $(P,\phi)$ adapted to $\xi$ such that $L$ sits
on the page $P\times\{1/2\}$ of the open book and the page
framing coincides with the contact framing. A $(+1)$-contact
surgery acts on the open book like a negative Dehn twist
along $L$, i.e.~$(P,\phi\circ D_L^{-,P})$ is an adapted
open book decomposition of $(Y_L^+,\xi_L^+)$ where $D_L^{-,P}$
denotes a negative Dehn twist along $L$ with respect to the
orientation of $P$. Observe that $L$ sits on the wrong
page for our construction of the exact sequence.
Fortunately, the identity
\begin{equation}
  \phi\circ D_L^{-,P}=D_{\phi(L)}^{-,P}\circ\phi \label{monIdent}
\end{equation}
holds. Thus, a surgery along $L$ can be interpreted
as a left-hand composition of the monodromy with a Dehn twist. In 
addition $(P,D_{\phi(L)}^{-,P}\circ\phi)$ is an adapted
open book decomposition of $(Y_L^+,\xi_L^+)$. To see the effect 
on the Heegaard Floer cohomology, we
have to change the surface orientation. We see that
\begin{equation}
  -Y_L^+=(-P,D_{\phi(L)}^{-,P}\circ\phi)=(-P,D_{\phi(L)}^{+,-P}\circ\phi).
  \label{csclearer}
\end{equation}
One very important ingredient for our construction is the fact 
that we may choose an $L$-adapted Heegaard diagram where $L$ 
sits on $P\times\{1/2\}$. Because of the identity 
$(\ref{monIdent})$ we need a Heegaard diagram with attaching
circles adapted to $\phi(L)$ in the following sense: the curve $\phi(L)$
intersects $\beta_1$ once, transversely and is disjoint from all
other $\beta$-circles. This condition is satisfied for $L$-adapted
Heegaard diagrams since $\phi(a_i)=b_i$. This means we are able to simultaneously match 
all conditions for setting up the exact sequence and seeing the
invariant $\loss(L)$. Recall that the sequence requires the point
$w$ defining $L$ to be in a specific domain of the Heegaard diagram.
This positioning of $w$ induces an orientation on $L$. On the other hand, a fixed 
orientation of $L$ determines where $w$ has to be placed. These two orientations, the one
coming from the sequence and the one from the knot $L$ itself, have to be observed carefully. We
have to see whether every possible choice of orientation of $L$ 
induces a positioning of $w$ inside the Heegaard diagram that is 
compatible with the requirements coming from the exact sequence.
\begin{theorem}\label{maps} Let $(Y,\xi)$ be a contact manifold and
$L\subset Y$ an oriented Legendrian knot.
\begin{itemize}
\item[(i)] 
Let $W$ be the cobordism induced by $(+1)$-contact surgery 
along $L$. Then the cobordism $-W$ induces a map 
\[
  \Gamma_{-W}\co\hfkhat(-Y,L)\lra\hfhat(-Y_L^+),
\]
such that $\Gamma_{-W}(\loss(L))=c(Y_L^+,\xi_L^+)$.
\item[(ii)] If $L$ carries a specific orientation and $W$ denotes
the cobordism induced by a $(-1)$-contact surgery along $L$. Then
the cobordism $-W$ induces a map
\[
  \Gamma_{-W}\co\hfhat(-Y_L^-)\lra\hfkhat(-Y,L)
\]
such that $\Gamma_{-W}(c(Y_L^-,\xi_L^-))=0$.
\end{itemize}
\end{theorem}
\begin{proof}
Recall that
\begin{eqnarray*}
-Y_L^+&=&(-P,D_{\phi(L)}^{+,-P}\circ\phi)\\
-Y_L^-&=&(-P,D_{\phi(L)}^{-,-P}\circ\phi).\\
\end{eqnarray*}
We choose a cut system which is $L$-adapted. This means that 
$L$ intersects $\alpha_1$ transversely, in a single point and is 
disjoint from the other $\alpha$-circles. Hence, $\phi(L)$ (sitting on the other side of
the Heegaard surface) intersects $\beta_1$ in a single point
and is disjoint from the other $\beta$-circles. We first try to
prove the results concerning the $(+1)$-contact surgery.
After possibly isotoping the knot $L$ slightly, we can achieve 
a neighborhood of $\phi(L)\cap\beta_1$ to look like the left 
or right part of Figure~\ref{Fig:contactpo}.
\begin{figure}[ht!]
\labellist\small\hair 2pt
\pinlabel {$w$} [l] at 780 317
\pinlabel {$L$} [B] at 38 205
\pinlabel {$L$} [B] at 602 205
\pinlabel {$z$} [r] at 225 231
\pinlabel {$z$} [r] at 789 231
\pinlabel {$\beta_1$} [r] at 300 200
\pinlabel {$\beta_1$} [r] at 862 200
\pinlabel {$\alpha_1$} [tl] at 467 147
\pinlabel {$\alpha_1$} [tl] at 1030 147
\pinlabel {$w$} [r] at 68 42
\pinlabel {binding of the open book} [l] at 168 41
\pinlabel {binding of the open book} [l] at 729 41
\pinlabel {$_2$} [l] at 434 120
\pinlabel {$_1$} [B] at 383 84
\pinlabel {$_2$} [l] at 999 120
\pinlabel {$_1$} [B] at 946 84
\endlabellist
\centering
\includegraphics[height=4cm]{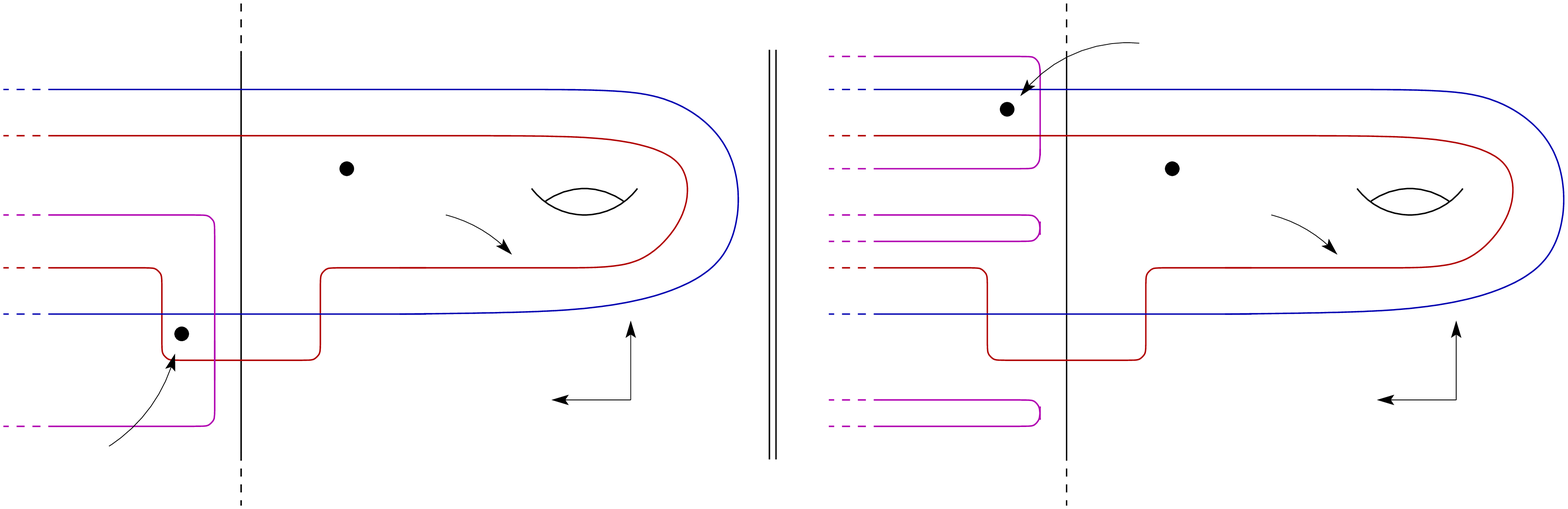}
\caption{Setting things up for a contact $(+1)$-surgery.}
\label{Fig:contactpo}
\end{figure}

In each part of the picture the knot $L$ and the point $w$ are placed
in such a way that the Dehn twist associated to the $(+1)$-contact
surgery connects the regions where the points $w$ and $z$ lie.
Thus, each picture shows a situation in which we may
apply the proof technique used for Proposition~\ref{THMTHM} 
(resp.~Proposition~\ref{THMTHM2}). Observe
that Figure~\ref{Fig:contactpo} shows the situation for each 
orientation of $L$.
Since we are doing a $(+1)$-contact surgery, we perform a
positive Dehn twist along $L$ with respect to the surface
orientation given in Figure~\ref{Fig:contactpo} (cf.~Equality (\ref{csclearer}) and cf.~discussion at
the beginning of this section). Thus, we
are able to define a map
\[
  \Gamma^+\co\hfkhat(-Y,L)\lra\hfhat(-Y_L^+).
\]
The situations in both pictures are designed to apply the proof
technique of Proposition~\ref{THMTHM}. The induced pair 
$(w,z)$ determines an orientation on $L$. To match the induced 
orientation with the one of the knot $L$ we either use the left 
or the right picture of Figure~\ref{Fig:contactpo}. By 
definition of $\Gamma^+$ we see that 
\[
  \Gamma^+(\loss(L))=c(Y_L^+,\xi_L^+).
\]
To cover $(-1)$-contact surgeries, look at Figure~\ref{Fig:contactmo}.
\begin{figure}[ht!]
\labellist\small\hair 2pt
\pinlabel {$L$} [B] at 38 205
\pinlabel {$z$} [r] at 225 231
\pinlabel {$z$} [r] at 789 231
\pinlabel {$\beta_1$} [r] at 307 200
\pinlabel {$\betaprime_1$} [r] at 873 200
\pinlabel {$\alpha_1$} [tl] at 467 147
\pinlabel {$\alpha_1$} [tl] at 1030 147
\pinlabel {binding of the open book} [l] at 163 31
\pinlabel {binding of the open book} [l] at 723 31
\pinlabel {$_2$} [l] at 434 105
\pinlabel {$_1$} [B] at 383 69
\pinlabel {$_2$} [l] at 999 105
\pinlabel {$_1$} [B] at 946 69
\pinlabel {$w$} [l] at 293 72
\pinlabel {$w$} [l] at 867 72
\pinlabel {Legendrian invariant} at 468 337
\pinlabel {contact element} [B] at 837 307
\endlabellist
\centering
\includegraphics[height=4cm]{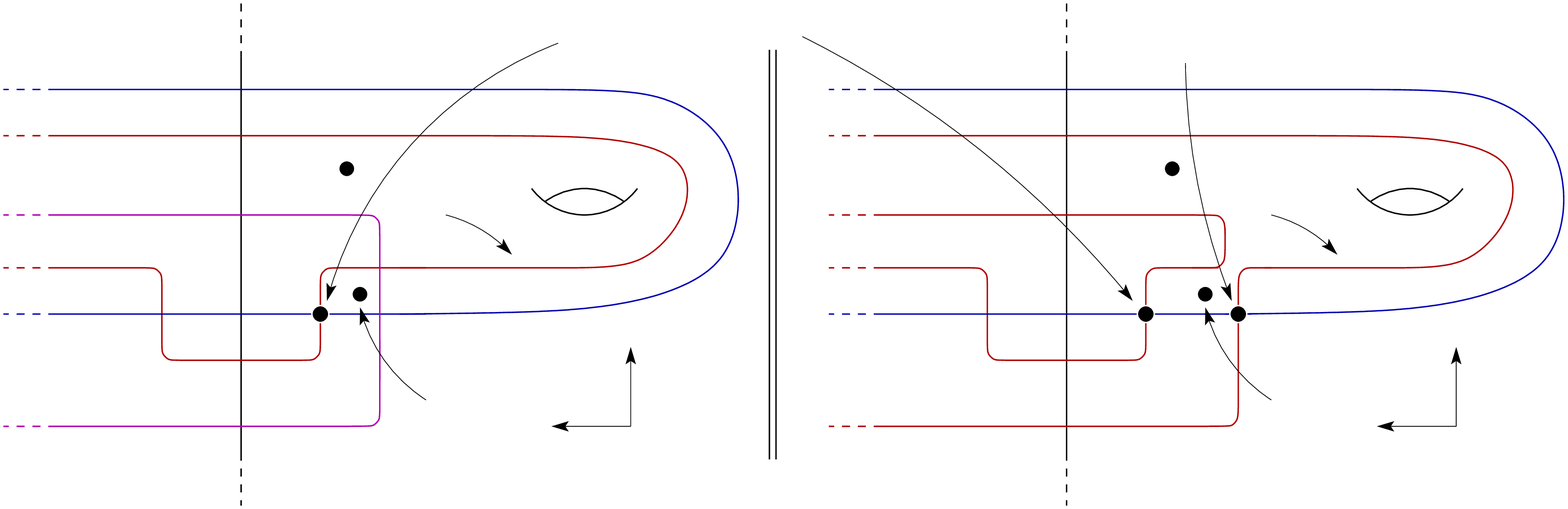}
\caption{Setting things up for a contact $(-1)$-surgery.}
\label{Fig:contactmo}
\end{figure}

The same line of arguments as above applies to define a map
\[
  \Gamma^-\co\hfhat(-Y_L^-)\lra\hfkhat(-Y,L).
\]
Again, recall that $w$ is placed in the Heegaard diagram in such 
a way that allows us to define the map $\Gamma^-$.
The pair $(w,z)$ induces an orientation on $L$. The opposite orientation
will be denoted by $\ob$.
What can be seen immediately from the picture is that the
Dehn twist separates the contact element and the invariant $\loss(L,\overline{\ob})$: 
The arguments show that we have the following exact sequence.
\begin{diagram}[size=2em,labelstyle=\scriptstyle]
 0 & \rTo &\cfkhat(Y_0(L),\mu)&&
 \rTo&&\cfhat(-Y_L^-)&&\rTo^{\Gamma^-}&&
 \cfkhat(-Y,(L,\overline{\ob}))&\rTo&0 \\
 &  &  \mbox{\textbullet} & &\rMto & &c & &     & &          & &\\
 &  &    & &      & & \mbox{\textbullet} & &\rMto & & \loss(L,\overline{\ob}) & &
\end{diagram}
To speak in the language of the proof of Proposition
\ref{THMTHM}: the element $c$ is an $\alpha\beta$-intersection, whereas 
the element $\loss(L,\overline{\ob})$ is an 
$\alpha\delta$-intersection. By exactness, the contact element 
$c$ lies in the kernel of $\Gamma^-$.
\end{proof}
\begin{definition} The orientation $\ob(P,\phi)$ from the last
proof is called the {\bf open book orientation}.
\end{definition}
To prove Corollary \ref{really} we have to recall that 
Honda, Kazez and Mati\'{c} introduced in \cite{HKM2} an invariant $EH(L)$ of 
a Legendrian knot $L$ in the Sutured Floer homology (cf.~\cite{AJU}) of a contact 
manifold with boundary. To be more precise, given $L\subset(Y,\xi)$, they
define an Legendrian isotopy invariant of $L$, called $EH(L)$, sitting in
$\sfh(-Y\backslash\nu L,\Gamma)$ where $\Gamma$ are suitably chosen sutures.
Furthermore, Stipsicz and Vertesi have shown in \cite{StipVert} that this invariant is equipped with a morphism
$\sfh(-Y\backslash\nu L,\Gamma)\lra\hfkhat(-Y,L)$ that maps $EH(L)$ to $\loss(L)$. Composing this morphism with the one coming from Theorem~\ref{maps} we get the following result.
\begin{cor}\label{really} There is a map
\[
  \gamma\co\sfh(-Y\backslash\nu L,\Gamma)\lra\hfhat(-Y_L^+)
\]
such that $\gamma(EH(L))=c(Y_L^+,\xi_L^+)$.\hfill $\square$
\end{cor}
\begin{cor} Let $L$ be a Legendrian knot in a contact manifold $(Y,\xi)$. Then $EH(L)=0$ implies 
that $c(Y_L^+,\xi_L^+)=~0$.\hfill $\square$
\end{cor}
It is also possible to derive these corollaries using methods coming from
\cite{StipVert}. 
\begin{prop}\label{stabilorient} Let $L$ be a Legendrian knot in 
a contact manifold $(Y,\xi)$ carrying the open book orientation induced by
an adapted open book $(P,\phi)$. Let $(P',\phi')$ be the once-stabilized
open book that carries the Legendrian 
knot $S_+(L)$ (see Proposition~\ref{knotstabil}). The open book
orientation $\ob(P',\phi')$ coincides with the orientation incuded by
the stabilization.
\end{prop}
We will give a proof of Proposition~\ref{stabilorient} in the
following paragraph.

\subsection{Stabilizations of Legendrian Knots and Open Books}

\subsubsection{Stabilizations as Legendrian Band Sums}
Recall that stabilization basically means to enter a 
zigzag into the front projection of a Legendrian knot. If we are 
not in the standard contact space, we perform this operation 
inside a Darboux chart. Which zigzag is regarded as a
positive/negative stabilization depends on the knot orientation.
Positivity/Negativity is fixed by the following equations
\begin{eqnarray*}
  tb(S_\pm(L))&=&tb(K)-1\\
  rot(S_\pm(L))&=&rot(L)\pm 1.
\end{eqnarray*}
This tells us that
\begin{equation}
\overline{S_+(L)}=S_-(\overline{L}).
\end{equation}
Given two Legendrian knots $L$ and $L'$, we can form their
{\bf Legendrian band sum} $L\#_{Lb}L'$ in the following
way: Pick a contact surgery representation of the contact manifold in such
a way that the surgery link $\mathbb{L}$ stays away from 
$L\cup L'$. In this way we can think of $L$ and $L'$ as sitting
in the standard contact space and, so, can perform the band sum.
We denote by $L_0$ and $\overline{L_0}$ the oriented Legendrian
shark with the orientations as indicated in Figure~\ref{Fig:shark}.
\begin{figure}[ht!]
\labellist\small\hair 2pt
\pinlabel {$L_0$} [t] at 50 28
\pinlabel {$\overline{L_0}$} [t] at 259 28
\endlabellist
\centering
\includegraphics[height=2cm]{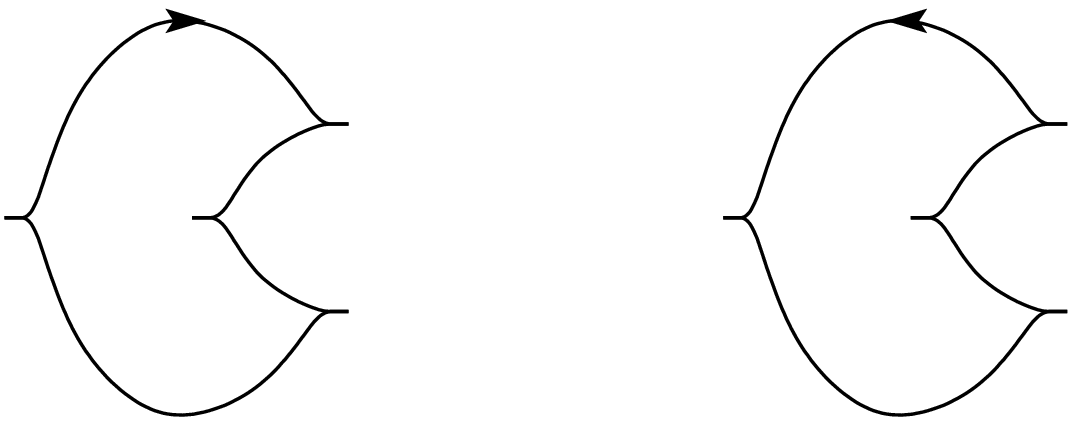}
\caption{The oriented Legendrian shark and its inverse.}
\label{Fig:shark}
\end{figure}

\begin{prop}\label{knotstabil02} Given a Legendrian 
knot $L$, we can realize its stabilizations as Legendrian
band sums, i.e.
\begin{eqnarray*}
 S_+(L)=L\#_{Lb}L_0\\
 S_-(L)=L\#_{Lb}\overline{L_0},
\end{eqnarray*}
where $\#_{Lb}$ denotes the Legendrian band-sum.
\end{prop}
\begin{proof}
We prove the equality for positive stabilizations. The case of
negative stabilizations is proved in a similar fashion. No 
matter what orientation the knot $L$ carries, we will find
at least one right up-cusp or one right down-cusp. In case
of a right down-cusp we perform a band-sum involving
this right down-cusp on $L$ an the left up-cusp on $L_0$.
In case we use a right up-cusp we perform the band-sum
as indicated in the left part of Figure~\ref{Fig:stabilizing}.
In Figure~\ref{Fig:stabilizing} we indicate the Legendrian isotopy
that illustrates that we have stabilized positively.
\end{proof}
\begin{figure}[ht!]
\centering
\includegraphics[height=2.5cm]{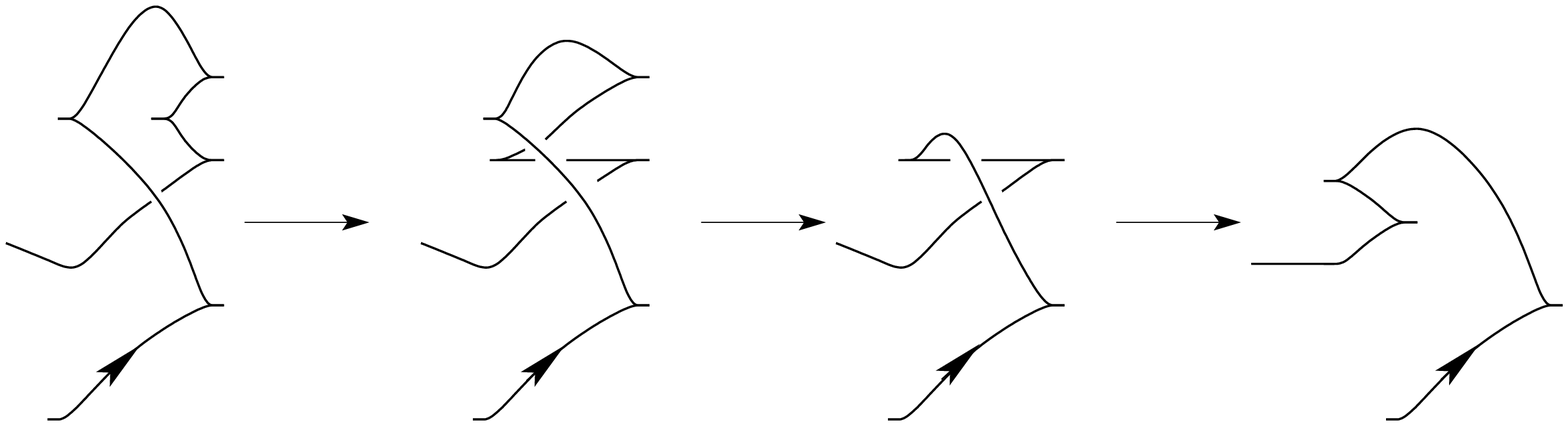}
\caption{The Legendrian band-sum in case of a right up-cusp and
a Legendrian isotopy.}
\label{Fig:stabilizing}
\end{figure}

\subsubsection{Open Books and Connected Sums}
Suppose we are given open books $(P_1,\phi_1)$ and $(P_2,\phi_2)$ for 
manifolds $(Y_1,\xi_1)$ and  $(Y_2,\xi_2)$. Let $B_1$ be the binding
of $(P_1,\phi_1)$. Denote by $\nu B_1$ an equivariant tubular 
neighborhood of $B_1$. Fix a point $p$ on $B_1$ and embed a $3$-ball
$D^3$ such that it is centered at $p$. 
Furthermore, the ball should sit
inside $\nu B_1$ such that the north and south pole of $D^3$ 
equal $B_1\cap\stwo$. Denote by $f_1\co D^3\lra \nu B_1\subset Y_1$ 
the embedding. Embed $g\co D^3\lra Y_2$ in the same fashion. 
Compose $g$ with a right-handed rotation $r$ that swaps the 
two hemispheres of $D^3$ to get another embedding $f_2=g\circ r$. Use 
these embeddings to perform the connected sum. By its definition, the 
gluing $f_2\circ f_1^{-1}$ preserves the open book structure. Note that
the rotation is needed to make the pages of the open book glue 
together nicely with their given orientation.
Moreover, we are able to explicitly describe the resulting 
open book. The new page $P$ equals $P_1\cup_{h^1}P_2$, where $h^1$ is a
$1$-handle connecting $P_1$ and $P_2$ and the 
binding $B$ equals $B_1\#B_2$. To define the monodromy, first 
extend $\phi_1$ and $\phi_2$ as the identity along the handle 
and the complementary page. Then define $\phi$ as the composition 
$\phi_2\circ\phi_1=\phi_1\circ\phi_2$.
\begin{lem}\label{weissnicht} The open book $(P,\phi)$ is 
an adapted open book for $(Y_1\#Y_2,\xi_1\#\xi_2)$.
\end{lem}
\begin{proof} Observe that the given operation is a special 
case of the Murasugi sum. The lemma then follows from \cite{Etnyre01}.
\end{proof}
\begin{cor}\label{niceresult} 
Let $(Y,\xi)$ and $(Y',\xi')$ be contact manifolds and
$L\subset Y$ a Legendrian knot. Then we have
\[
  \begin{array}{ccccc}
   \hfkhat(-Y\#Y',L)&\cong&\hfkhat(-Y,L)&\otimes&\hfhat(-Y')\\
   \loss(Y\#Y',L)  &=    &\loss(Y,L)    &\otimes&c(\xi')
  \end{array}.
\]
\end{cor}
\begin{proof}
Let $(P_1,\phi_1)$ be an open book decomposition adapted to the knot
$L$ and the contact structure $\xi$. Denote by $(P_2,\phi_2)$ an 
open book for $(Y',\xi')$. We define an open book $(P,\phi)$ by using 
the open books for $Y$ and $Y'$
as given above. Recall, that the page $P$ is given by joining the pages
$P_1$ and $P_2$ with a $1$-handle $h^1$, i.e.
\[
  P=P_1\cup_{h^1}P_2.
\]
Denote by $f\co\partial h^1\lra\partial P_1\sqcup\partial P_2$ the
attaching map. Furthermore, let $\{a_1,\dots,a_n\}$ be a cut system 
for $P_1$ and $\{a_1',\dots,a_m'\}$ a cut system for $P_2$. Choose isotopic
push-offs $b_i$ of the $a_i$ so that $a_i$ and $b_i$ intersect
each other in a pair $x_i^+$, $x_i^-$ of intersection points. The push-offs 
are chosen like specified in \S\ref{prelim:01:2} (cf.~also 
Figure~\ref{Fig:zpointpos}). Analogously, the curves $b_j'$, $j=1,\dots,m$, 
are defined; denote the points of intersection by $y_j^+$, $y_j^-$, 
$j=1,\dots,m$. The names are attached to the intersection points in such a
way that $\{x_1^+,\dots,x_n^+\}$ represents the class $\loss(Y,L)$ and that
$\{y_1^+,\dots,y_m^+\}$ represents $c(\xi')$. We additionally fix 
base points $z_i\in P_i$, 
$i=1,2$, and a third one, $w$ say, in $P_1$ determining the knot $L$. These
choices induce Heegaard diagrams we denote by 
$(\Sigma_i,\alpha_i,\beta_i)$, $i=1,2$. We require the chosen cut systems
to fulfill the following two conditions:
\begin{enumerate}
\item[(1)] 
$\im(f)\cap
\Bigl(
\bigcup_{i=1}^{n}\partial a_i
\cup
\bigcup_{j=1}^{m}\partial a_j'
\Bigr)=\emptyset$
\item[(2)]
$\im(f)
\subset
\partial\overline{\dom_{z_1}}\cup\partial\overline{\dom_{z_2}}$
\end{enumerate}
As a consequence of these two conditions and the fact that by definition
$\left.\phi\right|_{P_i}=\phi_i$, $i=1,2$ and 
$\left.\phi\right|_{h^1}=\mbox{\rm id}_{h^1}$ we see that
\begin{equation}
  \phi(a_i)\cap a_j'=\emptyset\;\;\mbox{\rm and}\;\;a_i\cap
  \phi(a_j')=\emptyset.\label{thepoint}
\end{equation}
The set $\{a_1,\dots,a_n\}\cup\{a_1',\dots,a_m'\}$ is a cut system for the
open book $(P,\phi)$. Denote by $(\Sigma,\alpha,\beta)$ the induced Heegaard
diagram, then with $(\ref{thepoint})$, we see that
\[
  \Sigma=\Sigma_1\#\Sigma_2,
  \;\alpha=\alpha_1\cup\alpha_2,
  \;\beta=\beta_1\cup\beta_2
\]
and the points $z_i$, $i=1,2$, lie in the regions unified by the connected
sum tube. Choose a base point $z\in\Sigma$ lying in this unified region. 
Thus, --- with the same reasoning as in the proof of 
\cite{OsZa02}, Proposition 6.1. --- we see that
\begin{equation}
  \hfkhat(-Y\#Y,L)
  \cong
  \hfkhat(-Y,L)
  \otimes
  \hfhat(-Y).\label{isomorphism}
\end{equation}
By construction, the intersection point $\{x_1^+,\dots,x_n^+,y_1^+,\dots,y_m^+\}$ represents
the class $\loss(Y\#Y',L)$. But the isomorphism 
giving $(\ref{isomorphism})$, $\varphi$ say, has that property that
\[
  \{x_1^+,\dots,x_n^+,y_1^+,\dots,y_m^+\}
  \lmt
  \{x_1^+,\dots,x_n^+\}
  \otimes
  \{y_1^+,\dots,y_m^+\},
\]
i.e.~$\varphi(\loss(-Y\#Y',L))=\loss(-Y,L)\otimes c(\xi')$.
\end{proof}
\begin{lem}(\cite{Etnyre01})\label{endlich} If $\gamma$ is a non-separating 
curve on a page of an open book $(P,\phi)$, we can isotope the 
open book slightly such that $\gamma$ is Legendrian and the 
contact framing agrees with the page framing.
\end{lem}
This fact follows from the Legendrian realization principle. 
As a consequence, we get the following corollary.
\begin{cor}\label{keineAhnung} If the Legendrian knots $L_i\subset P_i$
sit on the ages, then, on the page $P$ of $(P,\phi)$, we will find a
Legendrian knot $L$ with the following property: There is a naturally
induced contactomorphism $\phi_c$ such that $\phi_c(L)$ equals 
$L_1\#_{Lb}L_2$ after performing a right-handed twist along the
Legendrian band. Indeed, we obtain $L$ by a band sum of $L_1$ and
$L_2$ on the page $P$.
\end{cor}
\begin{proof} Let $(P_i,\phi_i)$ be open books adapted to
$(Y_i,\xi_i,L_i)$, $i=1,2$. On $P_i$ there is a set of embedded, simple
closed curves $c^i_1,\dots,c^i_n$ whose associated Dehn twists generate
the mapping class groups of $P_i$. The associated Dehn twists can be
interpreted as contact surgeries along suitable Legendrian 
knots (cf.~Theorem~2.7 in \cite{LOSS}). Thus, using the open book
decomposition we are able to find a (maybe very inefficient) contact
surgery representation of $(Y_i,\xi_i)$ which is suitable for our
purposes to perform the Legendrian band sum (cf.~beginning of this
section). Moreover, we can think of $L_1$ to pass the binding $B_1$
of $P_1$ very closely at some point: this means that there is a 
point $p_1$ in the binding, and a Darboux ball $D_1$ around
$p_1$, such that the curve intersects this Darboux ball. Suppose this 
is not the case, then we can isotope the Legendrian knot $L_1$, which 
sits on $P_1$, as a curve in $P_1$, to pass the binding closely (as described
above). The isotopy is not necessarily a Legendrian isotopy. However, by 
Theorem~2.7 of \cite{LOSS}, we know that the isotoped curve determines a 
uniquely defined Legendrian knot, which is Legendrian isotopic to
$L_1$. With a slight isotopy of the open book, we can think of this
new knot as sitting on $P_1$. By abuse of notation, we call the new
knot $L_1$. After possibly isotoping the open book we can think of 
$\mathbb{L}_1$ as sitting in the complement of $D_1$. We obtain
a situation like indicated in the top row of Figure \ref{Fig:ap01}.
\begin{figure}[ht!]
\labellist\small\hair 2pt
\pinlabel {\Large $\sthree(\mathbb{L}_1)=$} [r] at 157 755
\pinlabel {\Large $\sthree(\mathbb{L}_2)=$} [r] at 157 232
\pinlabel {Binding} [l] at 364 971
\pinlabel {Binding} [l] at 972 446
\pinlabel {$\widetilde{D}_1$} [l] at 477 930
\pinlabel {$\widetilde{D}_1$} [l] at 477 397
\pinlabel {$D_1$} [B] at 805 924
\pinlabel {$D_1$} [B] at 805 397
\pinlabel {$\mathbb{L}_1$} at 336 813
\pinlabel {$\mathbb{L}_2$} at 963 286
\pinlabel {\large $\partial$} [tl] at 665 720
\pinlabel {\large $\partial$} [tl] at 665 220
\endlabellist
\centering
\includegraphics[height=10cm]{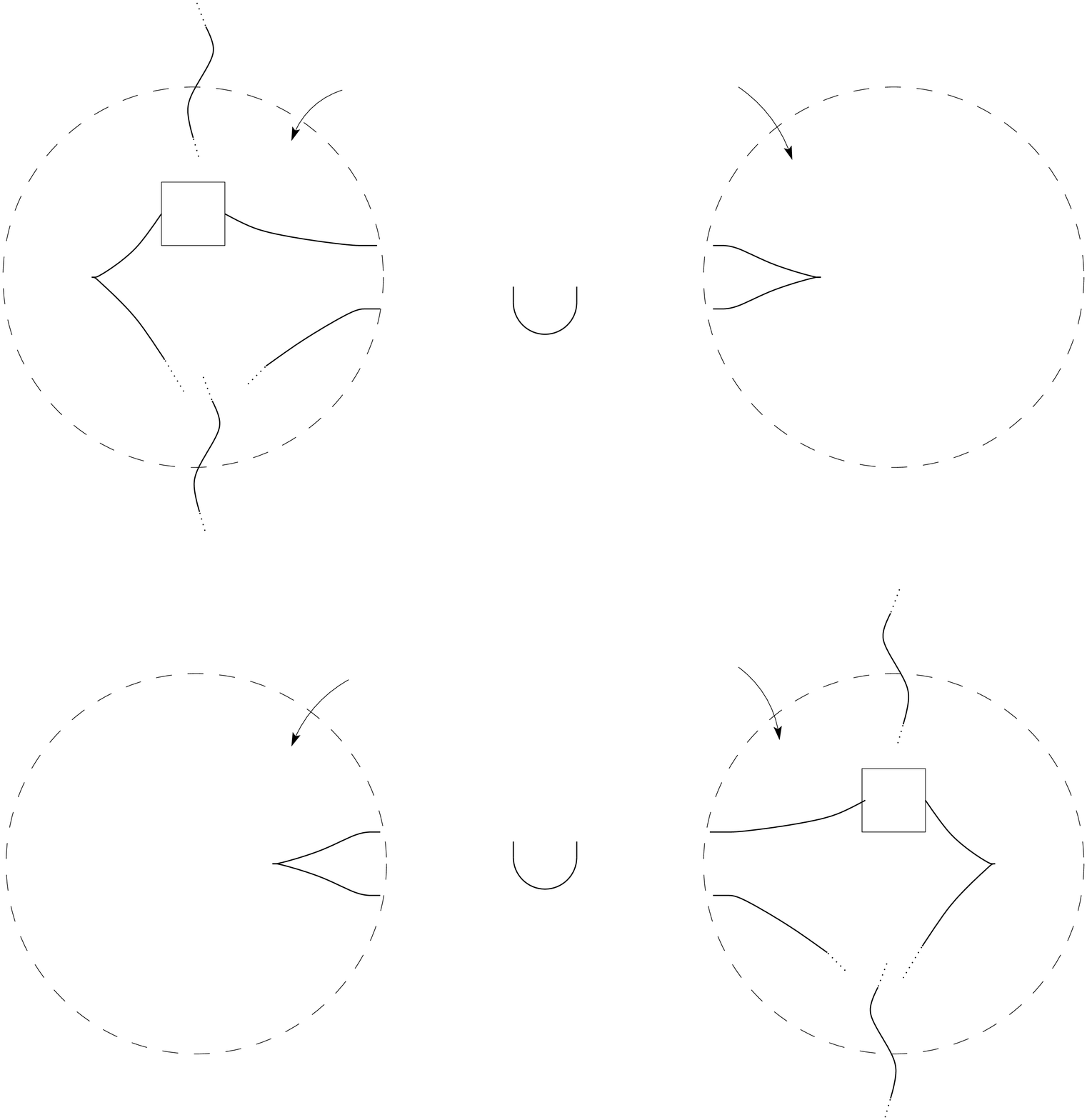}
\caption{Our specific arrangement for performing the connected sum.}
\label{Fig:ap01}
\end{figure}
Since we have the identification $(Y_1,\xi_1)\cong(\sthree(\mathbb{L}_1),\xi_{\mathbb{L}_1})$,
the ball $D_1$ can be thought of as sitting in $\sthree$. The complement
of $D_1$ in $\sthree$ is again a ball we denote by $\widetilde{D_1}$. 
We may make similar arrangements for $L_2$: however, we would like $L_2$ and the 
associated surgery link $\mathbb{L}_2$ to sit inside $D_1$ and $\widetilde{D_1}$ to be 
the ball in which $L_2$ comes close to $B_2$ (cf.~bottom row of Figure \ref{Fig:ap01}). We 
can form the connected sum
\begin{equation}
  \sthree(\mathbb{L}_1\sqcup\mathbb{L}_2)
  =
  \sthree(\mathbb{L}_1)\backslash D_1
  \cup_\partial\stwo\times[0,1]\cup_\partial
  \sthree(\mathbb{L}_2)\backslash\widetilde{D_1}
  \label{glubabyglu}
\end{equation}
where the gluing is determined by the naturally given embeddings (cf.~\S4.12 in \cite{Geiges})
\[
 \iota_1\co D_1\hookrightarrow\sthree
 \;\;\mbox{\rm and }\;\;
 \iota_2\co\widetilde{D_1}\hookrightarrow\sthree.
\]
For a detailed discussion of connected sums of contact manifolds we point the
reader to \cite{Geiges}. The induced contact structure is the connected sum 
$\xi_{\mathbb{L}_1}\#\xi_{\mathbb{L}_2}=\xi_{\mathbb{L}_1\sqcup\mathbb{L}_2}$ 
(cf.~\S4.12 of \cite{Geiges}). The knots $L_1$ and $L_2$ are contained in this 
connected sum and, here, we can perform the Legendrian band sum as defined at 
the beginning of this section; we can perform a band sum which looks like 
given in Figure \ref{Fig:ap03}.
\begin{figure}[ht!]
\labellist\small\hair 2pt
\pinlabel {$\mathbb{L}_1$} at 172 848
\pinlabel {$\mathbb{L}_2$} at 746 848
\pinlabel {$\mathbb{L}_1$} at 243 290
\pinlabel {$\mathbb{L}_2$} at 666 290
\pinlabel {Performing the band sum} [l] at 448 521
\endlabellist
\centering
\includegraphics[height=10cm]{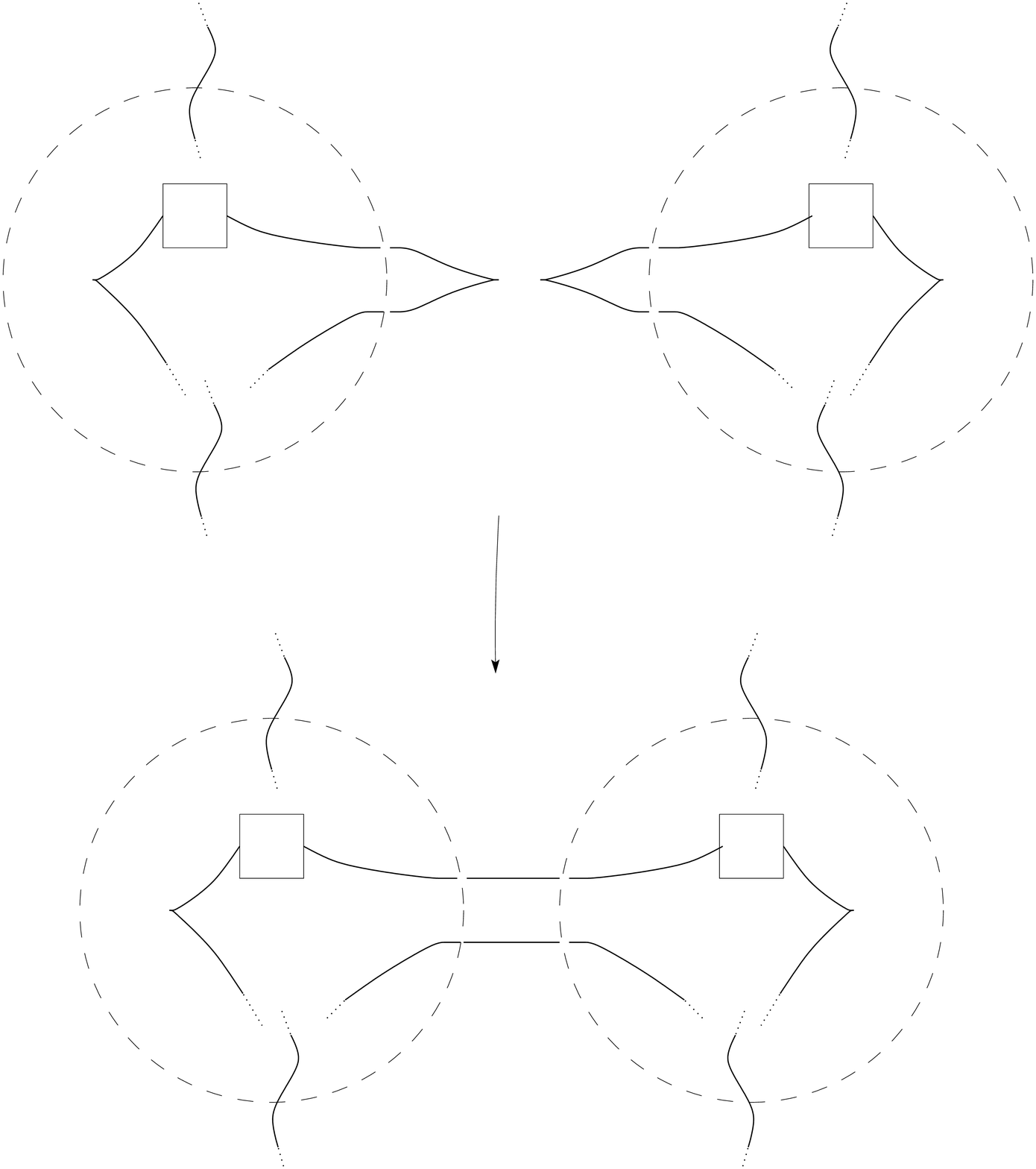}
\caption{Performing a band sum of $L_1$ and $L_2$ inside $\sthree(\mathbb{L}_1\sqcup\mathbb{L}_2)$.}
\label{Fig:ap03}
\end{figure}
Recall that we introduced a connected sum operation such that the open 
books $(P_i,\phi_i)$ glue together to give the open book $(P,\phi)$ 
where $P=P_1\cup_{h_1}P_2$ and $\phi$ is given as the composition of the two monodromies
$\phi_1$ and $\phi_2$. To perform the connected sum operation such that the 
open book structures are preserved, we have to modify the construction 
slightly. We modify the inclusion $\iota_1$ by composing it with a rotation 
about the $y$-axis with angle $\pi$. 
Without loss of generality we can think $L_1\cap\partial D_1$ and $L_2\cap\partial\widetilde{D_1}$ to be 
identified by the gluing induced by the inclusion maps $\iota_1$ and $\iota_2$. 
We can also assume that the rotation $r$ swaps the two intersection points $L_1\cap\partial D_1$. We 
obtain a new gluing map, $f$ say, and get
\[
  Y
  =
  \sthree(\mathbb{L}_1)\backslash D_1
  \cup_{f}
  \sthree(\mathbb{L}_2)\backslash\widetilde{D_1}
\]
with induced contact structure $\xi$. With this identification
the knots $L_1$ and $L_2$ glue together to give a knot $L$. This
knot $L$ corresponds to a band sum of $L_1$ and $L_2$ on the page
$P$ (after possibly applying Proposition~\ref{endlich}). Recall that contact structures on
$\stwo\times[0,1]$ are uniquely determined, up to isotopy, by the
characteristic foliations on $\stwo\times\{j\}$, $j=0,1$ 
(cf.~Lemma~4.12.1 and Theorem~4.9.4 of \cite{Geiges}). Consider 
the connected sum tube used in (\ref{glubabyglu}),
and extend it with small collar neighborhoods of the boundaries of
$\sthree(\mathbb{L}_1)\backslash D_1$ and $\sthree(\mathbb{L}_2)\backslash\widetilde{D_1}$.
The characteristic foliation $\xi_{\mathbb{L}_1\sqcup\mathbb{L}_2}$
induces at the boundary will coincide with the characteristic
foliation $\xi$ induces on a suitably chosen tubular
neighborhood of $\partial D_1\cong\stwo\times[0,1]$ in $Y$. Thus, there 
is a contactomorphism between $\nu D_1$ and this thickened connected sum tube.
Moreover, the contactomorphism can be extended to a contactomorphism
\[
  \phi_c
  \co
  (Y,\xi)
  \lra
  (\sthree(\mathbb{L}_1\sqcup\mathbb{L}_2),\xi_{\mathbb{L}_1\sqcup\mathbb{L}_2})
\]
which just affects the connected sum tube and fixes the rest. As 
one can derive with some effort, this contactomorphism basically rotates the $\stwo$-factor
once while going through the handle $\stwo\times[0,1]$. Thus, $\phi_c(L)$ looks 
like a band sum $L_1\#_{Lb}L_2$ in $\sthree(\mathbb{L}_1\sqcup\mathbb{L}_2)$
after twisting the band once. Figure \ref{Fig:ap02} applies.
\begin{figure}[ht!]
\labellist\small\hair 2pt
\pinlabel {$\mathbb{L}_1$} at 172 848
\pinlabel {$\mathbb{L}_2$} at 746 848
\pinlabel {$\mathbb{L}_1$} at 243 290
\pinlabel {$\mathbb{L}_2$} at 666 290
\pinlabel {Connected sum after applying $\phi_c$} [l] at 448 521
\endlabellist
\centering
\includegraphics[height=10cm]{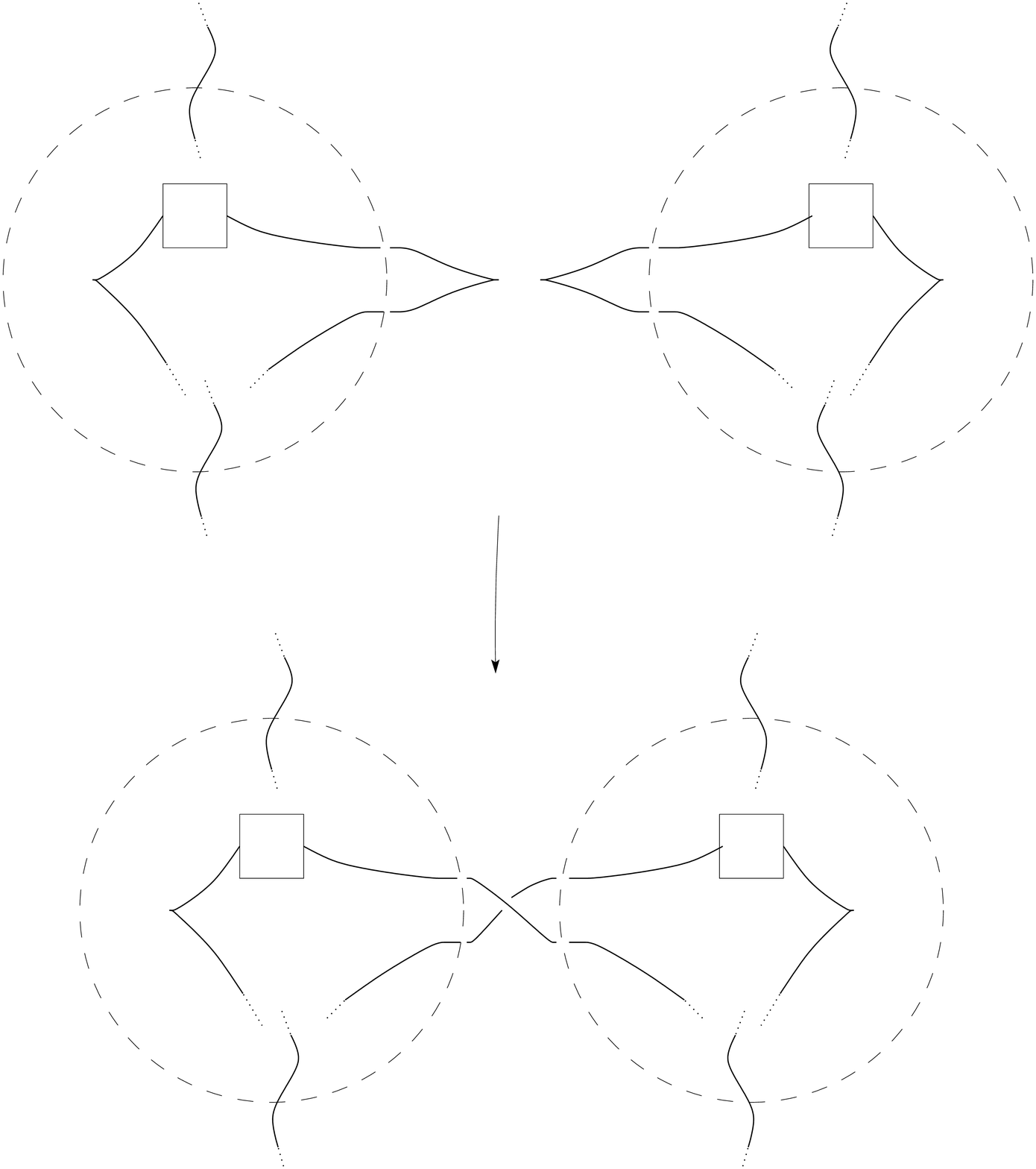}
\caption{Schematic picture of the band bum after idenifying $(Y,\xi)$ with
$(\sthree(\mathbb{L}_1\sqcup\mathbb{L}_2),\xi_{\mathbb{L}_1\sqcup\mathbb{L}_2})$.}
\label{Fig:ap02}
\end{figure}
\end{proof}
The following statement is due to Etnyre. Since there is no proof
in the literature, we include a proof here for the convenience of
the reader.
\begin{prop}(\cite{Etnyre01})\label{knotstabil} 
Let $(Y,\xi,L)$ be a contact manifold with Legendrian knot and
$(P,\phi)$ and open book adapted to $\xi$ with $L$ on its page such that
the page framing and contact framing coincide. By stabilizing the open
book once we can arrange either the stabilized knot $S_+(L)$ or 
$S_-(L)$ to sit on the page of the stabilized open book as 
indicated in Figure~\ref{Fig:illustration}.
\end{prop}
\begin{figure}[ht!]
\labellist\small\hair 2pt
\pinlabel {Legendrian knot} [B] at 156 192
\pinlabel {Legendrian knot} [r] at 540 19
\pinlabel {positive stabilization} [Bl] at 448 254
\endlabellist
\centering
\includegraphics[height=3.5cm]{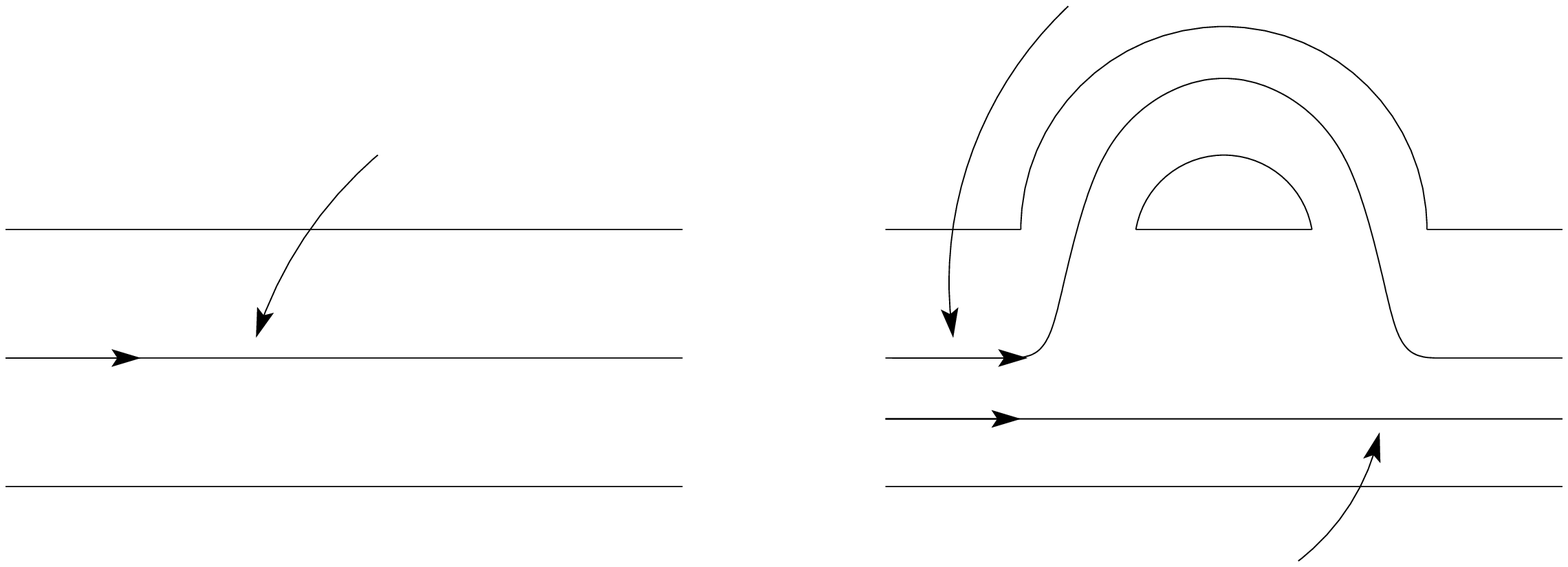}
\caption{The stabilized open book and a positive Legendrian stabilization.}
\label{Fig:illustration}
\end{figure}
The following result concerning the vanishing of the Legendrian invariant
under positive stabilizations is due to Lisca, Ozsv\'ath, Stipsicz and Szab\'o 
and follows from their connected sum formula given in \cite{LOSS}. Their proof
carries over verbatim even for knots which are homologically non-trivial. 
Here we reprove a special case of Theorem 7.2. of \cite{LOSS} using different
methods.
\begin{prop}[\cite{LOSS}, Theorem 7.2]\label{reproof} Given any 
Legendrian knot $L$ in a contact manifold $(Y,\xi)$, we have $\loss(S_+(L))=0$.
\end{prop}
\begin{proof} Let $(P,\phi)$ be an open book decomposition adapted
to $(Y,\xi,L)$. By
Proposition~\ref{knotstabil} we know that a stabilized
open book $(P',\phi')$ carries the stabilized knot $S_+(L)$.
Furthermore, from Figures~\ref{Fig:illustration} and \ref{Fig:proof2}
we can see how the induced Heegaard diagram (adapted to 
capturing the contact geometric information) will look like near
the base point $w$. This is done in Figure~\ref{Fig:easyproof}.
\begin{figure}[ht!]
\labellist\small\hair 2pt
\pinlabel {$x_1$} [B] at 449 414
\pinlabel {$x_2$} [l] at 534 336
\pinlabel {$w$} [tl] at 374 282
\pinlabel {$z$} [Bl] at 564 184
\pinlabel {$p$} [t] at 136 131
\pinlabel {$q$} [t] at 197 117
\pinlabel {Binding} [t] at 335 31
\pinlabel {$_1$} [t] at 456 47
\pinlabel {$_2$} [l] at 488 77
\endlabellist
\centering
\includegraphics[height=6cm]{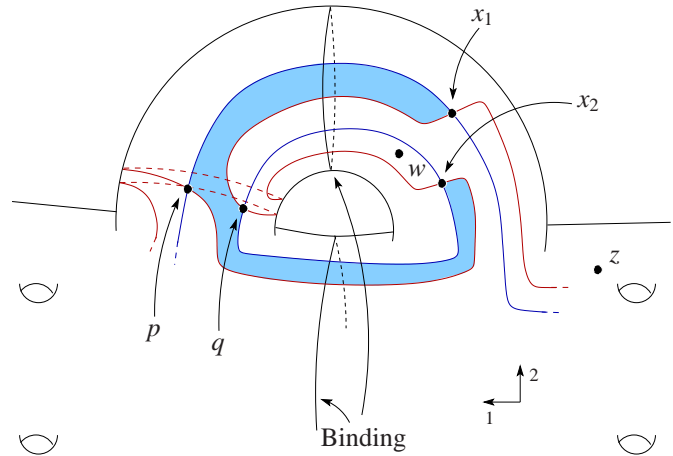}
\caption{Parts of the Heegaard diagram induced by the open book 
carrying the stabilized knot.}
\label{Fig:easyproof}
\end{figure}
We may use Proposition~\ref{stabilorient} to check that the
positioning of the point $w$ in Figure~\ref{Fig:easyproof} is 
correct. First observe that $\loss(S_+(L))$ is the homology class 
induced by the point
\[
  \{x_1,x_2,x_3,\dots,x_{2g}\}.
\]
Recall that by definition of the points $x_i$ every holomorphic
disc emanating from $x_i$ is constant.
Thus, a holomorphic disc emanating from 
$Q:=\{p,q,x_3,\dots,x_{2g}\}$ can only be non-constant
at $p,q$. By orientation reasons and the placement of $w$ the 
shaded region is the only region starting at $p,q$ which can 
carry a holomorphic disc. Since it is disc-shaped, it
does carry a holomorphic disc. Hence
\[
  \parhat^w Q
  =
  \{x_1,x_2,x_3,\dots,x_{2g}\}
\] 
showing that $\loss(S_+(L))$ vanishes.
\end{proof}
\begin{figure}[ht!]
\labellist\small\hair 2pt
\pinlabel {The knot $L$} [l] at 305 311
\pinlabel {The curve along which} [l] at 341 221
\pinlabel {to perform a Dehn Twist} [l] at 341 190
\endlabellist
\centering
\includegraphics[height=4cm]{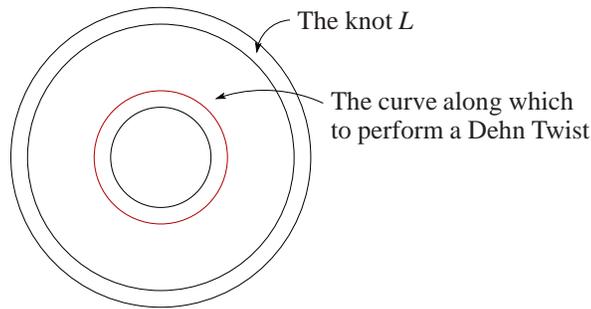}
\caption{The open book necessary to carry the Legendrian unknot with $tb=-1$ and $rot=0$.}
\label{Fig:finalopenbook}
\end{figure}
\begin{proof}[Proof of Proposition~\ref{knotstabil}] Given a triple
$(Y,\xi,L)$, there is an open book $(P,\phi)$ adapted to $\xi$ such that
$L$ sits on a page of the open book. By Proposition~\ref{knotstabil02},
Lemma~\ref{weissnicht} and Corollary~\ref{keineAhnung} we perform a 
connected sum $(Y,\xi)\#(\sthree,\xistd)$ on the level of open books using the 
open book of $(\sthree,\xistd)$ pictured in 
Figure~\ref{Fig:finalopenbook}. By construction, the new open book
carries the Legendrian knot $L_2$ pictured in Figure \ref{Fig:hopefulend}.
\begin{figure}[ht!]
\labellist\small\hair 2pt
\pinlabel {$\mathbb{L}_1$} at 95 173
\pinlabel {$L_2$} [tl] at 222 40
\pinlabel {$-1$} [l] at 460 185
\endlabellist
\centering
\includegraphics[height=2cm]{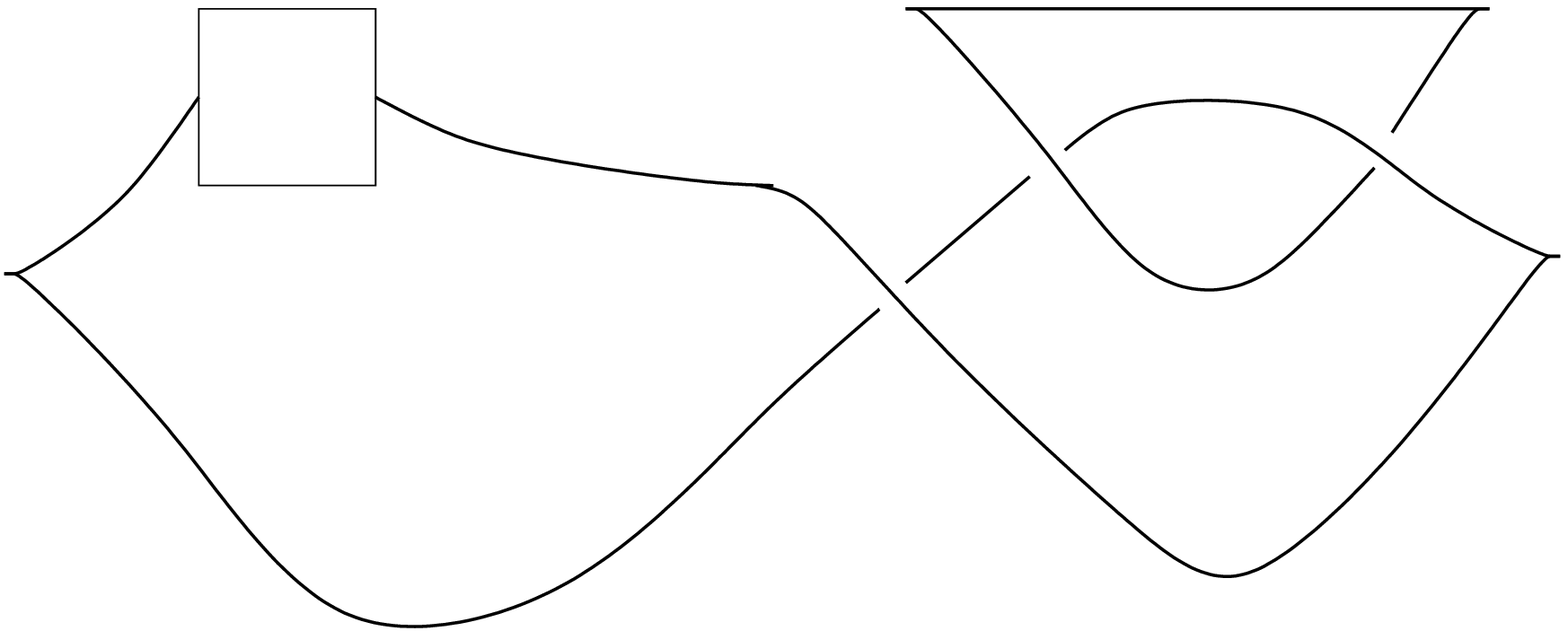}
\caption{The knot $L_2$ in $(Y,\xi)\#(\sthree,\xistd)$.}
\label{Fig:hopefulend}
\end{figure}
In Figure \ref{Fig:hopefulend02} an isotopy is given, showing that
$L_2$ corresponds to the band sum $L\#_{Lb}L_0$ and, thus, represents $S_\pm(L)$.

\begin{figure}[ht!]
\labellist\small\hair 2pt
\pinlabel {$\mathbb{L}_1$} at 82 1179
\pinlabel {$\mathbb{L}_1$} at 82 977
\pinlabel {$\mathbb{L}_1$} at 82 740
\pinlabel {$\mathbb{L}_1$} at 82 541
\pinlabel {$\mathbb{L}_1$} at 82 343
\pinlabel {$\mathbb{L}_1$} at 82 130
\pinlabel {$\mathbb{L}_1$} at 82 130
\pinlabel {$\mathbb{L}_1$} at 497 1179
\pinlabel {$\mathbb{L}_1$} at 497 977
\pinlabel {$\mathbb{L}_1$} at 497 740
\pinlabel {$\mathbb{L}_1$} at 497 541
\pinlabel {$\mathbb{L}_1$} at 497 343
\pinlabel {$\mathbb{L}_1$} at 497 130
\pinlabel {$\mathbb{L}_1$} at 497 130
\pinlabel {$L_2$} [t] at 90 1106
\pinlabel {$-1$} [B] at 248 1145
\pinlabel {$-1$} [t] at 84 890
\pinlabel {$-1$} [t] at 189 648
\pinlabel {$-1$} [B] at 251 503
\pinlabel {$-1$} [B] at 231 345
\pinlabel {$-1$} [B] at 200 156
\pinlabel {$-1$} [t] at 499 1092
\pinlabel {$-1$} [t] at 496 888
\pinlabel {$-1$} [B] at 653 767
\pinlabel {$-1$} [B] at 640 544
\pinlabel {$-1$} [B] at 594 380
\pinlabel {$-1$} [B] at 653 159
\endlabellist
\centering
\includegraphics[width=8cm]{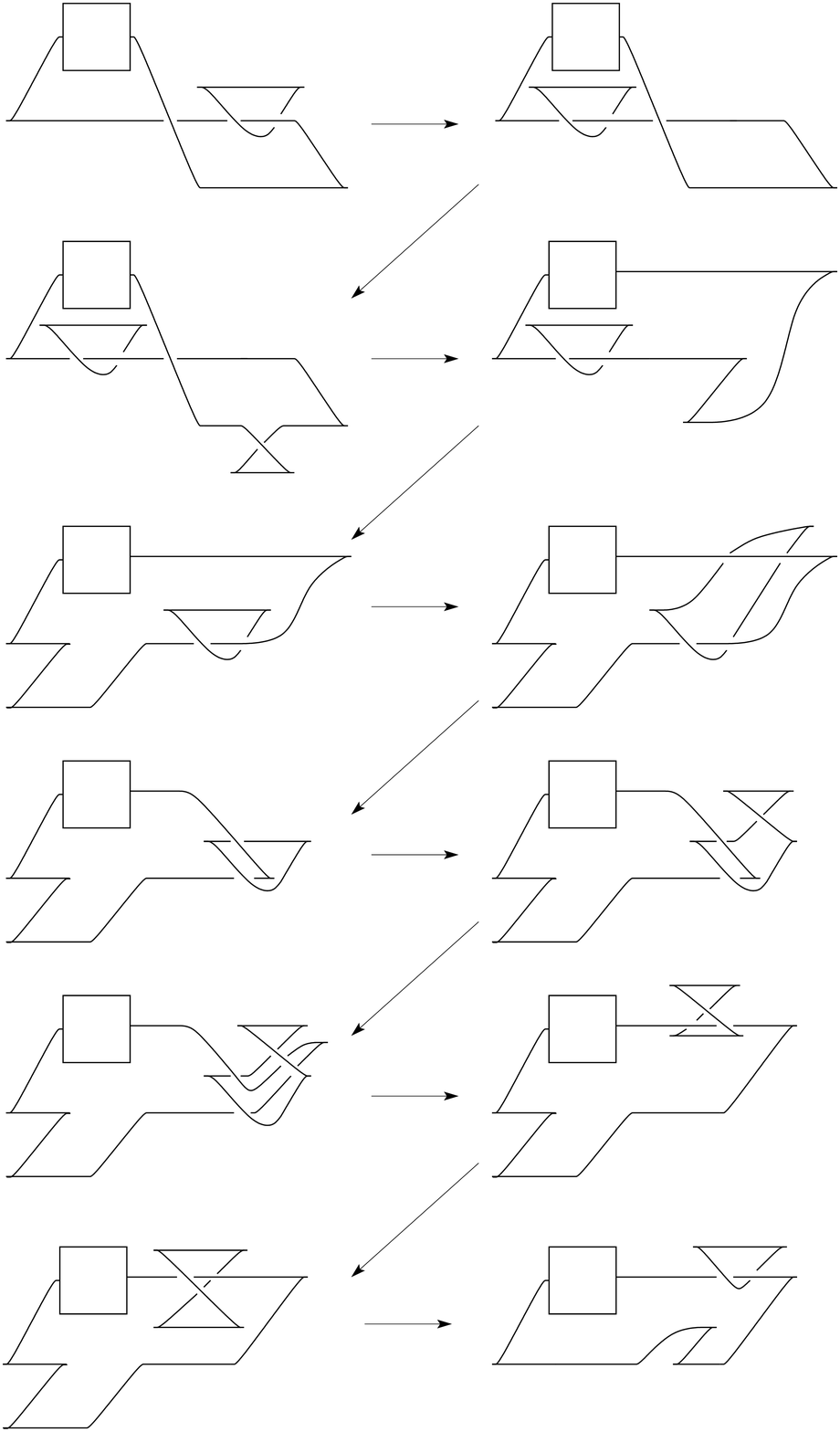}
\caption{Legendrian isotopy showing that $L_2$ corresponds to the Legendrian band sum of $L$ with the Legendrian shark $L_0$.}
\label{Fig:hopefulend02}
\end{figure}
By Figure~\ref{Fig:finalopenbook} what happens on the level of open books can 
be pictured as in Figure~\ref{Fig:proof}.
\begin{figure}[ht!]
\labellist\small\hair 2pt
\pinlabel {$(a)$} [t] at 90 104
\pinlabel {$(b)$} [t] at 347 104
\pinlabel {$(c)$} [t] at 599 104
\pinlabel {$L$} [tl] at 205 21
\pinlabel {stabilized knot} [t] at 419 35
\endlabellist
\centering
\includegraphics[height=5cm]{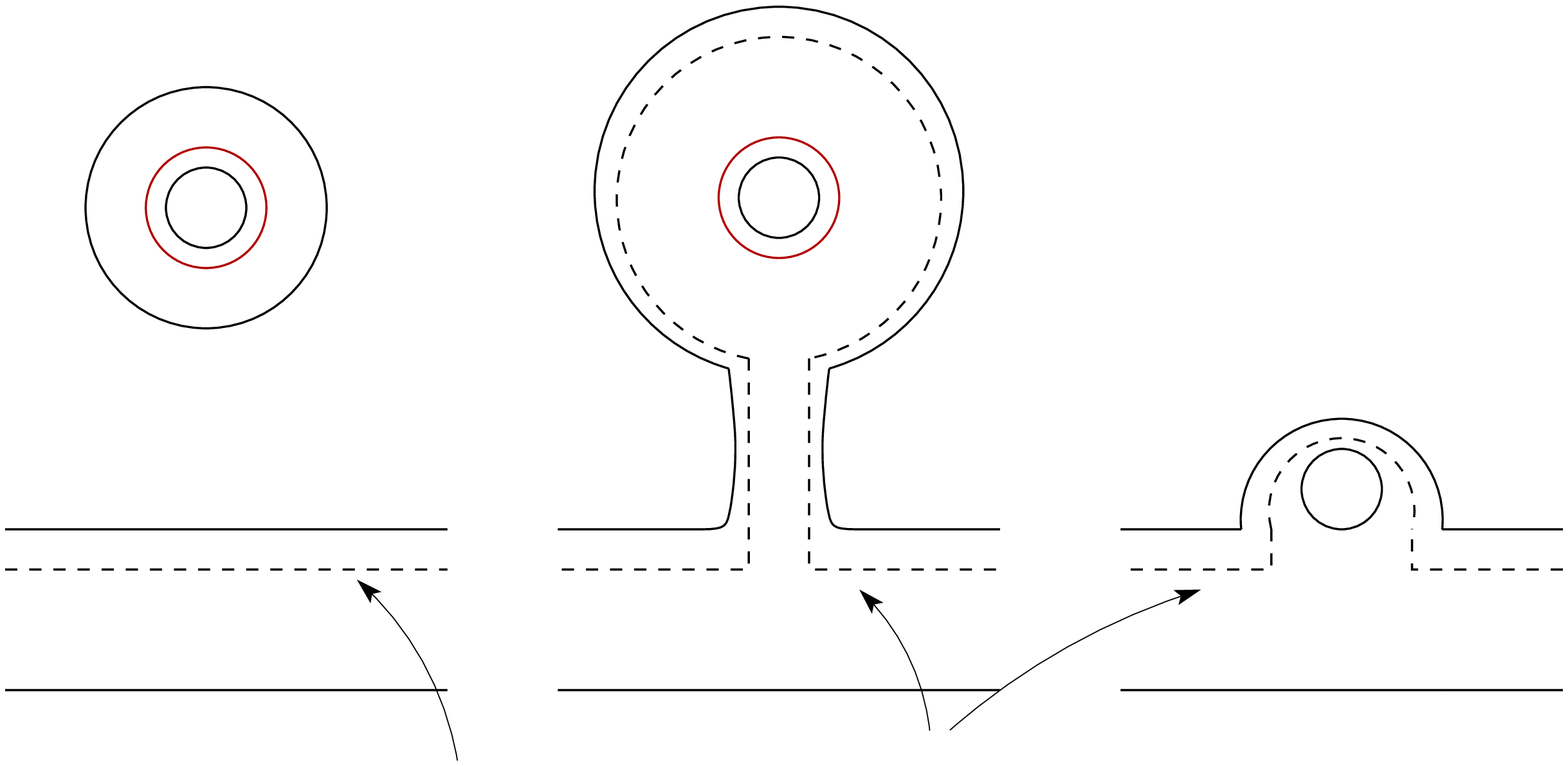}
\caption{What happens during stabilization.}
\label{Fig:proof}
\end{figure}
\end{proof}
\begin{proof}[Proof of Proposition~\ref{stabilorient}]
Using Proposition~\ref{knotstabil}, we have a tool to compare the
open book orientation before and after the stabilization. We start
with an open book adapted to the triple $(Y,\xi,L)$ and choose
an $L$-adapted cut system. By Proposition~\ref{knotstabil} we
can generate an open book adapted to the positive stabilization
by stabilizing the open book. Doing this appropriately,
we may extend the cut system to an adapted cut system of the
stabilized open book as indicated in Figure~\ref{Fig:proof2}. 
Recall the rule with which the knot orientation is determined 
by the points $(w,z)$ (see remark in \S\ref{invariantLOSS}).
In Figure~\ref{Fig:proof2} we can now compare the open book
orientation of the stabilized knot with the orientation induced
by the stabilization.
\begin{figure}[ht!]
\labellist\small\hair 2pt
\pinlabel {open book orientation on its stabilization} [B] at 167 312
\pinlabel {$S_+(L)$} [B] at 475 235
\pinlabel {$L$} [t] at 477 41
\pinlabel {open book orientation} [l] at 202 40
\pinlabel {on the knot} [l] at 202 20
\endlabellist
\centering
\includegraphics[height=5cm]{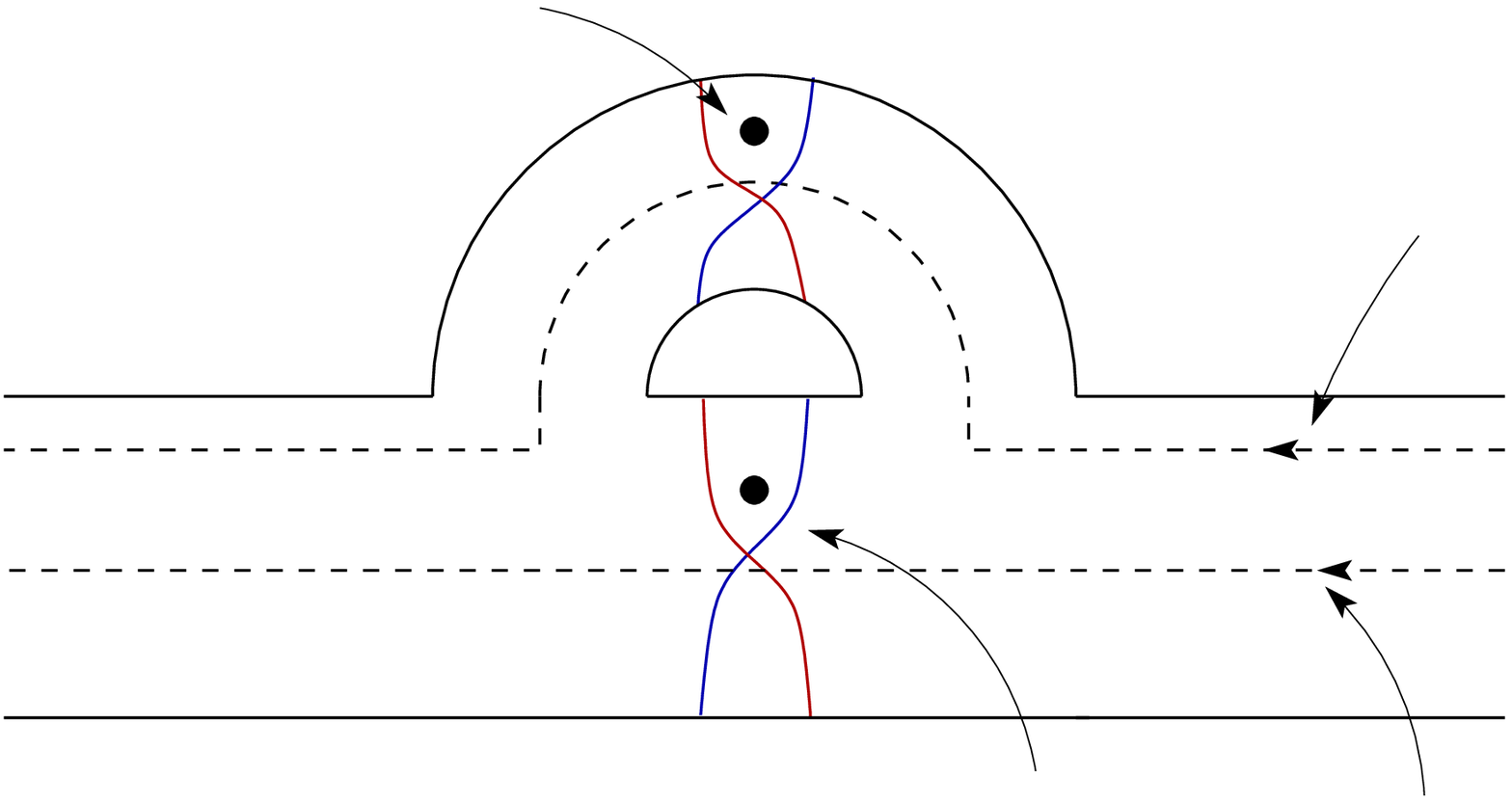}
\caption{Comparing induced with open book orientation.}
\label{Fig:proof2}
\end{figure}
We see that the orientations coincide.
\end{proof}
\section{Applications -- Vanishing Results of the Contact Element}\label{applics}
In this section we want to derive some applications of the theory
developed in \S\ref{introd}, \S\ref{naturality} and \S\ref{contsetup}.
First to mention would be Proposition~\ref{result01}, which can also be
derived using methods developed in \cite{Stip01}. There, Lisca and
Stipsicz show that $(+1)$-contact surgery along stabilized Legendrian
knots yield overtwisted contact manifolds, which implies the vanishing
of the contact element. A second application would be 
Proposition~\ref{calculation}, which is meant as a demonstration that
calculating the Legendrian knot invariant and using Theorem~\ref{maps}
to get information about a contact element under investigation can 
be more convenient than using other methods, since the 
knot Floer homologies have additional structures we may use. A third application 
would be Theorem~\ref{result02} which is a vanishing result of the 
contact element which can be easily read off from a surgery 
representation. This application uses the knot Floer homology for 
arbitrary knots and makes use of a phenomenon that seems
to be special about these, namely that 
there are knots for which the knot Floer homology vanishes. We do 
not know any other example with this property.
\begin{prop}\label{result01} If $(Y,\xi)$ is obtained 
from $(Y',\xi')$ by $(+1)$-contact surgery along a 
Legendrian knot $L$ which can be destabilized, the element $c(\xi)$ vanishes.
\end{prop}
\begin{proof} There are two cases to cover. Give the knot $L$
an orientation $\orient$. Suppose that
\[
  (L,\orient)=S_+(L',\orient').
\]
Then Proposition~\ref{reproof} shows the vanishing 
of $\loss(L,\orient)$. By Theorem~\ref{maps} the element $c(\xi)$ vanishes, too. Now assume that
\[
  (L,\orient)=S_-(L',\orient').
\]
We see that
\[
  (L,\overline{\orient})
  =\overline{S_-(L',\orient')}
  =S_+(L',\overline{\orient}),
\]
hence, $\loss(L,\overline{\orient})=0$. By 
Theorem~\ref{maps} again $c(\xi)=0$.
\end{proof}
There are some immediate consequences we may derive from this
theorem. The first corollary is well-known but with help of our
results we are able to reprove it.
\begin{cor}[Ozsv\'ath and Szab\'o] If $(Y,\xi)$ is overtwisted, the 
contact element vanishes.
\end{cor}
\begin{proof} Recall that the surgery diagram given in 
Figure~\ref{Fig:s3overtwisted} is an overtwisted contact 
structure $\xi'$ on $\sthree$. 
\begin{figure}[ht!]
\labellist\small\hair 2pt
\pinlabel $+1$ [Bl] at 425 165
\pinlabel $+1$ [Bl] at 100 165
\pinlabel $-1$ [Bl] at 278 206
\endlabellist
\centering
\includegraphics[height=3cm]{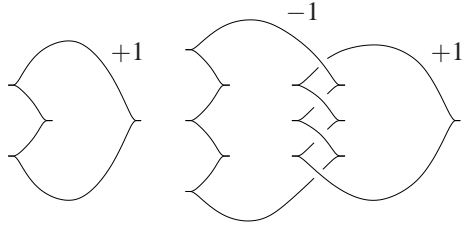}
\caption{Surgery diagram for an overtwisted $\sthree$ in the 
homotopy class of $\xistd$.}
\label{Fig:s3overtwisted}
\end{figure}

This overtwisted contact structure
is homotopic to $\xistd$ as $2$-plane fields (cf.~\cite{DiGei03}).
By Eliashberg's classification theorem (see \cite{eliash}), a
connected sum of $(Y,\xi)$ with $(\sthree,\xi')$ does not
change the contact manifold, i.e.
\[
  (Y,\xi)=(Y,\xi)\#(\sthree,\xi').
\]
Denote by $K$ the shark on the left of Figure~\ref{Fig:s3overtwisted}.
The manifold $(Y,\xi)$ admits a surgery representation $\sthree(\mathbb{L})$
where $\mathbb{L}=K\sqcup\mathbb{L}'$. Furthermore, $K$ and $\mathbb{L}'$
are not linked. Denote by $(Y',\xi'')$ the contact manifold
with surgery representation $\sthree(\mathbb{L}')$. We obtain $(Y,\xi)$
out of $(Y',\xi'')$ by $(+1)$-contact surgery along $K$, which can be
destabilized inside~$Y'$. Proposition~\ref{result01} implies the vanishing
of $c(\xi)$.
\end{proof}
\begin{rem} For a detailed discussion of the homotopy
invariants of overtwisted contact structures on $\sthree$ see
\cite{DiGeiSti02}. 
\end{rem}
Another consequence is that performing a simple Lutz twist along a transverse
knot kills the contact element. The resulting contact structure is clearly overtwisted. 
Thus, by work of Ozsv\'{a}th and Szab\'{o} the contact element vanishes. 
But besides this approach we can show the vanishing of the contact element 
without referring to overtwistedness at all. In \cite{DiGeiSti} a 
surgical description for simple Lutz twists along transverse knots is presented. This 
description involves $(+1)$-contact surgeries along a Legendrian approximation $L$ of 
the transverse knot and another Legendrian knot which is
a stabilized version of $L$. Proposition~\ref{result01} then implies the 
vanishing of the contact element.\vspace{0.3cm}\\
When looking at a homologically trivial knot $L$, to show the 
vanishing of a contact element after surgery along $L$ it can
be convenient to show the vanishing of $\loss(L)$ and then apply
Theorem~\ref{maps}, because of the various gradings on the knot Floer
homological level. The following proposition is meant 
as an illustration of this fact.
\begin{prop}\label{calculation} A $(+1)$-contact surgery 
along the Legendrian realizations $L_n$ given 
in Figure~\ref{Fig:elchek} of the Eliashberg-Chekanov twist 
knots $E_n$ with 
$n\in-2\mathbb{N}$ all give contact manifolds with vanishing 
contact element.
\end{prop}
\begin{figure}[ht!]
\labellist\small\hair 2pt
\pinlabel {$n$} at 126 155 
\pinlabel {$n$} at 418 155
\pinlabel {$L_n$} [Bl] at 212 246
\pinlabel {$E_n$} [Br] at 342 254
\pinlabel {Legendrian realizations of the twists} [l] at 135 80
\endlabellist
\centering
\includegraphics[width=8cm]{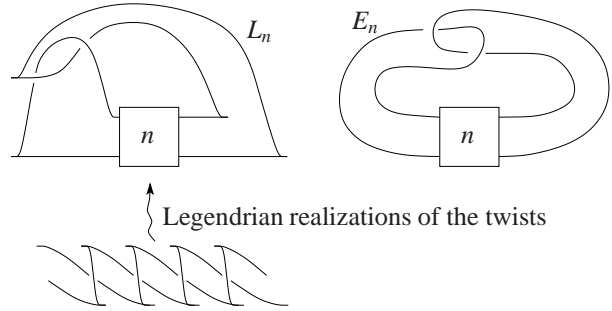}
\caption{The Eliashberg-Chekanov twist knots $E_n$ and 
Legendrian realizations $L_n$.}
\label{Fig:elchek}
\end{figure}
\begin{proof} Since the invariant $\loss(L_n)$ of the Legendrian 
realizations $L_n$ of the
knots $E_n$ live in $\hfkhat(-\sthree,E_n)$, and because of the
correspondence
\[
  \hfkhat(-\sthree,E_n)=\hfkhat(\sthree,\overline{E_n}),
\]
where $\overline{E_n}$ denotes the mirror knot, we have to compute the 
groups $\hfkhat(\sthree,\overline{E_n})$. The knots are all alternating. 
Therefore we will stick to Theorem 1.3 of \cite{OsZa08} for a convenient
computation of the groups. We compute the Alexander-Conway polynomial using 
its skein relation and get
\[
  \Delta_{\overline{E_n}}(T)= (1-n)+\frac{n}{2}(T^1+T^{-1}).
\]
To compute the signature of the knots $E_n$, we use the formula given in
Theorem 6.1 of \cite{OsZa08} and see that all these knots have signature 
$\sigma(\overline{E_n})=-n-2$. By Theorem 1.3 of \cite{OsZa08}, which describes
the knot Floer homology groups of an alternating knot in terms of the 
coefficients of the associated Alexander-Conway polynomial, the knot 
Floer homology of $\overline{E_n}$ looks like
\[
  \hfkhat_j(\sthree,\overline{E_n},i)
  =
  \left\{
  \begin{matrix}
  \Z^{-n/2},& i=-1,j=-1+\frac{-n-2}{2}\\ 
  \Z^{|1-n|},& i=0,j=\frac{-n-2}{2} \\
  \Z^{-n/2},& i=1,j=1+\frac{-n-2}{2}\\
  0        ,& \mbox{\rm otherwise}
  \end{matrix}\right. .
\]
According to \cite{OsSti}, the Legendrian invariant $\loss(L_n)$ lives 
in $\hfkhat_{M(L_n)}(-\sthree,E_n,A(L_n))$ where $A(L_n)$ is 
the {\it Alexander grading} of $L_n$ and $M(L_n)$ is called 
{\it Maslov grading}. These gradings are computed using the 
formulas (see \cite{OsSti})
\begin{eqnarray*}
  2\cdot A(L_n)&=&tb(L_n)-rot(L_n)+1\\
  d_3(\xistd)&=& 2A(L_n)-M(L_n),
\end{eqnarray*}
where $d_3$ denotes the Hopf-invariant (cf.~\cite{GoSt}). However, note that with 
the conventions used in Heegaard Floer theory $d_3(\xistd)=0$. With 
a straightforward computation we 
see that $tb(L_n)=-4$ and $rot(L_n)=1$, which give the following 
Alexander gradings and Maslov gradings
\begin{eqnarray*}
  A(L_n)&=&-1\\
  M(L_n)&=&-2.
\end{eqnarray*}
Consequently, we can show, by using the computed Alexander and Maslov gradings, that 
for every knot $L_n$, $n\not=0$, the invariant $\loss(L_n)$ is an element of a 
vanishing subgroup of $\hfkhat(\sthree,\overline{E_n})$. To show the vanishing 
of $\loss(L_0)$ we observe that $L_0$ can be destabilized.
\begin{figure}[ht!]
\centering
\includegraphics[width=10cm]{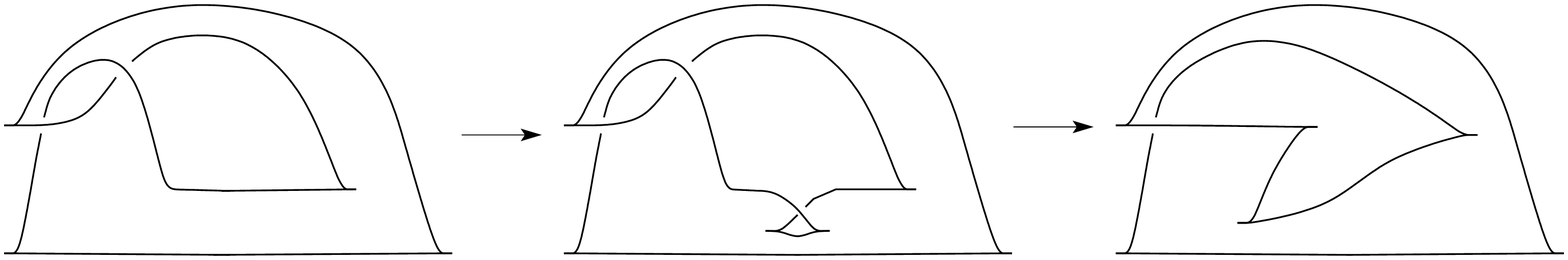}
\caption{The Legendrian isotopy showing that $L_0$ can be destabilized.}
\label{Fig:elchek2}
\end{figure}

The isotopy is pictured in Figure~\ref{Fig:elchek2}. 
By Proposition~\ref{result01} $c(\xi^+_{L_0})$ vanishes, too.  Using 
Theorem~\ref{maps} the proposition follows.
\end{proof}
The following theorem is a new vanishing result of the contact element, which 
uses the knot Floer homology for arbitrary knots.
Furthermore, we make use of the fact that in $\stwo\times\sone$ there are
homologically non-trivial knots whose associated knot Floer homology
vanishes.
\begin{theorem}\label{result02} Let $(Y,\xi)$ be a contact manifold 
given as a contact
surgery along a Legendrian link in $(\sthree,\xistd)$. If the surgery
diagram contains a configuration like given in 
Figure~\ref{Fig:convanish}, the contact element $c(Y,\xi)$ vanishes.
\end{theorem}
\begin{figure}[ht!]
\labellist\small\hair 2pt
\pinlabel {$+1$} [B] at 236 62
\pinlabel {$+1$} [tl] at 180 25
\pinlabel {$K'$} [B] at 14 62
\pinlabel {$K$} [tr] at 55 20
\endlabellist
\centering
\includegraphics[height=1.5cm]{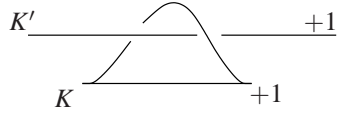}
\caption{Configuration in a surgery diagram of $(Y,\xi)$ killing 
the contact element.}
\label{Fig:convanish}
\end{figure}
\begin{proof} We start looking at the knot Floer homology group of the
pair $(\stwo\times\sone,G)$ where $G$ is a specific knot representing
a generator of $H_1(\stwo\times\sone)$: 
\begin{figure}[ht!]
\labellist\small\hair 2pt
\pinlabel {$\alpha$} [l] at 279 334
\pinlabel {$x$} [Bl] at 283 244
\pinlabel {$z$} [l] at 422 219
\pinlabel {$w$} [l] at 310 183
\pinlabel {$y$} [tl] at 284 121
\pinlabel {$\beta$} [r] at 146 44
\endlabellist
\centering
\includegraphics[width=6cm]{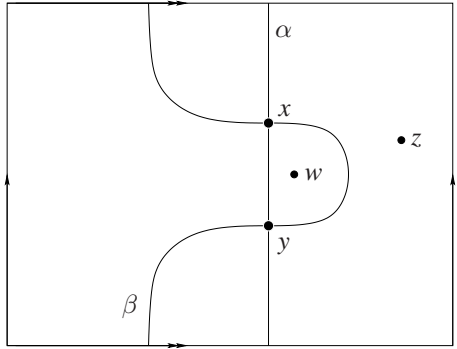}
\caption{Heegaard diagram adapted to $G$}
\label{Fig:modelknotfh}
\end{figure}

Figure~\ref{Fig:modelknotfh} 
is a Heegaard diagram adapted to this specific knot $G$. A 
straightforward calculation gives $\hfkhat(\stwo\times\sone,G)=0$.
In Figure \ref{Fig:identknot} we see a surgery diagram of $\stwo\times\sone$ 
with the knot $G$ in it. 
\begin{figure}[ht!]
\labellist\small\hair 2pt
\pinlabel {$G$} [Bl] at 110 100
\pinlabel {$0$} [tl] at 76 17
\endlabellist
\centering
\includegraphics[width=2.5cm]{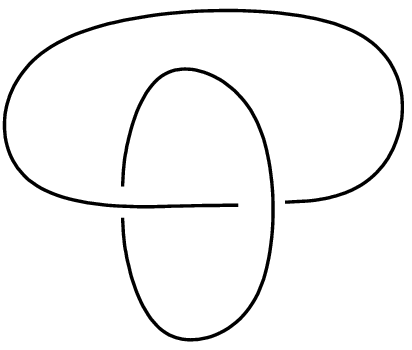}
\caption{Surgery diagram of $\stwo\times\sone$ with knot $G$ in it.}
\label{Fig:identknot}
\end{figure}

Returning to Figure~\ref{Fig:convanish}, we can interpret $K'$ as an 
ordinary knot and remove it from the surgery description. We obtain a
contact manifold $(Y'\#\stwo\times\sone,\xi')$ 
and $K'$ is a Legendrian knot in it. A $(+1)$-contact surgery along $K'$ 
will yield $(Y,\xi)$. Furthermore, as a topological knot, $K'$ can be 
written as $K''\#G$ where $K''\subset Y$ and $G\subset\stwo\times\sone$ is 
the knot given in Figure \ref{Fig:identknot}. Hence, we 
have (cf.~\cite{LOSS})
\[
  \hfkhat(Y'\#(\stwo\times\sone),K')
  =
  \hfkhat(Y',K'')\otimes\hfkhat(\stwo\times\sone,G)
  =0.
\]
The same holds if we reverse the orientation on the manifold.
We perform a $(+1)$-contact surgery along $K'$ to obtain $(Y,\xi)$.
Denote by $W$ the induced cobordism. By Theorem~\ref{maps} this induces
a map
\[
  \Gamma_{-W}
  \co
  \hfkhat(-Y'\#(\stwo\times\sone),K')
  \lra
  \hfhat(-Y)
\]
with $c(Y,\xi)=\Gamma_{-W}(\loss(K'))$. So, the contact element vanishes, since $\loss(K')=0$.
\end{proof}

\bibliographystyle{amsplain}

\providecommand{\bysame}{\leavevmode\hbox to3em{\hrulefill}\thinspace}
\providecommand{\MR}{\relax\ifhmode\unskip\space\fi MR }
\providecommand{\MRhref}[2]{%
  \href{http://www.ams.org/mathscinet-getitem?mr=#1}{#2}
}
\providecommand{\href}[2]{#2}

\end{document}

%% file: abbreviations.tex
\newcommand{\sone}{\mathbb{S}^1}
\newcommand{\lra}{\longrightarrow}
\newcommand{\lmt}{\longmapsto}
%

%
%
\newcommand{\ztwo}{\mathbb{Z}_2}
\newcommand{\RP}{\mbox{\rm RP}}
\newcommand{\spinc}{\mbox{\rm Spin}^c}
\def\co{\colon\thinspace}
\newcommand{\stwo}{\mathbb{S}^2}
\newcommand{\pd}{\text{PD}}
\newcommand{\sprung}{\\[0.3cm]}
\newcommand{\fund}{\pi_1}
\newcommand{\kerg}{\mbox{\rm Ker}_G\,}
\def\Ker#1{\mbox{\rm Ker}_{#1}\,}
\newcommand{\im}{\mbox{\rm Im}\,}
\newcommand{\id}{\mbox{\rm id}}
\newcommand{\sothree}{\mathbb{SO}_3}
\newcommand{\sthree}{\mathbb{S}^{3}}
\newcommand{\disc}{\mbox{\rm D}}
\newcommand{\systwo}{C^{\infty}(Y,\stwo)}
\newcommand{\inner}{\text{int}}
\newcommand{\bk}{\backslash}

%
%
\newcommand{\contstand}{(\mathbb{R}^3,\xi_0)}
\newcommand{\cont}{(M,\xi)}
\newcommand{\sym}{\xi_{sym}}
\newcommand{\xistd}{\xi_{std}}
\newcommand{\cyl}{\mbox{\rm Cyl}_{r_0}^\mu}
\newcommand{\stap}{\mbox{\rm S}_+}
\newcommand{\stam}{\mbox{\rm S}_-}
\newcommand{\stapm}{\mbox{\rm S}_\pm}
\newcommand{\modulo}{\;\;\mbox{\rm mod}\,}

%
%
\newcommand{\crit}{\mbox{\rm Crit}}
\newcommand{\MC}{\mbox{\rm MC}}
\newcommand{\bmorse}{\partial^{\mbox{\rm \begin{tiny}M\!C\end{tiny}}}}
\newcommand{\M}{\mathcal{M}}
\newcommand{\stable}{\mbox{\rm W}^s}
\newcommand{\unstable}{\mbox{\rm W}^u}
\newcommand{\ind}{\mbox{\rm ind}}
\newcommand{\Mhat}{\widehat{\M}}
\newcommand{\modphi}{\mathcal{M}_\phi}
%
%
%

%
%
\newcommand{\moduli}{\mathcal{M}_{J_s}(x,y)}
\newcommand{\modulit}{\mathcal{M}_{J_{s,t}}}
\newcommand{\modulittau}{\mathcal{M}_{J_{s,t(\tau)}}}
\newcommand{\modulittaubig}{\mathcal{M}_{J_{s,t,\tau}}}
\newcommand{\modhat}{\widehat{\mathcal{M}}_{J_s}(x,y)}
\newcommand{\modhatone}{\widehat{\mathcal{M}}_{J_{s,1}}}
\newcommand{\modhatzero}{\widehat{\mathcal{M}}_{J_{s,0}}}
\newcommand{\modhatphi}{\widehat{\mathcal{M}}_{\phi}}
\newcommand{\moduliiso}{\mathcal{M}^{\Psi_t}}
\newcommand{\modspace}{\mathcal{M}}
\newcommand{\phihat}{\widehat{\phi}}
\newcommand{\Phiinfty}{\Phi^\infty}
\newcommand{\Dhat}{\widehat{\D}}
\newcommand{\cops}{\partial_{J_s}}
\newcommand{\talpha}{\mathbb{T}_\alpha}
\newcommand{\tbeta}{\mathbb{T}_\beta}
\newcommand{\tgamma}{\mathbb{T}_\gamma}
\def\marge#1{\marginpar{\scriptsize{#1}}}
\def\br#1{\begin{rotate}{90}#1\end{rotate}}
\newcommand{\hfhat}{\widehat{\mbox{\rm HF}}}
\newcommand{\sfh}{\mbox{\rm SFH}}
\newcommand{\sbottom}{\underline{s}}
\newcommand{\tbottom}{\underline{t}}
\newcommand{\cfhat}{\widehat{\mbox{\rm CF}}}
\newcommand{\cfinfty}{\mbox{\rm CF}^\infty}
\def\cfbb#1{\mbox{\rm CF}^+_{\leq #1}}
\def\cfbo#1{\mbox{\rm CF}^-_{\geq -#1}}
\newcommand{\cfleq}{\mbox{\rm CF}^{\leq 0}}
\newcommand{\cfcirc}{\mbox{\rm CF}^\circ}
\newcommand{\cfkinfty}{\mbox{\rm CFK}^{\infty}}
\newcommand{\cfkhat}{\widehat{\mbox{\rm CFK}}}
\newcommand{\cfkminus}{\mbox{\rm CFK}^{-}}
\newcommand{\cfkpstar}{\mbox{\rm CFK}^{+,*}}
\newcommand{\cfkostar}{\mbox{\rm CFK}^{0,*}}
\newcommand{\hfkcirc}{\mbox{\rm HFK}^\circ}
\newcommand{\hfkhat}{\widehat{\mbox{\rm HFK}}}
\newcommand{\hfkplus}{\mbox{\rm HFK}^+}
\newcommand{\hfkminus}{\mbox{\rm HFK}^-}
\def\cfinftyfilt#1#2{\mbox{\rm CFK}^{#1,#2}}
\newcommand{\hfinfty}{\mbox{\rm HF}^\infty}
\newcommand{\hfinftwist}{\underline{\mbox{\rm {HF}}^\infty}}
\newcommand{\fhat}{\widehat{f}}
\newcommand{\fcirc}{f^\circ}
\newcommand{\Fhat}{\widehat{F}}
\newcommand{\Fcirc}{F^\circ}
\newcommand{\hattheta}{\widehat{\Theta}}
\newcommand{\shattheta}{\widehat{\theta}}
\newcommand{\cfminus}{\mbox{\rm CF}^-}
\newcommand{\hfminus}{\mbox{\rm HF}^-}
\newcommand{\cfplus}{\mbox{\rm CF}^+}
\newcommand{\hfplus}{\mbox{\rm HF}^+}
\newcommand{\hfcirc}{\mbox{\rm HF}^\circ}
\def\hfbb#1{\hfplus_{\leq #1}}
\def\hfbo#1{\hfminus_{\geq -#1}}
\newcommand{\Hs}{\mathcal{H}_s}
\newcommand{\gr}{\mbox{gr}}
\newcommand{\parinfty}{\partial^\infty}
\newcommand{\parhat}{\widehat{\partial}}
\newcommand{\parplus}{\partial^+}
\newcommand{\parminus}{\partial^-}
\newcommand{\symg}{\mbox{\rm Sym}^g(\Sigma)}
\newcommand{\symgg}{\mbox{\rm Sym}^{2g}(\Sigma)}
\newcommand{\symc}{\mbox{\rm Sym}^g(\mathbb{C})}
\newcommand{\pitwo}{\pi_2}
\newcommand{\pitwoham}{\pi_2^{\Psi_t}}
\newcommand{\symcon}{\mbox{\rm Sym}^g(\Sigma_1\#\Sigma_2)}
\newcommand{\symgone}{\mbox{\rm Sym}^{g_1}(\Sigma_1)}
\newcommand{\symgtwo}{\mbox{\rm Sym}^{g_2}(\Sigma_2)}
\newcommand{\symgmo}{\mbox{\rm Sym}^{g-1}(\Sigma)}
\newcommand{\symggmo}{\mbox{\rm Sym}^{2g-1}(\Sigma)}
\newcommand{\dom}{\mathcal{D}}
\newcommand{\bigtrans}{\left.\bigcap\hspace{-0.27cm}\right|\hspace{0.1cm}}
\newcommand{\tlt}{\times\ldots\times}
\newcommand{\Isotopy}{\Gamma^\infty_{\Psi_t}}
\newcommand{\orient}{\mathnormal{o}}
\newcommand{\ob}{\mathnormal{ob}}
\newcommand{\SL}{\mbox{\rm SL}}
\newcommand{\rhotilde}{\widetilde{\rho}}
\newcommand{\domstar}{\dom_*}
\newcommand{\domststar}{\dom_{**}}
\newcommand{\betaprime}{\beta'}
\newcommand{\betapp}{\beta''}
\newcommand{\betatilde}{\widetilde{\beta}}
\newcommand{\deltaprime}{\delta'}
\newcommand{\tbetaprime}{\mathbb{T}_{\beta'}}
\newcommand{\talphaprime}{\mathbb{T}_{\alpha'}}
\newcommand{\tdelta}{\mathbb{T}_{\delta}}
\newcommand{\phidelta}{\phi^{\Delta}}
\newcommand{\domtilde}{\widetilde{\dom}}
\newcommand{\loss}{\widehat{\mathcal{L}}}
\newcommand{\bargamma}{\overline{\Gamma}}
\newcommand{\alphaprime}{\alpha'}
\newcommand{\ga}{\Gamma_{\alpha;\beta',\beta''}}
\newcommand{\gbone}{\Gamma_{\alpha;\beta,\widetilde{\beta}}^{w,1}}
\newcommand{\gbtwo}{\Gamma_{\alpha;\beta,\widetilde{\beta}}^{w,2}}
\newcommand{\gbthree}{\Gamma_{\alpha;\beta,\widetilde{\beta}}^{w,3}}
\newcommand{\gbfour}{\Gamma_{\alpha;\beta,\widetilde{\beta}}^{w,4}}
\newcommand{\gcone}{\Gamma_{\alpha;\delta,\delta'}^{w,1}}
\newcommand{\gctwo}{\Gamma_{\alpha;\delta,\delta'}^{w,2}}
\newcommand{\gcthree}{\Gamma_{\alpha;\delta,\delta'}^{w,3}}
\newcommand{\gcfour}{\Gamma_{\alpha;\delta,\delta'}^{w,4}}

\newcommand{\eab}{\epsilon_{\alpha\beta}}
\newcommand{\ead}{\epsilon_{\alpha\delta}}
\newcommand{\hqhat}{\widehat{\mbox{\rm HQ}}}
\newcommand{\cupb}{\cup_\partial}
\newcommand{\oa}{\overline{a}}
\newcommand{\ab}{\alpha\beta}
\newcommand{\ad}{\alpha\delta}
\newcommand{\adb}{\alpha\delta\beta}
\newcommand{\tila}{\widetilde{a}}
\newcommand{\tilb}{\widetilde{b}}
\def\pdehn#1#2{D_{#1}^{+,#2}} 
\def\ndehn#1#2{D_{#1}^{-,#2}}

%
%
\newcommand{\bund}{\mathcal{P}}
\newcommand{\diag}{\Delta^{\!\!E}}
\newcommand{\inter}{m_{\diag}}
\newcommand{\ozs}{Ozsv\'{a}th}
\newcommand{\sza}{Szab\'{o}}